\documentclass[a4paper,11pt,reqno]{amsart}
\usepackage{graphicx,amsmath,amsthm,amssymb,mathtools,multirow}
\usepackage[a4paper, left=1in, right=1in, top=1.1in, bottom=1.1in]{geometry}
\title{Anosov representations of amalgams}
\date{April 30, 2025}

\author{Subhadip Dey}
\address{Max Planck Institute for Mathematics in the Sciences in Leipzig,
    Inselstraße 22,
    04103 Leipzig,
    Germany
}
\email{subhadip.dey@mis.mpg.de}

\author{Konstantinos Tsouvalas}
\address{Max Planck Institute for Mathematics in the Sciences in Leipzig,
    Inselstraße 22,
    04103 Leipzig,
    Germany
}
\email{konstantinos.tsouvalas@mis.mpg.de}


\usepackage[dvipsnames]{xcolor}
\usepackage{enumitem,overpic,caption,amssymb}

\makeatletter
\def\section{\@startsection{section}{1}
  \z@{.7\linespacing\@plus\linespacing}{.5\linespacing}
  {\normalfont\bfseries\centering}}
\renewcommand\@seccntformat[1]{
  \ifcsname the#1\endcsname
    \textbf{\csname the#1\endcsname.}\hspace{1pt}
  \fi
}
\makeatother

\setenumerate{label=(\arabic*),itemsep=\smallskipamount,topsep=3pt}
\usepackage[linktocpage]{hyperref}
\hypersetup{
    colorlinks=true,
    linkcolor=cyan,
    hypertexnames=false,
    filecolor=black,      
    urlcolor=blue,
    citecolor=magenta
}

\numberwithin{equation}{section} 

\usepackage{cleveref}
\newtheorem{theorem}{Theorem}[section]
\newtheorem{proposition}[theorem]{Proposition}
\newtheorem{question}[theorem]{Question}
\newtheorem{lemma}[theorem]{Lemma}
\newtheorem{corollary}[theorem]{Corollary}

\newtheorem{nclaim}[theorem]{Claim}
\newtheorem{maintheorem}{Theorem}

\newtheorem{maincor}[maintheorem]{Corollary}
\theoremstyle{remark}
\newtheorem{remark}[theorem]{Remark}

\newtheorem{example}[theorem]{Example}
\theoremstyle{definition}
\newtheorem{definition}[theorem]{Definition}

\newtheorem{fact}[theorem]{Fact}

\def\acts{\curvearrowright}

\def\H{\mathbb{H}}
\def\N{\mathbb{N}}
\def\R{\mathbb{R}}

\def\<{\langle}
\def\>{\rangle}
\def\acts{\curvearrowright}

\def\M{\textup{M}}

\def\F{\mathcal{F}}
\def\G{\Gamma}

\def\g{\gamma}

\def\H{\mathrm H}

\def\o{\circ}

\DeclareMathOperator{\diam}{diam}

\def\geo{\partial_{\infty}}

\makeatletter
\renewcommand*\env@matrix[1][\arraystretch]{
\edef\arraystretch{#1}
\hskip -\arraycolsep
\let\@ifnextchar\new@ifnextchar
\array{*\c@MaxMatrixCols c}}
\makeatother

\begin{document}

\begin{abstract}
For uniform lattices $\Gamma$ in rank 1 Lie groups, we construct Anosov representations of virtual doubles of $\Gamma$ along certain quasiconvex subgroups. We also show that virtual HNN extensions of these lattices over some cyclic subgroups admit Anosov embeddings. In addition, we prove that for any Anosov subgroup $\Gamma$ of a real semisimple linear Lie group $\mathsf{G}$ and any infinite abelian subgroup $\mathrm{H} $ of $ \Gamma$, there exists a finite-index subgroup $\Gamma' $ of $ \Gamma$ containing $\mathrm{H}$ such that the double $\Gamma' *_{\mathrm{H}} \Gamma'$ admits an Anosov representation, thereby confirming a conjecture of \cite{TholozanT}. These results yield numerous examples of one-ended hyperbolic groups that do not admit discrete and faithful representations into rank 1 Lie groups but do admit Anosov embeddings into higher-rank Lie groups.
\end{abstract}

\maketitle

\section{Introduction}

Anosov subgroups, introduced by Labourie \cite{labourie} and further developed by Guichard–Wienhard \cite{Guichard-Wienhard}, extend the notion of convex cocompactness from classical Kleinian groups to higher-rank settings. They form a rich family of discrete subgroups of semisimple Lie groups, playing a central role in higher Teichmüller theory and the study of geometric structures.

A wide variety of hyperbolic groups are known to admit Anosov representations into real semisimple linear Lie groups $\mathsf{G}$, including convex cocompact subgroups of rank 1 Lie groups (e.g., $\mathsf{SO}(n,1)$ for $n \geq 2$), fundamental groups arising from certain Gromov–Thurston constructions \cite{kapovich2007convex} (see also \cite{Monclair-Schlenker-Tholozan}), hyperbolic right-angled Coxeter groups \cite{danciger2018convex,Lee-Marquis,DGKLM-cc}, and cubulated hyperbolic groups \cite{douba2023cubulated}, including those from the strict hyperbolization method of Charney–Davis \cite{Lafont-Ruffoni}. This breadth of examples naturally prompts the following question:

\begin{question}\label{ques:lin_hyp}
    Does there exist a linear hyperbolic group which does not admit an Anosov (or even discrete and faithful) representation into any real semisimple Lie group?
\end{question}

The linearity assumption in \Cref{ques:lin_hyp} is necessary, as there exist hyperbolic groups that do not admit faithful representations into any linear Lie group. The first such examples were constructed by Kapovich \cite{kapovich_polygon}, who, using the superrigidity of uniform lattices $\Gamma$ in $\mathsf{G} = \mathsf{Sp}(n,1)$, showed that no infinite hyperbolic quotient of $\Gamma$ embeds faithfully into linear Lie groups over any field. Later, the second named author with Canary and Stover in \cite{CST} and with Tholozan in \cite{TholozanT, Tholozan-T-nonlinear} produced further examples by amalgamating uniform lattices in $\mathsf{G}$ along certain free subgroups, again relying on superrigidity. These groups admit no linear embeddings whatsoever, and therefore no Anosov embeddings either. 

However, establishing the existence (or nonexistence) of a linear hyperbolic group without Anosov embeddings, as posed in \Cref{ques:lin_hyp}, remains a difficult problem. A promising strategy is to revisit amalgams of two torsion-free uniform lattices in $\mathsf{G}= \mathsf{Sp}(n,1)$ along maximal cyclic subgroups that {\em are} $\mathsf{G}$-conjugate. The simplest such construction arises by {\em doubling} uniform lattices $\Gamma$ along maximal cyclic subgroups $\langle \gamma \rangle$. The resulting groups $\Gamma *_{\langle \gamma \rangle} \Gamma$ are hyperbolic and, moreover, linear \cite{TholozanT}. Although it is unclear whether any such doubles would yield examples for \Cref{ques:lin_hyp}, we show that, after passing to suitable finite-index subgroups of $\Gamma$, these doubles indeed admit Anosov representations; see Theorems \ref{main-rank1-amalgam} and \ref{cyclic-amalgam} below.

{\renewcommand{\arraystretch}{1.1}
\begin{table}
    \centering
    \begin{tabular}{|c|r@{\ $\times$\ }l|}\hline
        $\mathsf{G}$\, ($2\le n$) & \multicolumn{2}{c|}{$\mathsf{L}$\, ($1\le k\le n$)} \\ \hline\hline
        $\mathsf{O}(n,1)$ & $\mathsf{O}(n-k)$ & $\mathsf{O}(k,1)$ \\ \hline
        \multirow{2}{*}{$\mathsf{U}(n,1)$} & $\mathsf{U}(n-k)$ & $\mathsf{O}(k,1)$ \\ 
        & $\mathsf{U}(n-k)$ & $\mathsf{U}(k,1)$ \\ \hline
        \multirow{3}{*}{\hspace{1em}$\mathsf{Sp}(n,1)$\hspace{1em}} & $\mathsf{Sp}(n-k)$ & $\mathsf{O}(k,1)$ \\ 
        & $\mathsf{Sp}(n-k)$ & $\mathsf{U}(k,1)$ \\ 
        & \hspace{1em}$\mathsf{Sp}(n-k)$ & $\mathsf{Sp}(k,1)$\hspace{1em} \\ \hline
    \end{tabular}
    \caption{}
    \label{tab:compatible}
\end{table}
}
     
\begin{samepage}
\begin{maintheorem}\label{main-rank1-amalgam}
     Let $\mathsf G$ and $\mathsf L$ be any pair of Lie groups listed in the same row of \Cref{tab:compatible}.
     Let $\Gamma$ be a convex cocompact group (for instance, a uniform lattice) in $\mathsf{G}$, and let $\H = \Gamma\cap \mathsf{L}$. Suppose further that $\H$ is a lattice in $\mathsf{L}$.
     Then there exists a finite-index subgroup $\Gamma'$ of $\Gamma$, containing $\H$, such that the amalgam $\Gamma' *_\H \Gamma'$ admits a faithful $1$-Anosov representation into $\mathsf{SL}_d(\mathbb{C})$ for some $d\in\mathbb{N}$ depending only on $\mathsf{G}$.
\end{maintheorem}
\end{samepage}

Furthermore, the following result shows that HNN extensions of many uniform lattices in rank 1 Lie groups over some cyclic subgroups also admit Anosov representations:

\begin{maintheorem}\label{main:cyclic-HNN}
    Let $\mathsf{G} = \mathsf{O}(n,1)$, $\mathsf{U}(n,1)$, $\mathsf{Sp}(n,1)$ with $n \geq 2$, or $\mathsf{F}_4^{-20}$, and let $\Gamma$ be a non-elementary convex cocompact subgroup of $\mathsf{G}$. Then there exists a finite-index normal subgroup $\Gamma'$ of $\Gamma$ and a pair of infinite cyclic subgroups $\langle \alpha \rangle$ and $\langle \beta \rangle$ of $\Gamma'$ such that the HNN extension $\Gamma' *_{t\alpha t^{-1} = \beta}$ admits a faithful $1$-Anosov representation into $\mathsf{SL}_d(\mathbb{C})$, where $d$ depends only on $n$.
\end{maintheorem}

In general, it is interesting to know if the doubles of general Anosov subgroups $\Gamma$ of a real linear semisimple Lie group $\mathsf{G}$ along certain quasiconvex subgroups $\H$ also admit Anosov embeddings.\footnote{A necessary condition for the linearity of $\Gamma *_{\H} \Gamma$ is the separability of $\H$ in $\Gamma$ \cite{Long-Niblo}. Whether all quasiconvex subgroups are separable in general hyperbolic groups is a major open problem.} Tholozan–Tsouvalas (see \cite[Conj. 10]{TholozanT}) conjectured this to be the case when $\H$ is a maximal cyclic subgroup. Recently, Danciger–Guéritaud–Kassel \cite[Cor. 1.24]{DGK-comb}  proved this conjecture for $1$-Anosov subgroups $\Gamma$ of $\mathsf{G} = \mathsf{SL}_d(\mathbb{R})$, $d \geq 3$, preserving a properly convex domain in $\mathbb{P}(\mathbb{R}^d)$, where $\H = \langle \gamma \rangle$ and $\gamma$ is similar to $\textup{diag}(\lambda, \textup{I}_{d-2}, \lambda^{-1})$ with $|\lambda| > 1$.

We confirm this conjecture in full generality; in fact, we prove the following more general statement:

\begin{maintheorem}\label{cyclic-amalgam}
    Let $\mathsf{G}$ be a linear, noncompact, real semisimple Lie group with finitely many connected components, and let $\Theta$ be a nonempty set of simple restricted roots of $\mathsf{G}$. Let $\Gamma$ be a $\Theta$-Anosov subgroup of $\mathsf{G}$, and let $\H $ be an infinite\footnote{This result is also valid when $\H$ is finite, though possibly without a uniform control on the dimension $d$. See \Cref{rem:1.2} for further details.}
    abelian subgroup of $\Gamma$. Then there exists a finite-index subgroup $\Gamma'$ of $\Gamma$, containing $\H$, such that the amalgam $\Gamma' *_{\H} \Gamma'$ admits a faithful $1$-Anosov representation into $\mathsf{SL}_d(\mathbb{C})$ for some $d$ depending only on $\mathsf{G}$ and $\Theta$.
\end{maintheorem}

Finally, we state an application of the results above.
Let $\mathsf{G}$ be either $\mathsf{Sp}(n,1)$ or $\mathsf{F}_4^{-20}$, and let $\Gamma$ be a uniform lattice in $\mathsf{G}$ with $\H$ an infinite cyclic subgroup of $\Gamma$. By \cite[Thm. 1.7]{TholozanT}, if $\mathsf{G} = \mathsf{Sp}(n,1)$ with $n \geq 4$, then $\Gamma *_{\H} \Gamma$ admits no discrete and faithful representations into any rank 1 Lie group, and the same holds for $\mathsf{F}_4^{-20}$ by Corlette's superrigidity theorem \cite{Corlette} (see also \Cref{rem:complex-quat-amalgam}). Nevertheless, \Cref{main-rank1-amalgam} (or \Cref{cyclic-amalgam}) produces numerous such doubles that admit Anosov representations. This leads to the following:

\begin{maincor}\label{cor:one-ended}
    There exist one-ended hyperbolic groups admitting Anosov representations into some semisimple Lie group, but no (even virtual) discrete and faithful representations into any rank 1 Lie group.
\end{maincor}

We conclude this introduction with the following remarks:
\begin{enumerate}[label=(\alph*)]
    \item A special case of \Cref{main-rank1-amalgam} can be described as follows: Let $M$ be a closed quaternionic hyperbolic $n$-manifold with $n\ge 2$, and suppose that $N$ is a closed, connected, embedded, totally geodesic complex hyperbolic submanifold. Then there exists a finite connected cover $M'$ of $M$ with a lift of $N$ in $M'$ such that $\Gamma = \pi_1(M' \sqcup_N M')$ admits a discrete and faithful (in fact, Anosov) representation into $\mathsf{SL}_d(\mathbb{C})$, where one can take $d = 2025\, n^8$. However, $\Gamma$ admits no discrete and faithful representation into any rank 1 Lie group; see \Cref{indiscrete-rank1}.
    
    \item To our knowledge, the examples provided by Theorem \ref{main:cyclic-HNN} are the first known instances of Anosov (and even linear) one-ended HNN extensions of lattices in $\mathsf{U}(n,1)$, $\mathsf{Sp}(n,1)$, and $\mathsf{F}_4^{-20}$. These HNN extensions are not cubulable (by \cite{Niblo-Reeves,Py}), so their Anosov representations do not follow from the existing results mentioned at the beginning of this introduction.
    
    \item Concerning \Cref{cor:one-ended}, the first examples of such hyperbolic groups, though with infinitely many ends, were exhibited by Douba–Tsouvalas \cite{DT24}, constructed as free products of cocompact lattices in $\mathsf{F}_4^{-20}$. We also note that \Cref{cor:one-ended} follows from \Cref{main:cyclic-HNN}, since the resulting HNN extensions are one-ended and do not admit any discrete and faithful representations into rank one Lie groups, even virtually.
\end{enumerate}

\subsection*{Outline} 
For the convenience of the reader, we give a brief outline of the proofs of the main results. The proofs of \Cref{main-rank1-amalgam} and \Cref{cyclic-amalgam} rely on a Virtual Amalgam Theorem (\Cref{thm:amalgam}), established in \Cref{sec:2}, which asserts that sufficiently antipodal limit sets of two Anosov subgroups intersecting in a quasiconvex subgroup enables virtual amalgamations. The proof of this theorem is completed by constructing an interactive pair in the corresponding flag manifold and then applying the Combination Theorems for Anosov subgroups from \cite{DeyK:amalgam}. Using \Cref{thm:amalgam}, together with the recent separability result of infinite abelian subgroups of $\Gamma$ from \cite{Tsouvalas-cyclic-separability}, we prove \Cref{cyclic-amalgam} in \Cref{sec:3} by constructing two $1$-Anosov embeddings of $\Gamma$ into $\mathsf{SL}_d(\mathbb{C})$ satisfying the required antipodality condition. Here, the dimension $d$ can be taken to be the minimal dimension admitting a $\Theta$-proximal embedding of $\mathsf{G}$ into $\mathsf{SL}_d(\mathbb{R})$; cf. \Cref{theta-compatible}.

In \Cref{sec:4}, we establish certain existence results concerning Anosov representations of doubles of convex cocompact rank one groups along non-elementary subgroups, under a technical compatibility assumption (\Cref{def:compatible}). See \Cref{amalgam-rank1} for a precise statement. This result is crucial for the proof of \Cref{main-rank1-amalgam}, where amalgamation along non-elementary subgroups is allowed. Using this result, we verify in \Cref{sec:5} the compatibility condition, thereby completing proof the \Cref{main-rank1-amalgam}; cf. \Cref{complex-quat-amalgam}.

The purpose of the final section, \Cref{sec:6}, is to prove \Cref{main:cyclic-HNN} concerning HNN extensions. To this end, we first establish a more general result (\Cref{HNN-ext}) regarding Anosov representations of HNN extensions, whose structure closely resembles that of the Virtual Amalgam Theorem. The proof strategy is modeled on that of \Cref{thm:amalgam}, relying on the construction of an interactive triple in the relevant flag manifold and an application of the Combination Theorems from \cite{DeyK:amalgam}. The construction of the interactive triple, however, is more delicate than that of interactive pairs, depending crucially on the double coset separability of given pair of subgroups inside the ambient group $\Gamma$. This separability for a pair of cyclic subgroups in rank 1 uniform lattices, originally established by Hamilton \cite{Hamilton-double-coset}, enables the completion of the proof of \Cref{main:cyclic-HNN} in the case of lattices. In fact, we also provide an extension of Hamilton’s theorem, allowing us to derive a strengthened version of \Cref{main:cyclic-HNN}, formulated as \Cref{thm:cyclic-HNN2}.

\subsection*{Acknowledgment}
We are grateful to Ian Agol, Misha Kapovich, Ashot Minasyan, and Henry Wilton for useful conversations on subgroup separability.
We thank Sami Douba, Beatrice Pozzetti, and Nicolas Tholozan for enjoyable discussions related to the results in this paper.

\section{Virtual amalgams}\label{sec:2}
Let $\Gamma_1$, $\Gamma_2$, and $\H$ be abstract groups together with monomorphisms  $\iota_{r}:\H \xhookrightarrow{} \Gamma_r$, $r\in \{1,2\}$. We view $\H$ as a common subgroup of $\Gamma_1$ and $\Gamma_2$. The {\em amalgamated free product} (or {\em amalgam}) {\em of $\Gamma_1$ and $\Gamma_2$ along $\H$} is the group $$\Gamma_1\ast_{\H}\Gamma_2=(\Gamma_1\ast \Gamma_2)/N$$ where $N\coloneqq \big\langle \big\{g(\iota_1(h)\iota_2(h^{-1}))g^{-1}:\ h\in \H,\ g\in \Gamma_1\ast \Gamma_2\big\}\big\rangle$. If $\Gamma_1=\Gamma_2$ and $\iota_1 = \iota_2$, we refer to $\Gamma_1\ast_{\H}\Gamma_1$ as the {\em double} of $\Gamma_1$ along $\H$. See Lyndon-Schupp \cite[Ch. IV]{Lyndon-Schupp} for more details.

\medskip
In this paper, we construct Anosov subgroups of $\mathsf{G}$ (as in \Cref{cyclic-amalgam}) by amalgamating two Anosov subgroups along a shared quasiconvex subgroup. We refer the reader to Labourie \cite{labourie} and Guichard–Wienhard \cite[Def. 2.10]{Guichard-Wienhard} for the notion of {\em $\Theta$-Anosov subgroups}. Briefly, if $\Gamma < \mathsf{G}$ is $\Theta$-Anosov, then $\Gamma$ is Gromov-hyperbolic and its Gromov boundary $\geo\Gamma$ admits a canonical $\Gamma$-equivariant embedding $\xi : \geo\Gamma \to \mathsf{G}/\mathsf{P}_\Theta$, called the {\em $\Theta$-Anosov limit map}. Here, $\mathsf{P}_\Theta$ is the parabolic subgroup associated to the set of simple restricted roots $\Theta$ and $\mathsf{G}/\mathsf{P}_\Theta$ is the corresponding flag variety. The image of $\xi$, called the {\em limit set} of $\Gamma$ in $\mathsf{G}/\mathsf{P}_\Theta$, is denoted by $\Lambda_\Gamma$.

A pair of points $x \in \mathsf{G}/\mathsf{P}_\Theta$ and $y \in \mathsf{G}/\mathsf{P}_\Theta^{-}$ is called \emph{antipodal} if $(x,y)$ lies in the unique open $\mathsf{G}$-orbit in $(\mathsf{G}/\mathsf{P}_\Theta) \times (\mathsf{G}/\mathsf{P}_\Theta^{-})$, where $\mathsf{P}_\Theta^{-}$ is a parabolic subgroup of $\mathsf{G}$ opposite to $\mathsf{P}_\Theta$. Similarly, two nonempty subsets $Z_1 \subset \mathsf{G}/\mathsf{P}_\Theta$ and $Z_2 \subset \mathsf{G}/\mathsf{P}_\Theta^{-}$ are said to be antipodal if each point in $Z_1$ is antipodal to each point in $Z_2$. A special feature of Anosov limit maps $\xi : \partial_\infty \Gamma \to \mathsf{G}/\mathsf{P}_\Theta$ and $\xi^{-} : \partial_\infty \Gamma \to \mathsf{G}/\mathsf{P}_\Theta^{-}$ is that for all distinct points $z_1, z_2 \in \partial_\infty \Gamma$, $\xi(z_1)$ is antipodal to $\xi^-(z_2)$.

For the remainder of this section, we assume that $\mathsf{G}$ is connected\footnote{This assumption, sufficient for our purposes, can be weakened; see Assumption 3.1 of \cite{DeyK:amalgam}.} (e.g., $\mathsf{G} = \mathsf{SL}_d(\mathbb{R})$) and that $\Theta$ is preserved under the opposition involution. In this case, $\mathsf{P}_{\Theta}^{-}$ is conjugate to $\mathsf{P}_\Theta$, so the notion of antipodality is well-defined for pairs of points in $\mathsf{G}/\mathsf{P}_\Theta$. In particular, any distinct pair of points in the limit set $\Lambda_\Gamma \subset \mathsf{G}/\mathsf{P}_\Theta$ of a $\Theta$-Anosov subgroup $\Gamma$ is antipodal.

The goal of this section is to prove the following theorem, which enables us to virtually amalgamate two $\Theta$-Anosov subgroups of $\mathsf{G}$ along a shared quasiconvex subgroup—a scenario that arises frequently in this paper:

\begin{theorem}[Virtual Amalgam Theorem]
\label{thm:amalgam}
    Let $\G_1,\G_2$ be $\Theta$-Anosov subgroups of $\mathsf{G}$ such that the intersection $\H = \G_1\cap\G_2$ is quasiconvex and separable in $\G_1$ and $\G_2$.
    Suppose that 
    \[
    \Lambda'_1 \coloneqq \Lambda_{\G_1}\smallsetminus\Lambda_\H \,\text{ is antipodal to } \,\Lambda'_2 \coloneqq \Lambda_{\G_2}\smallsetminus\Lambda_\H.
    \]Then there exist finite index subgroups $\G_1'<\G_1$ and $\G_2'<\G_2$, containing $\H$, such that the subgroup $\G \coloneqq \<\G'_1,\G'_2\>$ of $\mathsf{G}$ is $\Theta$-Anosov and naturally isomorphic to $\G_1' *_\H \G_2'$.
\end{theorem}

\begin{remark}
\Cref{thm:amalgam} generalizes the main theorem of Dey–Kapovich–Leeb \cite{DeyKL} to the case of amalgamated free products, which can also be viewed as a generalization of Baker–Cooper's virtual amalgam theorem \cite[Thm. 5.3]{baker2008combination} in the special case of convex cocompact Kleinian groups.
\end{remark}

The proof of \Cref{thm:amalgam} utilizes the following Combination Theorem:

\begin{theorem}[Dey--Kapovich {\cite[Thm. A]{DeyK:amalgam}}]\label{thm:amalgam1}
 Let $\G_1,\G_2$ be $\Theta$-Anosov subgroups of $\mathsf{G}$ such that the intersection $\H = \G_1\cap\G_2$ is quasiconvex in $\G_1$ or $\G_2$. Suppose that there exist compact sets $A,B \subset \F_\Theta = \mathsf{G}/\mathsf{P}_\Theta$ with nonempty interior such that the following conditions hold:
\begin{enumerate}
    \item  The interiors $A^\o$, $B^\o$ of $A$, $B$, respectively, are antipodal to each other.
    \item $\Lambda_{\G_1}\cap \partial A = \Lambda_\H = \Lambda_{\G_2}\cap \partial B$.
    \item For all $\g_1\in\G_1\smallsetminus \H$, $\g_1(B)\subset A^\o$ and for all $\g_2\in\G_2\smallsetminus \H$, $\g_2 (A) \subset B^\o$.
    \item $\H$ leaves $A$ and $B$ invariant.
\end{enumerate}    
    Then the natural homomorphism $\G_1 *_\H \G_2 \to \mathsf{G}$ is an injective $\Theta$-Anosov representation, where $\G_1 *_\H \G_2$ denotes the abstract free product of $\G_1$ and $\G_2$ amalgamated along $\H$.
\end{theorem}

There are two steps in the proof of \Cref{thm:amalgam}. First, we construct $\H$-invariant compact subsets $A$ and $B$ in $\mathcal{F}_\Theta$ (called an {\em interactive pair} for $\G_1$ and $\G_2$) that satisfy conditions 1 and 2 of \Cref{thm:amalgam1}. Additionally, the interiors $A^\circ$ and $B^\circ$ will contain the sets $\Lambda_1'$ and $\Lambda_2'$, respectively. In the second step, we will pass to finite-index subgroups $\Gamma_1' < \Gamma_1$ and $\Gamma_2' < \Gamma_2$ such that condition 3 of \Cref{thm:amalgam1} is met. \Cref{thm:amalgam} will then follow as a consequence of \Cref{thm:amalgam1}. 

We now proceed to discuss the proof in detail; note that we can (and will) make the assumption that $\H$ is infinite index in both $\G_1$ and $\G_2$.

\subsection{Step 1: Construction of an interactive pair}\label{sec:2.1}
Throughout the rest of this section, we will work under the hypotheses of \Cref{thm:amalgam}. Observe that $\H$ leaves $\Lambda_1' = \Lambda_{\G_1}\smallsetminus \Lambda_\H$ and $\Lambda_2' = \Lambda_{\G_2}\smallsetminus \Lambda_\H$ invariant.
Since $\H$ is assumed to be infinite index in both $\G_1$ and $\G_2$, $\Lambda_1'$ and $\Lambda_2'$ are nonempty.

\begin{lemma}\label{lem:1}
    There exist nonempty compact subsets $D_1\subset \Lambda_1'$, $D_2 \subset \Lambda_2'$ such that
    \[
     \Lambda_1'= \H D_1 \quad\text{and}\quad
     \Lambda_2' = \H D_2.
    \]
\end{lemma}

\begin{proof}
    We will only prove the existence of $ D_1 $, as the existence of $ D_2 $ follows similarly.

We will use the fact that $\G_1$ is hyperbolic and $\H$ is a quasiconvex subgroup of $\Gamma_1$. Fix a word metric on $ \Gamma_1 $ and consider a nearest point projection map $ \pi: \Gamma_1 \to \H $. Define $ D_1' \coloneqq \overline{\pi^{-1}(\{1\})} \cap \geo \Gamma_1 $, where the closure is taken in the visual compactification $\Gamma_1 \sqcup \geo \Gamma_1$ of $\Gamma_1$. Since $ D_1' $ is closed in $ \geo \Gamma_1 $, it is compact. Moreover, since $ \H $ is quasiconvex in $ \Gamma_1 $, we have that $ D_1' \cap \geo \H = \emptyset $ (see Corollary 2.7 of \cite{DeyK:amalgam}).

Now, if $ \zeta \in \geo \Gamma_1 \smallsetminus \geo \H $, consider a sequence $ \gamma_n \in \Gamma_1 $ converging to $ \zeta $. Since $ \zeta \notin \geo \H $, the set $ S = \{ \pi(\gamma_n) : n \in \mathbb{N} \} $ is bounded in $ \Gamma_1 $, and hence it is finite (see Corollary 2.7 of \cite{DeyK:amalgam}). Thus, $ \zeta \in S D_1' $. It follows that $ \geo \Gamma_1 \smallsetminus \geo \H \subset \Gamma D_1' $. Moreover, since $\geo \Gamma_1 \smallsetminus \geo \H$ is $\H$-invariant and $D_1'\subset \geo \Gamma_1 \smallsetminus \geo \H$, we also have $ \geo \Gamma_1 \smallsetminus \geo \H \supset \Gamma D_1' $. Thus, $ \geo \Gamma_1 \smallsetminus \geo \H = \Gamma D_1' $.

Let $ \xi^1: \geo \Gamma \to \Lambda_{\Gamma_1} \subset \mathcal{F}_\Theta $ denote the $ \Gamma $-equivariant limit map. Define $ D_1 = \xi^1(D_1') $. Clearly, $ D_1 $ is a compact subset of $ \Lambda_1' = \xi^1(\geo \Gamma_1 \smallsetminus \geo \H) $, and $ \Lambda_1' = \H D_1 $.
\end{proof}

We fix a distance function $d$ on $\F_\Theta$ compatible with the manifold topology. 

\begin{lemma}\label{lem:1.5}
    Let $D\subset \F_\Theta$ be a nonempty compact subset antipodal to the limit set of $\Lambda_\H$. Then 
    \begin{equation}\label{eqn0:lem:1.5}
        \overline{\H D} = \H D \sqcup \Lambda_\H.
    \end{equation}
\end{lemma}

\begin{proof}
    If $\H$ is finite, then $\Lambda_\H = \emptyset$, and the claim follows easily. So, we assume that $\H$ is infinite.
    Enumerate $\H = \{\eta_1,\eta_2,\dots\}$.
    Since $\Theta$-Anosov subgroups are $\Theta$-regular (a.k.a. $\Theta$-divergent) in the sense of \cite[Sect. 4]{MR3736790}, we have that
    \begin{equation}\label{eqn2:lem:1.5}
    \diam(\eta_i D) \to 0
    \quad\text{as } i\to\infty.
    \end{equation}
    (see, e.g., \cite[Lem. 3.9]{DeyK:amalgam}).
    
    Moreover, 
    \begin{equation}\label{eqn1:lem:1.5}
     d(\Lambda_\H,\eta_i D) \to 0 \quad\text{as }i\to\infty   
    \end{equation}
    Indeed, if this were not true, then by the preceding paragraph, after extraction, there exists $x\in\F\smallsetminus\Lambda_\H$ such that
    \begin{equation}\label{eqn:lem:1.5}
        d(\eta_i D, x) \to 0 \quad\text{ as }i\to\infty.
    \end{equation}
    However, since $\H$ is $\Theta$-regular, after further extraction, there exists $z_\pm\in \Lambda_\H$ such that $\eta_i\vert_{C(z_-)} : C(z_-) \to \F_\Theta$, where $C(z_-)$ is the set of all points in $\F_\Theta$ antipodal to $z_-$, converges uniformly on compact sets to the constant map $C(z_-) \to \{z_+\}$, see \cite[Sect. 4]{MR3736790}. Since $D\subset C(z_-)$, \eqref{eqn:lem:1.5} implies that $z_+ = x$, contradicting our assumption that $x\not\in\Lambda_\H$.

    Therefore, by \eqref{eqn2:lem:1.5} and \eqref{eqn1:lem:1.5}, we have that $\H D \sqcup \Lambda_\H$ is closed. 
    
    If $\H$ is elementary, then $\H$ contains a finite index infinite cyclic subgroup $\< \eta \>$. In this case, $\Lambda_\H = \{ \eta^+,\eta^-\}$ and, so, $\eta^{\pm n} D \to \eta^\pm$ as $n\to \infty$. This verifies \eqref{eqn0:lem:1.5}.

    If $\H$ is non-elementary, then $\H$ acts minimally on $\Lambda_\H$. Since $\overline{\H D}$ is $H$-invariant and $\overline{\H D} \cap \Lambda_\H \ne\emptyset$ (by \eqref{eqn1:lem:1.5}), by minimality of the action $\H\acts \Lambda_\H$, it follows that $\Lambda_\H \subset \overline{\H D}$, which implies \eqref{eqn0:lem:1.5}.
\end{proof}

Let $D_1\subset \Lambda_1'$ and $D_2 \subset \Lambda_2'$ be as in \Cref{lem:1}. Note that $D_1$, $D_2$, and $\Lambda_\H$ are pairwise antipodal compact subsets of $\F_\Theta$.
Since antipodality is an open condition, there exists $\epsilon_0>0$ such that $\mathcal{N}_\epsilon(D_1)$, $\mathcal{N}_\epsilon(D_2)$, and $\mathcal{N}_\epsilon(\Lambda_\H)$  are pairwise antipodal for all $\epsilon\in (0,\epsilon_0]$.
Here, for $\epsilon>0 $ and nonempty subsets $D\subset \F_\Theta$, $\mathcal{N}_\epsilon(D)$ denotes the closed $\epsilon$-neighborhood of $D$ in $\F_\Theta$.

For $\epsilon\in (0,\epsilon_0]$, let
\begin{align}
\begin{split}\label{eqn:1}
    A_\epsilon &\coloneqq \overline{\H \mathcal{N}_\epsilon(D_1)} = \H \mathcal{N}_\epsilon(D_1) \sqcup \Lambda_\H,\\
    B_\epsilon &\coloneqq \overline{\H \mathcal{N}_\epsilon(D_2)} =\H \mathcal{N}_\epsilon(D_2) \sqcup \Lambda_\H.
\end{split}
\end{align}
Note that \Cref{lem:1.5} justifies the second equality in each of the above equations.

\begin{lemma}\label{lem:2.5}
    There exists $\epsilon_1<\epsilon_0$ such that for all $\epsilon\in(0,\epsilon_1]$, we have:

\begin{enumerate}
    \item $\Lambda_\H \subset \partial A_\epsilon$ and $\Lambda_\H \subset \partial B_\epsilon$, where $\partial A_\epsilon$, $\partial B_\epsilon$ are the frontiers of $A_\epsilon$, $B_\epsilon$ in $\F_\Theta$, respectively.
    \item $A_\epsilon\smallsetminus\Lambda_\H $ is antipodal to $B_\epsilon \smallsetminus \Lambda_\H$.
    \item For all $\epsilon\in(0,\epsilon_1]$, $A^\o_\epsilon$ is antipodal to $B^\o_\epsilon$.
\end{enumerate}
\end{lemma}
\begin{proof}
    (1) Since $D_1$ is antipodal to $\Lambda_{\G_2}$, for all small enough $\epsilon\in (0,\epsilon_0]$, $\mathcal{N}_\epsilon(D_1)$ is antipodal to $\Lambda_{\G_2}$. Since $\Lambda_{\G_2}\smallsetminus \Lambda_\H$ is $\H$-invariant, it follows that $\H \mathcal{N}_\epsilon(D_1)$ is also antipodal to (and, in particular, disjoint from)  $\Lambda_{\G_2}\smallsetminus \Lambda_\H$. However, $\overline{\Lambda_{\G_2}\smallsetminus \Lambda_\H} \supset \Lambda_\H$. Thus, $\Lambda_\H \subset \partial(\overline{\H \mathcal{N}_\epsilon(D_1)}) = \partial A_\epsilon$. The other inclusion $\Lambda_\H \subset \partial B_\epsilon$ follows similarly.

    \smallskip\noindent
    (2) Observe that it is enough to prove that for all small enough $\epsilon\in (0,\epsilon_0]$ and all $\eta\in\H$, $\eta \mathcal{N}_\epsilon(D_1)$ is antipodal to $\mathcal{N}_\epsilon(D_2)$.
    Consider the family of compact sets $\{\eta \mathcal{N}_\epsilon(D_1) :\ \eta\in\H\}$. Since $\H$ is $\Theta$-Anosov and $\Lambda_\H$ is antipodal to $\mathcal{N}_\epsilon(D_1)$, it follows from \eqref{eqn2:lem:1.5} and \eqref{eqn1:lem:1.5} that all but finitely many sets in this family lie  in the $\epsilon_0$-neighborhood of $\Lambda_{\H}$ and hence they are antipodal to $\mathcal{N}_\epsilon(D_2)$ for $\epsilon \le \epsilon_0$; let $\eta_1 \mathcal{N}_\epsilon(D_1),\dots, \eta_n \mathcal{N}_\epsilon(D_1)$ be the (possibly empty set of) exceptions. However, as $\epsilon\to 0$, $\bigcup_{i=1}^n \eta_i \mathcal{N}_\epsilon(D_1)$ converges (in the Hausdorff distance) to $\bigcup_{i=1}^n \eta_i D_1$, which is a compact set antipodal to $D_2$. Since antipodality is an open condition, there exists a positive $\epsilon_1<\epsilon_0$ such that for all $\epsilon \in (0,\epsilon_0]$, the sets $\eta_1 \mathcal{N}_\epsilon(D_1),\dots, \eta_n \mathcal{N}_\epsilon(D_1)$ are antipodal to $\mathcal{N}_\epsilon(D_2)$, which proves the claim.

    \smallskip\noindent
    (3) By part (1), $A_\epsilon^\o \subset A_\epsilon\smallsetminus\Lambda_\H$ and $B_\epsilon^\o \subset B_\epsilon\smallsetminus\Lambda_\H$. Thus, part (3) follows from part (1).
\end{proof}

We will fix $\epsilon \in(0,\epsilon_1]$ and define
\begin{equation}\label{eqn:3}
  A \coloneqq A_\epsilon 
 \quad\text{and }
 B \coloneqq B_\epsilon.   
\end{equation}
By construction, we have the following:

\begin{lemma}\label{lem:final1}
    Let $\G_1' < \G_1$ and $\G_2' < \G_2$ be any finite index subgroups, each containing $\H$.
    Then $\G_1'$, $\G_2'$, $\H = \G_1'\cap \G_2'$, $A$, and $B$ satisfy conditions (1), (2), and (4) in the hypothesis of \Cref{thm:amalgam1}.
    Moreover, $\Lambda_1' = \Lambda_{\G_1'}\smallsetminus \Lambda_\H$ is antipodal to $B$ and $\Lambda_2' = \Lambda_{\G_2'}\smallsetminus \Lambda_\H$ is antipodal to $A$.
\end{lemma}

\begin{proof}
    Note that since $\G_i'<\G_i$, $i=1,2$, are assumed to be finite index subgroups, $\Lambda_{\G_i} = \Lambda_{\G_i'}$. Thus, (1) follows from \Cref{lem:2.5} (3) and (4) follows from \eqref{eqn:1}.
    For (2), note that $\Lambda_1' = \Lambda_{\G_1'}\smallsetminus\Lambda_\H \subset A^\o$, while by \Cref{lem:2.5} (1), $\Lambda_\H \subset \partial A$. Thus, $\Lambda_{\G_1'} \cap \partial A = \Lambda_\H$.
    The other conclusion, $\Lambda_{\G_2'} \cap \partial B = \Lambda_\H$, follows similarly.

    Since $\Lambda_1'\subset A^\o$ and $\Lambda_2'\subset B^\o$, the ``moreover'' part follows from \Cref{lem:2.5} (1) and the fact that $\Lambda_\H$ is antipodal to both $\Lambda_1'$ and $\Lambda_2'$.
\end{proof}

Thus, to finish the proof of \Cref{thm:amalgam}, what remains is to show that there exists finite index subgroups $\G_1' < \G_1$ and $\G_2' < \G_2$, both containing $\H$, such that 
\begin{equation}
    \g_1(B) \subset A^\o \text{ and }  \g_2 (A) \subset B^\o
    \quad
    \text{for all } \g_1\in\G_1'\smallsetminus\H, \, \g_2\in\G_2'\smallsetminus\H
\end{equation}
(cf. condition (3) in the hypothesis of \Cref{thm:amalgam1}).
We demonstrate the existence of such $\G_1'$ and $\G_2'$ in the next subsection.

\subsection{Step 2: Passing to finite index subgroups}

Recall that we are working under the hypothesis of \Cref{thm:amalgam}. Fix left invariant word metrics $d_1$, $d_2$ on $\G_1$, $\G_2$, respectively. 
For $i=1,2$ and $\g_i\in \G_i$, let $\| \g_i\|_{i}$ denote the minimal word length of elements in the left coset $\g_i \H \subset \G_i$.
For example, for $\g_1\in\G_1$, $\|\g_1\|_{1} = 0$ if and only if $\g_1\in\H$.

\begin{lemma}\label{lem:2}
    For $N\in\N$ and $i=1,2$, there exists a finite index subgroup $\G_{i,N} < \G_i$ containing $\H$ such that $\|\g_i\|_{i}\ge N$ for all $\g_i\in \G_{i,N}\smallsetminus\H$.
\end{lemma}

\begin{proof}
    Note that there are only finitely many nontrivial left cosets $\g_{i,1}\H$,\dots,$\g_{i,n_i}\H$ of $\H$ in $\G_i$ such that for all $\g_i\in\G_1\smallsetminus\H$, $0<\|\g_i\|_{i} < N$ if and only if $\g_i\in\bigcup_{j=1}^{n_i} \g_{i,j}\H$. Since $\H$ is separable in $\G_i$, there exists a finite index subgroup $\G'_{i,N} < \G_1$ containing $\H$ such that $\g_{i,j} \not\in\G_{i,N}$ for $j=1,\dots,n_i$. Thus, $\|\g_i\|_{i} \ge N$ for all $\g_i\in \G_i\smallsetminus\H$. 
\end{proof}

Let $A,B \subset \F_\Theta$ be the compact subsets defined in \eqref{eqn:3}.

\begin{lemma}\label{lem:3}
    For all large enough $N\in\N$, $\g_1 B \subset A^\o$ for all $\g_1\in \G_{1,N}\smallsetminus\H$, and $\g_2 A \subset B^\o$ for all $\g_2\in \G_{2,N}\smallsetminus\H$.
\end{lemma}

\begin{proof}
    We will only prove the first assertion, since the proof of the second is similar.
    
    Suppose that the first assertion is false. Then there exist an increasing sequence of natural numbers $(N_j)$ and, for each $j\in\N$, some element $\g_{1,j} \in \G_{1,N_j}\smallsetminus \H$ such that \begin{equation}\label{eqn:lem:3}\gamma_{1,j} B \not\subset A^\circ\quad \text{for all } j\in\mathbb{N}.\end{equation}

    For each $j\in\N$, let $\eta_j\in \H$ be a nearest point projection of $\g_{1,j}^{-1}\in\G_1$. Define $\tilde\g_{1,j} = \g_{1,j} \eta_j$. Note that 
\[
d_1(\tilde\g_{1,j}^{-1},\H) = d_1(1, \g_{1,j} \H) = \| \g_{1,j}\|_{1} \ge N_j,
\]
where the last inequality follows since $\g_{1,j}\in \G_{1,N_j}\smallsetminus\H$.
Thus, viewed as a sequence in the compactification $\G_1\sqcup \geo\G_1$, the inverse sequence $(\tilde\g_{1,j}^{-1})$ has all the accumulation points in $\geo\G_1\smallsetminus\geo\H$ (see, e.g., \cite[Lemma 2.9]{DeyK:amalgam}). After extraction, we will assume that $(\tilde\g_{1,j}^{-1})$ converges to some point $x_-$ in $\geo\G_1\smallsetminus\geo\H$.

For each $j\in\N$, let $\tilde\eta_j\in\H$ be a nearest point projection of $\tilde\g_{1,j} \in \G_1$. Let $\hat\g_{1,j} \coloneqq \tilde\eta_j^{-1}\tilde\g_{1,j}$.
Then
\[
 d_1(\hat\g_{1,j},\H) = d_1(1,\tilde\g_{1,j}^{-1}\H) = \| \tilde\g_{1,j}^{-1}\|_{1} \ge N_j,
\]
since $\tilde\g_{1,j}^{-1}\in \G_{1,N_j}\smallsetminus\H$.
Thus, we see that $(\hat\g_{1,j}^{-1})$ also converges to $x_-\in \Lambda_1'$.
After another extraction, we can (and will) assume that $(\hat\g_{1,j})$ converges to some point $x_+\in \geo\G_1\smallsetminus\geo\H$.
(For the last two sentences, see, e.g., \cite[Lem. 2.8]{DeyK:amalgam}.)

Let $z_\pm \coloneqq \xi^1(x_\pm)$, where $ \xi^1: \geo \Gamma \to \Lambda_{\Gamma_1} \subset \mathcal{F}_\Theta $ is the $ \Gamma $-equivariant limit map.
Since $x_\pm \in \geo\G_1\smallsetminus\geo\H$, $z_\pm \in \Lambda_1' = \Lambda_{\G_1}\smallsetminus\Lambda_\H$ and, therefore, $z_\pm$ are antipodal to $B$ (see \Cref{lem:final1}).

Since $B$ is antipodal to $z_-$, it follows that  the sequence $(\hat\g_{1,j} B)$ converge to $\{z_+\}$ (in the Hausdorff distance) as $j\to\infty$. But $z_+$ lies in the interior of $A$. So, for all large enough $j$, $\hat\g_{1,j} B \subset A^\o$.
Since $\g_{1,j} = \tilde\eta_j\hat\g_{1,j} \eta_j^{-1} $ and $A^\o$ and $B$ are $\H$-invariant, it follows that $\g_{1,j} B \subset A^\o$ for all large enough $j\in \N$, which is a contradiction to \eqref{eqn:lem:3}.
\end{proof}

\subsection{Proof of \Cref{thm:amalgam}}
Now we can conclude the proof of \Cref{thm:amalgam}. Let $A,B \subset \F_\Theta$ be the compact subsets defined in \eqref{eqn:3}.
Pick $N>1$ large enough so that the $\G'_1\coloneqq \G_{1,N} < \G_1$  and $\G'_2\coloneqq \G_{2,N} < \G_2$ satisfy the conclusion of \Cref{lem:3}. By \Cref{lem:final1,lem:3}, $\G_1'$, $\G_2'$, $\H = \G_1'\cap \G_2'$, $A$, and $B$ satisfy the hypothesis of \Cref{thm:amalgam1}.
Therefore, by \Cref{thm:amalgam1}, the subgroup $\G \coloneqq \<\G'_1,\G'_2\>$ is $\Theta$-Anosov and naturally isomorphic to $\G_1' *_\H \G_2'$.
\qed 

\section{Amalgams of Anosov groups along abelian subgroups}\label{sec:3}

In this section, we prove \Cref{cyclic-amalgam}.

\subsection{Faithful linear representations} Let $\mathbb{K}=\mathbb{R}$ or $\mathbb{C}$. 
By a $k$-Anosov subgroup of $\mathsf{G} = \mathsf{SL}_q(\mathbb{K})$, where $1\le k\le q-1$, we mean a $\Theta_k$-Anosov subgroup where $\Theta_k$ corresponds to the partial flag manifold $\mathsf{Gr}_k(\mathbb{K}^q)$ of $\mathsf{SL}_q(\mathbb{K})$.

Throughout this paper, we shall use the following proposition, which will allow us to treat $\Theta$-Anosov representations into $\mathsf{G}$ as $1$-Anosov representations into $\mathsf{SL}_q(\mathbb{R})$, $q\in \mathbb{N}$.

\begin{proposition} \textup{(see \cite[Prop. 4.10]{Guichard-Wienhard} and \cite[Prop. 3.5]{GGKW})}\label{theta-compatible} Let $\mathsf{G}$ be a linear real semisimple Lie group,\footnote{Recall that we always assume $\mathsf{G}$ has finitely many connected components.} and let $\Theta$ be a set of restricted roots of $\mathsf{G}$.
Then there exists a faithful Lie group homomorphism $\tau_{\Theta}:\mathsf{G}\rightarrow \mathsf{SL}_q(\mathbb{R})$, $q=q(\mathsf{G},\Theta)$, with the following property: a subgroup $\Gamma<\mathsf{G}$ is $\Theta$-Anosov if and only if $\tau_{\Theta}(\Gamma)<\mathsf{SL}_q(\mathbb{R})$ is $1$-Anosov.\end{proposition}

\begin{proof} Let $\mathsf{P}_{\Theta},\mathsf{P}_{\Theta}^{-}<\mathsf{G}$ be the pair of opposite parabolic subgroups associated to $\Theta$. Following \cite[Rem. 4.12]{Guichard-Wienhard}, we explain how we can choose $\tau_{\Theta}$ to be faithful in the case where $\mathsf{G}$ has non-trivial center. 
For this, we fix an embedding $\mathsf{G}\xhookrightarrow{} \mathsf{SL}_{r}(\mathbb{R})$, $r\in \mathbb{N}$ odd, and compose the adjoint representation of $\textup{Ad}:\mathsf{SL}_{r}(\mathbb{R})\rightarrow \mathsf{SL}(\mathfrak{sl}_r(\mathbb{R}))$ with the exterior power $\bigwedge^p:\mathsf{SL}(\mathfrak{sl}_r(\mathbb{R}))\rightarrow \mathsf{SL}\big(\bigwedge^p\mathfrak{sl}_r(\mathbb{R})\big)$, where $p=\operatorname{dim}{\mathsf{P}}_{\Theta}$. Note that $\bigwedge^p\textup{Ad}(\mathsf{G})$ preserves the subspace $\textup{Lie}(\mathsf{G})\subset \mathfrak{sl}_n(\mathbb{R})$.
As $\mathsf{G}$ is semisimple, we may decompose $\bigwedge^p\textup{Ad}|_{\mathsf{G}}$ into a direct sum of irreducible representations so that we obtain a faithful representation $\tau_{0}:\mathsf{G}\rightarrow \mathsf{GL}_m(\mathbb{R})\times \mathsf{GL}_n(\mathbb{R})$ with the property that the projection $\phi_0:\mathsf{G}\rightarrow \mathsf{GL}_m(\mathbb{R})$ of $\tau_0$ into the first factor is the irreducible representation from \cite[Rem. 4.12]{Guichard-Wienhard} and: $\Gamma<\mathsf{G}$ is $\Theta$-Anosov if and only if $\phi_0(\Gamma)<\mathsf{GL}_m(\mathbb{R})$ is $1$-Anosov. Since $\mathsf{G}$ is semisimple, $\phi_0(\mathsf{G}^0)<\mathsf{SL}_m(\mathbb{R})$.
Thus, by composing $\phi_0$ with the embedding $\mathsf{GL}_m(\mathbb{R})\xhookrightarrow{} \mathsf{SL}_{m+1}(\mathbb{R})$, $g\mapsto \textup{diag}(g,\textup{det}(g)^{-1})$, and possibly increasing $n,m$ by at most $2$, we may assume that $\phi_0(\mathsf{G})<\mathsf{SL}_m(\mathbb{R})$ and $n,m$ are odd. 
In this case, $\mathsf{SL}_m(\mathbb{R})\times \mathsf{SL}_n(\mathbb{R})$ has trivial center. Therefore, by \cite[Prop. 4.10]{Guichard-Wienhard} and \cite[Prop. 3.5]{GGKW}, there exists an irreducible faithful representation \hbox{$\tau':\mathsf{SL}_m(\mathbb{R})\times \mathsf{SL}_n(\mathbb{R})\rightarrow \mathsf{SL}_{s}(\mathbb{R})$} with the property that a subgroup $\Gamma<\mathsf{SL}_m(\mathbb{R})\times \mathsf{SL}_n(\mathbb{R})$ projects into $\mathsf{SL}_m(\mathbb{R})$ as a $1$-Anosov subgroup if and only if $\tau'(\Gamma)<\mathsf{SL}_s(\mathbb{R})$ is $1$-Anosov. Then the composition $\tau_{\Theta} \coloneqq \tau'\circ \tau_0:\mathsf{G}\rightarrow \mathsf{SL}_s(\mathbb{R})$  is a faithful representation satisfying the conclusion above.
\end{proof}

\subsection{Positioning the limit sets}

 Given an Anosov subgroup $\G$ of $\mathsf{G}$, the following lemma will allow us to exhibit two embeddings of $\G$ into $\mathsf{SL}_d(\mathbb{C})$ as Anosov subgroups, where $d$ is large enough, such that the resulting limit sets are positioned appropriately in order to apply the Virtual Amalgam Theorem (\Cref{thm:amalgam}).

We begin by introducing some notation. Fix the standard inner product $\langle \cdot,\cdot \rangle$ on $\mathbb{K}^d$, where $\mathbb{K}=\mathbb{R}$ or $\mathbb{C}$. For a subspace $V \subset \mathbb{K}^d$, $V^{\perp}$ denotes the orthogonal complement of $V$, and let $(e_1,\ldots,e_d)$ denote the standard basis of $\mathbb{K}^d$.

\begin{lemma}\label{antipodal-cyclic} Let $\Gamma<\mathsf{SL}_d(\mathbb{R})$ be a $1$-Anosov subgroup. Let $(\xi^1,\xi^{d-1}):\partial_{\infty}\Gamma\rightarrow \mathbb{P}(\mathbb{C}^d)\times \mathsf{Gr}_{d-1}(\mathbb{C}^d)$ be the Anosov limit maps of $\Gamma$, where we regard $\Gamma$ as a $1$-Anosov subgroup of $\mathsf{SL}_d(\mathbb{C})$. 
Let $\langle \gamma \rangle<\Gamma$ be an infinite cyclic subgroup and let $\H=\textup{Stab}_{\Gamma}(\gamma^{+})$ be the stabilizer of $\gamma^{+}\in \partial_{\infty}\Gamma$ in $\Gamma$. Then there exists $g\in \mathsf{SL}_d(\mathbb{C})$, centralizing $\H$, such that $\Gamma \cap g\Gamma g^{-1}=\H$ and $g^{\pm 1}\xi^1\left(\partial_{\infty}\Gamma \smallsetminus \{\gamma^{+},\gamma^{-}\}\right)$ is antipodal to $\xi^{d-1}(\partial_{\infty}\Gamma \smallsetminus \{\gamma^{+},\gamma^{-} \})$.\end{lemma}

\begin{proof} Up to conjugating by an element of $\mathsf{GL}_d(\mathbb{R})$, we may assume that  
$$\gamma=\begin{pmatrix}[0.8] M_{\gamma} & &\\
 & \lambda_1 & \\ & & \lambda_d \end{pmatrix}, \quad 
 \text{for some }M_{\gamma}\in \mathsf{GL}_{d-2}(\mathbb{R}) \text{ and } \lambda_1,\lambda_d\in\R,$$ 
 where $|\lambda_1|$ (resp. $|\lambda_d|$) is the largest (resp. smallest) modulus eigenvalue of $\gamma$.
 Therefore,
 \begin{equation}\label{eqn:antipodal-cyclic}
     \begin{array}{lll}
        \xi^1(\gamma^{+})=[e_{d-1}],  && \xi^{d-1}(\gamma^{+})=e_{d}^{\perp}, \\
        \xi^1(\gamma^{-})=[e_d], && \xi^{d-1}(\gamma^{-})= e_{d-1}^{\perp}.
     \end{array}
 \end{equation}
 Note that every element of $\H$ also fixes the repelling fixed point $\gamma^{-}\in \partial_{\infty}\Gamma$ and hence the points $\xi^1(\gamma^{\pm})$ and the hyperplanes $\xi^{d-1}(\gamma^{\pm})$ in $\mathbb{P}(\mathbb{C}^d)$.
 Thus, $\H$ is a subgroup of $\mathsf{GL}_{d-2}(\mathbb{R})\times \mathsf{GL}_1(\mathbb{R})\times \mathsf{GL}_1(\mathbb{R})$. Now, consider the matrix 
 \[
 g\coloneqq\begin{pmatrix}[0.8] \textup{I}_{n-2} & &\\
 & 1 & \\ & & i \end{pmatrix}
 \] 
 centralizing $\H$. Observe that 
 $$g\Gamma g^{-1}\cap \mathsf{SL}_d(\mathbb{R})=\Gamma \cap \begin{pmatrix} \mathsf{GL}_{d-1}(\mathbb{R}) & \\ & \mathsf{GL}_1(\mathbb{R}) \end{pmatrix}.$$
 Thus, $\Gamma \cap g\Gamma g^{-1}$ fixes
$\xi^{1}(\gamma^{-})=[e_d]$ and, since $\langle \gamma \rangle<\Gamma$ is maximal, we have $\Gamma \cap g\Gamma g^{-1}=\textup{Stab}_{\Gamma}(\gamma^{-})=\H$. 
 
 Let $x,y\in \partial_{\infty}\Gamma\smallsetminus \{\gamma^{+},\gamma^{-}\}$. We write $\xi^1(x)=[v_x]$ and $\xi^{d-1}(y)=v_y^{\perp}$ for some non-zero vectors $v_x,v_y\in \mathbb{R}^d$. Since $\xi^1(x)\notin \xi^{d-1}(\gamma^{\pm})$ and $\xi^1(\gamma^{\pm})\notin \xi^{d-1}(y)$, it follows from \eqref{eqn:antipodal-cyclic} that $\langle v_x,e_d\rangle \ne 0$, $\langle v_y,e_d\rangle \neq 0$. 
 Since 
 $$\left \langle g^{\pm 1}v_x,v_y\right \rangle =\pm i\left \langle v_x,e_d\right \rangle \left \langle v_y,e_d\right \rangle+ \sum_{k=1}^{d-1}\left \langle v_x,e_{k}\rangle \right \langle v_y,e_{k}\rangle \ne 0,$$
 we get $g^{\pm 1}\xi^1(x)\notin \xi^{d-1}(y)$.
 This completes the proof of the lemma.\end{proof}

\subsection{Proof of \Cref{cyclic-amalgam}}
After composing $\Gamma<\mathsf{G}$ with a faithful  representation $\tau_{\Theta}:\mathsf{G} \rightarrow \mathsf{SL}_d(\mathbb{R})$ of minimal dimension $d=d(\mathsf{G},\Theta)$ (see Proposition \ref{theta-compatible}), we will assume that $\mathsf{G}=\mathsf{SL}_d(\mathbb{R})$ and that $\Gamma<\mathsf{SL}_d(\mathbb{R})$ is $1$-Anosov.

Let $(\xi^1,\xi^{d-1}):\partial_{\infty}\Gamma \rightarrow \mathbb{P}(\mathbb{R}^d)\times \mathsf{Gr}_{d-1}(\mathbb{R}^d)$ be the pair of Anosov limit maps of $\Gamma$. Let $\gamma \in \H$ be any infinite order element. Then $\langle \gamma \rangle$ has finite index in $\H$. Let $\widehat\H$ be the stabilizer of $\{\gamma^{+},\gamma^{-}\}$ in $\Gamma$. 
Then $\H$ is a finite-index subgroup of $\widehat\H$. Since $\H$ is separable in $\Gamma$  (by \cite{Tsouvalas-cyclic-separability}), there exists a finite-index subgroup $\Gamma_0$ of $\Gamma$, containing $\H$, such that $\Gamma_0\cap \widehat\H=\H$. 

By Lemma \ref{antipodal-cyclic}, there exists $g\in \mathsf{SL}_d(\mathbb{C})$, centralizing $\gamma$, such that $\Gamma \cap g \Gamma g^{-1}=\widehat\H$ and $g\xi^{r}\left(\partial_{\infty}\Gamma\smallsetminus \{\gamma^{+},\gamma^{-}\}\right)$ is antipodal to $\xi^{s}\left(\partial_{\infty}\Gamma\smallsetminus \{\gamma^{+},\gamma^{-}\}\right)$ whenever $\{r,s\}=\{1,d-1\}$. In particular, we have that  $g\Gamma_0g^{-1}\cap \Gamma_0=\H$. Thus, by applying Theorem \ref{thm:amalgam} to the $1$-Anosov subgroups $\Gamma_0,g\Gamma_0 g^{-1}<\mathsf{SL}_d(\mathbb{R})$ and their intersection $\H<\Gamma_0$, which is separable in $\Gamma_0$, we obtain the conclusion. \qed

\medskip

We conclude this section with a final remark that \Cref{cyclic-amalgam} (except for the bound on the dimension $d$) holds even if $A$ is finite; in fact, in this situation, one does not need to assume that $A$ is abelian.
 
\begin{remark}\label{rem:1.2}
For $1$-Anosov subgroups $\Gamma_1,\Gamma_2<\mathsf{SL}_n(\mathbb{R})$ with a finite intersection $\H\coloneqq \Gamma_1\cap \Gamma_2$, the amalgam $\Gamma_1\ast_{\H}\Gamma_2$ admits an Anosov representation into $\mathsf{SL}_m(\mathbb{R})$, for some $m\in \mathbb{N}$. To see this, observe that $\Gamma_1\ast_{\H}\Gamma_2$ contains a finite-index subgroup $\Gamma'$ which is isomorphic to $\Delta_1\ast\cdots \ast \Delta_p\ast F_{\ell}$, where $\Delta_i$ is a finite-index subgroup of either $\Gamma_1$ or $\Gamma_2$ and $F_{\ell}$ is the free group of rank $\ell$. Then, as $\Gamma'$ admits an Anosov representation (see \cite{DGK-comb, DT24}), $\Gamma$ also does (see \cite[Lem. 2.1]{DT24}).
\end{remark}

\section{Anosov amalgams along non-elementary subgroups}\label{sec:4}
Recall that a Lie group $\mathsf{G}$ has real rank $1$ if it is isogenous to $\mathsf{O}(n,1)$, $\mathsf{U}(n,1)$, $\mathsf{Sp}(n,1)$, $n\geq 2$, or to the exceptional Lie group $\mathsf{F}_{4}^{-20}$.
The main result of this section, stated below, provides sufficient conditions for constructing Anosov representations of (virtual) doubles of convex cocompact groups in $\mathsf{G}$ along certain non-elementary subgroups. To state this result, we first introduce a compatibility condition, defined as follows:

\begin{definition}\label{def:compatible}
    Let $\mathsf{G}$ be a rank 1 Lie group and let  $\mathsf{L}$ be a Zariski closed subgroup of $\mathsf{G}$. The pair $(\mathsf{G},\mathsf{L})$ is called {\em compatible} if there exists a Lie group automorphism $\varphi:\mathsf{G}\rightarrow \mathsf{G}$ such that the following two conditions hold:
\begin{enumerate}[label=(\roman*)]
    \item $\mathsf{L}=\big\{g\in \mathsf{G}:\ \varphi(g)=g\big\}$;
    \item $\Lambda_{\mathsf{L}}=\big\{x\in \partial_{\infty}(\mathsf{G}/\mathsf{K}):\ \partial\varphi(x)=x\big\}$.
\end{enumerate} 
Here $\mathsf{K}<\mathsf{G}$ is a maximal compact subgroup, $\mathsf{G}/\mathsf{K}$ is the rank 1 symmetric space associated to $\mathsf{G}$, $\geo (\mathsf{G}/\mathsf{K})$ is the visual boundary of $\mathsf{G}/\mathsf{K}$, 
$\partial \varphi: \partial_{\infty}(\mathsf{G}/\mathsf{K})\rightarrow \partial_{\infty}(\mathsf{G}/\mathsf{K})$ is the boundary extension of the $\varphi$-equivariant isometry
\[
\tilde\varphi: \mathsf{G}/\mathsf{K} \to \mathsf{G}/\mathsf{K}, \quad 
\tilde\varphi(g\mathsf{K}) = \varphi(g)g_0\mathsf{K},
\]
where $g_0\in\mathsf{G}$ be an element such that $\varphi (\mathsf K) = g_0 \mathsf{K} g_0^{-1}$,
and $\Lambda_{\mathsf{L}}$ is the set of accumulation points of any $\mathsf{L}$-orbit $\mathsf{L} \cdot x$, $x\in \mathsf{G}/\mathsf{K}$, in $\geo(\mathsf{G}/\mathsf{K})$.
\end{definition}

The main result of this section is as follows:

\begin{theorem}\label{amalgam-rank1} Let $\mathsf{G}$ be a rank 1 Lie group, let $(\mathsf{G},\mathsf{L})$ be a compatible pair, and let $\Gamma$ be a convex cocompact subgroup of $\mathsf{G}$.

\begin{enumerate} \item \label{amalgam-rank1-1} Suppose that $\Gamma \cap \mathsf{L}$ is quasiconvex in $\Gamma$ and that $\Lambda_{\Gamma \cap \mathsf{L}} = \Lambda_{\Gamma} \cap \Lambda_{\mathsf{L}}$ (e.g., $\mathsf{L}$ is reductive and $\Gamma \cap \mathsf{L}$ is a lattice in $\mathsf{L}$). Then there exists a finite-index subgroup $\Gamma_1$ of $\Gamma$, containing $\Gamma \cap \mathsf{L}$, such that the double $\Gamma_1 *_{\Gamma \cap \mathsf{L}} \Gamma_1$ admits a $2$-Anosov representation into $\mathsf{SL}_{m}(\mathbb{C})$, where $m = 2d(\mathsf{G})(d(\mathsf{G}) + 1)$.

\item \label{amalgam-rank-2} Let $(\mathsf{G},\mathsf{L}')$ be another compatible pair such that $\mathsf{L}'$ is the fixed point subgroup of an automorphism $\psi:\mathsf{G}\rightarrow \mathsf{G}$ of even order. Suppose that $\M\coloneqq\Gamma \cap \mathsf{L}\cap \mathsf{L}'$ is a quasiconvex subgroup of $\Gamma$ and $\Lambda_{\M} = \Lambda_{\Gamma} \cap \Lambda_{\mathsf{L}}\cap \Lambda_{\mathsf{L}'}$. Then there exists a finite-index subgroup $\Gamma_2$ of $\Gamma$, containing $\M$, such that the double $\Gamma_2 *_{\M} \Gamma_2$ admits a $2$-Anosov representation into $\mathsf{SL}_{r}(\mathbb{C})$, where 
$r\in \mathbb{N}$ depends only on $\mathsf{G}$ and the order of the automorphism $\psi$.
\end{enumerate}
\end{theorem}

This result, which we prove in \Cref{sec:4.2}, will play a crucial role in the proof of \Cref{main-rank1-amalgam} in the following section. The number $d(\mathsf{G})\in \mathbb{N}$ appearing in the statement denotes the minimal dimension $d$ in which $\mathsf{G}$ admits a proximal embedding into $\mathsf{SL}_d(\mathbb{R})$. Before discussing the proof of \Cref{amalgam-rank1}, we establish several technical results in the next subsection.

\subsection{Some preliminary results} First, we introduce some notation. We shall consider the embedding $\tau_2:\mathsf{Mat}_d(\mathbb{C})\rightarrow \mathsf{Mat}_{2d}(\mathbb{R})$, defined as follows:
\begin{equation}\label{tau2}
\tau_2(X+iY)=\begin{pmatrix} X& -Y  \\  Y & X\end{pmatrix}, \quad X,Y\in \mathsf{Mat}_d(\mathbb{R}).
\end{equation} 
This restricts to a representation $\tau_2:\mathsf{SL}_d(\mathbb{C})\rightarrow \mathsf{SL}_{2d}(\mathbb{R})$, which gives rise to the embeddings
\begin{align*}
\iota_2^{+}:\mathbb{P}(\mathbb{C}^d)\rightarrow \mathsf{Gr}_2(\mathbb{R}^{2d}),\quad 
\iota_2^{-}:\mathsf{Gr}_{d-1}(\mathbb{C}^d)\rightarrow \mathsf{Gr}_{2d-2}(\mathbb{R}^{2d}),
\end{align*} 
defined as follows:
\begin{alignat}{2}
\label{iota2}
\iota_2^{+}\big([u+iv]\big)&=\textup{span}\left\{\begin{pmatrix}u \\v\end{pmatrix},\begin{pmatrix}-v \\ u\end{pmatrix}\right\}, & u,v\in \mathbb{R}^d,\\ 
\label{iota2-}
\iota_2^{-}(W^{\perp})&=(\iota_2^{+}(W))^{\perp},  &W\in \mathbb{C}^d\smallsetminus \{0_d\}. 
\end{alignat} 
Note that for a $1$-Anosov subgroup $\Gamma$ of $\mathsf{SL}_d(\mathbb{C})$ with limit maps $(\xi^1,\xi^{d-1}):\partial_{\infty}\Gamma \rightarrow \mathbb{P}(\mathbb{C}^d)\times \mathsf{Gr}_{d-1}(\mathbb{C}^d)$, the subgroup $\tau_2(\Gamma)<\mathsf{SL}_{2d}(\mathbb{R})$ is $2$-Anosov with limit maps given by 
\[
(\iota_2^{+}\circ \xi^1,\iota_{2}^{-}\circ \xi^{d-1}) : \geo \G \to \mathsf{Gr}_2(\mathbb{R}^{2d}) \times \mathsf{Gr}_{2d-2}(\mathbb{R}^{2d}).
\]

For a subspace $V\subset \mathbb{C}^d$, let us denote its conjugate by  \[\overline{V}={\textup{span}\big\{x-iy:\ x+iy\in V,\ x,y\in \mathbb{R}^d\big\}}.\] We equip $\mathbb{C}^d$ with the standard orthonormal basis $\{e_1,\ldots,e_d\}$, whereas the exterior power $\bigwedge^2 \mathbb{C}^d$ is equipped with the inner product $\langle \cdot,\cdot\rangle$ which turns $\{e_i\wedge e_j:\ i<j\}$ into an orthonormal basis. 

\begin{lemma}\label{antipodal-1} 
Let $\Gamma<\mathsf{SL}_d(\mathbb{C})$ be a $1$-Anosov subgroup with limit maps $(\xi^1,\xi^{d-1}):\partial_{\infty}\Gamma \rightarrow \mathbb{P}(\mathbb{C}^d)\times \mathsf{Gr}_{d-1}(\mathbb{C}^d)$. 
Let $w\in \mathsf{GL}_{2d}(\mathbb{C})$ be given by
$$w =  \begin{pmatrix} \textup{I}_{d} &  \\   & i\textup{I}_{d}\end{pmatrix}.$$
Let $\H<\Gamma(\mathbb{R})$ be a quasiconvex subgroup of $\Gamma$ such that the following holds: 
\[
\text{for all }x\in \partial_{\infty}\Gamma,\quad 
\xi^1(x)\in \xi^{d-1}(x)\cap \overline{\xi^{d-1}}(x) \text{ if and only if } x\in \partial_{\infty}\H.
\]
Then, for all $x,y\in \partial_{\infty}\Gamma \smallsetminus \partial_{\infty}\H$, $w^{\pm 1}\iota_2^{+}(\xi^1(x))$ is antipodal to $\iota_2^{-}(\xi^{n-1}(y))$.
\end{lemma}

\begin{proof} Fixing $x,y\in \partial_{\infty}\Gamma \smallsetminus \partial_{\infty}\H$, let us write $\xi^1(x)=[u+iv]$ for some $u,v\in \mathbb{R}^d$. 
Observe that
$$\iota_2^{+}(\xi^1(x))=\textup{span}\left\{ \begin{pmatrix}u \\ v\end{pmatrix}, \begin{pmatrix}-v \\ u\end{pmatrix}\right\}, \quad  g\iota_2^{+}(x)=\textup{span}\left\{\begin{pmatrix}u \\ iv\end{pmatrix}, \begin{pmatrix}-v \\ iu\end{pmatrix}\right\}.$$ 
Let us also write $\xi^{d-1}(y)=\textup{span}\big\{a+ib\big\}^{\perp}$, for some $a,b\in \mathbb{R}^d$, such that $$\iota_2^{-}(\xi^{d-1}(y))=\textup{span}\left \{ \begin{pmatrix}a \\ b\end{pmatrix}, \begin{pmatrix}-b \\ a\end{pmatrix}\right\}^{\perp}.$$ 

If $g\iota_2^{+}(\xi^1(x))$ were not antipodal to $\iota_2^{-}(\xi^{d-1}(y))$, then the vector $\begin{pmatrix}u \\ iv\end{pmatrix}\wedge \begin{pmatrix}-v \\ iu\end{pmatrix} \in \bigwedge^2 \mathbb{C}^d$ would be contained in the hyperplane $\left\{\begin{pmatrix}a \\ b\end{pmatrix}\wedge \begin{pmatrix}-b \\ a\end{pmatrix}\right\}^{\perp}$ of $\bigwedge^2 \mathbb{C}^d$; in other words, we would have
$$\left \langle \begin{pmatrix}u \\ iv\end{pmatrix}\wedge \begin{pmatrix}-v \\ iu\end{pmatrix}, \begin{pmatrix}a \\ b\end{pmatrix}\wedge \begin{pmatrix}-b \\ a\end{pmatrix} \right \rangle=0.$$ 
Equivalently, by the definition of the Hermitian inner product on $\bigwedge^2\mathbb{C}^d$, we would have that \begin{align*} \textup{det}\begin{pmatrix} \langle (u,iv)^t, (a,b)^t\rangle   & \langle (u,iv)^t, (-b,a)^t\rangle  \\ \langle (-v,iu)^t, (a,b)^t\rangle & \langle (-v,iu)^t, (-b,a)^t\rangle \end{pmatrix}&=\textup{det}\begin{pmatrix}  \langle u,a\rangle+i\langle v,b \rangle  &  -\langle u,b\rangle +i\langle v,a\rangle  \\ -\langle v, a\rangle+i\langle u,b\rangle & i\langle u,a\rangle+ \langle v,b\rangle  \end{pmatrix}\\ &= i\big(\langle u,a\rangle^2+\langle u,b\rangle^2+\langle v,a\rangle^2+\langle v,b\rangle^2 \big), \end{align*}
implying that 
$$\langle u,a\rangle^2+\langle u,b\rangle^2+\langle v,a\rangle^2+\langle v,b\rangle^2=0,$$ hence $\langle u\pm iv, a+ib\rangle=0$ and $\xi^1(x)\in \xi^{d-1}(y)\cap \overline{\xi^{d-1}}(y)$. 
Since $\xi^1$ and $\xi^{d-1}$ are antipodal, we would get $x=y$ and then, by our hypothesis, $x\in \partial_{\infty}\H$. However, this would contradict the choice of $x,y\in \partial_{\infty}\Gamma$. Thus, the lemma follows.\end{proof}

A corollary of this lemma and \Cref{thm:amalgam}, which will be useful later in the section, is as follows:

\begin{corollary}\label{complex-1} Let $\Gamma<\mathsf{SL}_d(\mathbb{C})$ be a $1$-Anosov subgroup with limit maps $(\xi^1,\xi^{d-1}):\partial_{\infty}\Gamma\rightarrow \mathbb{P}(\mathbb{C}^d)\times \mathsf{Gr}_{d-1}(\mathbb{C}^d)$. Suppose that $\Gamma(\mathbb{R})$ is a quasiconvex subgroup of $\Gamma$ with the property that a point $x\in \partial_{\infty}\Gamma$ is in $\partial_{\infty}\Gamma(\mathbb{R})$ if and only if $\xi^1(x)\in \xi^{d-1}(x)\cap \overline{\xi^{d-1}}(x)$. 
Then there exists a finite index subgroup $\Gamma_1<\Gamma$, containing $\Gamma(\mathbb{R})$, such that the amalgam $\Gamma_1\ast_{\Gamma(\mathbb{R})}\Gamma_1$ admits a $2$-Anosov representation into $\mathsf{SL}_{2d}(\mathbb{C})$.\end{corollary}

\begin{proof} Let $\tau_2:\mathsf{GL}_d(\mathbb{C})\rightarrow \mathsf{GL}_{2d}(\mathbb{R})$ be the representation defined in (\ref{tau2}). 
Consider the matrix $g=\begin{pmatrix}[0.7]\textup{I}_{d} & \\   & i\textup{I}_{d}\end{pmatrix}$ centralizing $\tau_2(\Gamma(\mathbb{R}))<\tau_2(\Gamma)$. By Lemma \ref{antipodal-1}, the pairs of sets 
\begin{equation*}
g\iota_2^{+}\big(\xi^1(\partial_{\infty}\Gamma)\smallsetminus \xi^1(\partial_{\infty}\Gamma(\mathbb{R})) \big), \ \ \iota_{2}^{-}\left(\xi^{d-1}(\partial_{\infty}\Gamma)\smallsetminus \xi^{d-1}(\partial_{\infty}\Gamma(\mathbb{R})) \right) \end{equation*}
and
\begin{equation*}
        \iota_{2}^{+}\left(\xi^{d-1}(\partial_{\infty}\Gamma)\smallsetminus \xi^{d-1}(\partial_{\infty}\Gamma(\mathbb{R})) \right), \  \ g\iota_2^{-}\big(\xi^{d-1}(\partial_{\infty}\Gamma)\smallsetminus \xi^{d-1}(\partial_{\infty}\Gamma(\mathbb{R})) \big)
 \end{equation*} are antipodal. 
 Consequently, if 
 \[
 \xi_{2,2d}:\partial_{\infty}\Gamma\rightarrow \mathcal{F}_{2,2d-2}(\mathbb{C}), 
 \quad \eta\mapsto \left(\iota_2^{+}(\xi^1(\eta)),\iota_2^{-}(\xi^{d-1}(\eta))\right)
 \] 
 denotes the $\{2,2d-2\}$-Anosov limit map of $\tau_2(\Gamma)$, 
 then 
 $g\xi_{2,2d-2}(\partial_{\infty}\Gamma\smallsetminus \partial_{\infty}\Gamma(\mathbb{R}))$
 is antipodal to
 $\xi_{2,2d-2}(\partial_{\infty}\Gamma\smallsetminus \partial_{\infty}\Gamma(\mathbb{R}))$. 
 Since $\Gamma(\mathbb{R})<\Gamma$ is a separable subgroup, applying Theorem \ref{thm:amalgam} to $\tau_2(\Gamma), g\tau_2(\Gamma)g^{-1}<\mathsf{SL}_{2d}(\mathbb{C})$ and their intersection $\tau_2(\Gamma)\cap g\tau_2(\Gamma)g^{-1}=\tau_2(\Gamma(\mathbb{R}))$, 
 we obtain that there exists a finite index subgroup $\Gamma_1<\Gamma$, containing $\Gamma(\mathbb{R})$, such that the subgroup $\left\langle \tau_2(\Gamma_1),g\tau_2(\Gamma_1)g^{-1}\right\rangle<\mathsf{SL}_{2d}(\mathbb{C})$ is $2$-Anosov and is naturally isomorphic to $\Gamma_1\ast_{\Gamma(\mathbb{R})}\Gamma_1$.\end{proof}

Let 
\begin{align*}
 \nu_2^{+} :\mathsf{Gr}_2(\mathbb{C}^d)\rightarrow \mathbb{P}\Big(\bigwedge^2\mathbb{C}^d\Big)
 \quad\text{and}\quad
 \nu_2^{-} :\mathsf{Gr}_{d-2}(\mathbb{C}^d)\rightarrow \mathsf{Gr}_{m-1}\Big(\bigwedge^2\mathbb{C}^d\Big),
\end{align*}
where $m=\binom{d}{2}$,
denote the Pl\"ucker embeddings:
$$\nu_2^{+}\big(\textup{span}\{v_1,v_2\}\big)\coloneqq [v_1\wedge v_2], \quad \nu_2^{-}\big(\textup{span}\{v_1,v_2\}^{\perp}\big)\coloneqq \big(\textup{span}\{v_1\wedge v_2\}\big)^{\perp}.$$ 
Here $V^{\perp}$ denotes the orthogonal complement of a subspace $V$ of $\bigwedge^2 \mathbb{C}^d$.

For a $2$-Anosov subgroup $\Gamma<\mathsf{SL}_d(\mathbb{C})$ with limit maps $(\xi^2,\xi^{d-2}):\partial_{\infty}\Gamma \rightarrow \mathsf{Gr}_2(\mathbb{C}^d)\times \mathsf{Gr}_{d-2}(\mathbb{C}^d)$, the subgroup $\bigwedge^2\Gamma<\mathsf{SL}(\bigwedge^2\mathbb{C}^d)$
is $1$-Anosov with limit maps given by 
\[
\big(\nu_2^{+}\circ \xi^2,\nu_{2}^{-}\circ \xi^{d-2}\big) :\partial_{\infty}\Gamma \rightarrow \mathbb{P}\Big(\bigwedge^2\mathbb{C}^{d}\Big) \times \mathsf{Gr}_{m-1}\Big(\bigwedge^2\mathbb{C}^{d}\Big), 
\] $m=\binom{d}{2}$, see for example \cite[Prop. 3.5]{GGKW}.

\begin{proposition}\label{rho1-rho2} 
Let $\Gamma$ be a hyperbolic group and let $\rho_i:\Gamma \rightarrow \mathsf{GL}_d(\mathbb{R})$, $i=1,2$, be two faithful $1$-Anosov representations with Anosov limit maps $(\xi_{\rho_i}^{1},\xi_{\rho_i}^{d-1}):\partial_{\infty}\Gamma\rightarrow \mathbb{P}(\mathbb{R}^d)\times \mathsf{Gr}_{d-1}(\mathbb{R}^d)$.
Suppose that $$\H\coloneqq\big\{\gamma \in \Gamma: \rho_1(\gamma)=\rho_2(\gamma)\big\}$$ is a quasiconvex, separable subgroup of $\Gamma$ with the following property: for $x\in \partial_{\infty}\Gamma$, we have $x\in \partial_{\infty}\H$ if and only if either $\xi_{\rho_1}^{1}(x)\in \xi_{\rho_2}^{d-1}(x)$ or $\xi_{\rho_2}^{1}(x)\in \xi_{\rho_1}^{d-1}(x)$. 

Then there exists a finite-index subgroup $\Gamma_1<\Gamma$, containing $\H$, such that the amalgam $\Gamma_1\ast_{\H}\Gamma_1$ admits a faithful $2$-Anosov representation into $\mathsf{SL}_{m}(\mathbb{C})$, $m=2d(2d+1)$.\end{proposition}

\begin{proof} 
Consider the representation $\rho:\Gamma \rightarrow \mathsf{SL}(\mathbb{C}^{2d}\oplus \mathbb{C}e_{2d+1})$ defined as follows:
for $\gamma \in \Gamma$,
\begin{align}\label{rho-rho1-rho2}\rho(\gamma)=\begin{pmatrix}g \begin{pmatrix}[0.5] \rho_1(\gamma) & \\ & \rho_2(\gamma)\end{pmatrix} g^{-1} & \\ &  &  1\end{pmatrix}, \quad \textup{where} \ g=\begin{pmatrix} \textup{I}_{d} & \frac{i}{2}\textup{I}_{d}\\ i\textup{I}_d &  \frac{1}{2}\textup{I}_d \end{pmatrix}\in \mathsf{SL}_{2d}(\mathbb{C}).\end{align} 
Note that $\rho$ is $2$-Anosov and its Anosov limit maps $(\xi_{\rho}^2,\xi_{\rho}^{2d-1}):\partial_{\infty}\Gamma \rightarrow \mathsf{Gr}_2(\mathbb{C}^{2d+1})\times \mathsf{Gr}_2(\mathbb{C}^{2d+1})$ are given by
\begin{align*}
\xi_{\rho}^2(x)&=g\big( \xi_{\rho_1}^1(x)\oplus \xi_{\rho_2}^1(x)\big),\\
\xi_{\rho}^{2d-1}(x)&= g\big( \xi_{\rho_1}^{d-1}(x)\oplus \xi_{\rho_2}^{d-1}(x)\big)\oplus \mathbb{C}e_{2d+1},
\end{align*} 
for $x\in \partial_{\infty}\Gamma$.

The second exterior power of $\rho$, $\bigwedge^2 \rho:\Gamma \rightarrow \mathsf{SL}(\bigwedge^2 \mathbb{C}^{2d+1})$ is a faithful $1$-Anosov representation whose Anosov limit maps $(\xi_{\bigwedge^2 \rho}^1,\xi_{\bigwedge^2 \rho}^{m-1})$, where $m=(2d+1)d$,  are given by \begin{align}\label{xi-1-psi} \xi_{\bigwedge^2 \rho}^1(x)&= \bigwedge^2\begin{pmatrix}[0.6] g & \\ & 1\end{pmatrix} \big(\nu_2^{+}\big( \xi_{\rho_1}^1(x)\oplus \xi_{\rho_2}^1(x)\big)\big),\\ \label{xi-psi-m-1} 
\xi_{\bigwedge^2 \rho}^{m-1}(x)&= \bigwedge^2\begin{pmatrix}[0.6] g & \\ & 1\end{pmatrix}\big(\nu_2^{-}\big( \xi_{\rho_1}^{d-1}(x)\oplus \xi_{\rho_2}^{d-1}(x)\oplus \mathbb{C}e_{2d+1}\big)\big),\end{align}
for $x\in \partial_{\infty}\Gamma$.

We will check that the $1$-Anosov group $\bigwedge^2\rho(\Gamma)$ and its subgroup $\bigwedge^2\rho(\H)$ satisfy the conditions of Corollary \ref{complex-1}:
First, since 
\[
\bigwedge^2 \begin{pmatrix}[0.6] \mathsf{SL}_{2d}(\mathbb{C}) & \\ &1 \end{pmatrix} \cap \mathsf{SL}\Big(\bigwedge^2 \mathbb{R}^{2d+1}\Big)=\bigwedge^2 \begin{pmatrix}[0.6] \mathsf{SL}_{2d}(\mathbb{R}) & \\ &1 \end{pmatrix}
\quad\text{and}\quad
\rho(\Gamma)\cap \begin{pmatrix}[0.6] \mathsf{SL}_{2d}(\mathbb{R}) & \\ &1 \end{pmatrix}=\rho(\H),
\]
we have that $$\Big(\bigwedge^2\rho(\Gamma)\Big)(\mathbb{R})=\bigwedge^2\rho(\Gamma)\cap \bigwedge^2 \mathsf{SL}_{2d}(\mathbb{C})\cap \mathsf{SL}\Big(\bigwedge^2\mathbb{R}^{2d+1}\Big)=\bigwedge^2 \rho(\H).$$ 
Now suppose that $x\in \partial_{\infty}\Gamma$ such that $ \xi_{\bigwedge^2 \rho}^1(x) \in \overline{\xi_{\bigwedge^2 \rho}^{m-1}}(x)$.
Equivalently (by the definition of the limit map of $\bigwedge^2 \rho$), 
\begin{equation*} \bigwedge^2g \nu_2^{+}\big( \big( \xi_{\rho_1}^1(x)\oplus \xi_{\rho_2}^1(x)\big)\big)\in  \bigwedge^2\overline{g} \big(\nu_2^{-}\big(\xi_{\rho_1}^{d-1}(x)\oplus \xi_{\rho_2}^{d-1}(x)\big)\big).
\end{equation*}
Hence
\begin{equation}
\label{non-antipodal} \bigwedge^2(\overline{g}^{-1}g) \nu_2^{+}\big( \big( \xi_{\rho_1}^1(x)\oplus \xi_{\rho_2}^1(x)\big)\big)
\in  \bigwedge^2\overline{g} \big(\nu_2^{-}\big(\xi_{\rho_1}^{d-1}(x)\oplus \xi_{\rho_2}^{d-1}(x)\big)\big).\end{equation}
We write
\begin{equation*}
     \begin{array}{lll}
        \xi_{\rho_1}^1(x)=[a_{1x}],  && \xi_{\rho_1}^{d-1}(x)=\textup{span}\{b_{1x}\}^{\perp}, \\
        \xi_{\rho_2}^1(x)=[a_{2x}], && \xi_{\rho_2}^{d-1}(x)=\textup{span}\{b_{2x}\}^{\perp},
     \end{array}
 \end{equation*} for some $a_{1x},a_{2x},b_{1x},b_{2x}\in \mathbb{R}^d\smallsetminus \{0_d\}$.
 Then, by (\ref{non-antipodal}), we have 
 $$ \left \langle  \Big(\bigwedge^2\overline{g}^{-1}g\Big)\begin{pmatrix}a_{1x}\\ 0_{d}\end{pmatrix}\wedge  \begin{pmatrix}0_d\\ a_{2x} \end{pmatrix}, \begin{pmatrix}b_{1x}\\ 0_d \end{pmatrix}\wedge  \begin{pmatrix}0_d\\ b_{2x}\end{pmatrix}  \right \rangle =0.$$
 Since 
 $\overline{g}^{-1}g=\begin{pmatrix}[0.8]  & \frac{i}{2}\textup{I}_{d}\\ 2i\textup{I}_d &  \end{pmatrix}$
 and
 $$\bigwedge^2 (\overline{g}^{-1}g) \left(\begin{pmatrix}a_{1x}\\ 0_d\end{pmatrix}\wedge  \begin{pmatrix}0_d \\ a_{2x}\end{pmatrix}\right)=\begin{pmatrix}0_d\\ 2i a_{1x}\end{pmatrix}\wedge  \begin{pmatrix} \frac{i}{2}a_{2x} \\ 0_d\end{pmatrix},$$ we conclude that 
 $$\left \langle \Big(\bigwedge^2 \overline{g}^{-1}g\Big) \begin{pmatrix}a_{1x}\\ 0_d\end{pmatrix}\wedge  \begin{pmatrix}0_d\\ a_{2x}\end{pmatrix}, \begin{pmatrix}b_{1x}\\ 0_d\end{pmatrix}\wedge  \begin{pmatrix}0_d\\ b_{2x}\end{pmatrix} \right \rangle=\langle a_{1x},b_{2x}\rangle \langle a_{2x},b_{1x}\rangle=0.$$ 
 Therefore, either $\langle a_{1x},b_{2x}\rangle=0$ or $\langle a_{2x},b_{1x}\rangle=0$; equivalently, we have either $\xi_{\rho_1}^1(x)\in \xi_{\rho_2}^{d-1}(x)$ or $\xi_{\rho_2}^1(x)\in \xi_{\rho_1}^{d-1}(x)$. In particular, by our hypothesis, we see that $x\in \partial_{\infty}\H$. Thus, $\bigwedge^2\rho(\Gamma)\cong \Gamma<\mathsf{SL}\big(\bigwedge^2 \mathbb{C}^{2d+1}\big)$ and its quasiconvex separable subgroup $\bigwedge^2\rho(\H)\cong \H$ satisfy the assumptions of Corollary \ref{complex-1}, which concludes the proof.
\end{proof}

\begin{proposition} \label{prop:4.4}
Let $\Gamma$ be a hyperbolic group and let $\psi:\Gamma \rightarrow \mathsf{SL}_d(\mathbb{C})$ be a $1$-Anosov representation with Anosov limit maps $(\xi_{\psi}^{1},\xi_{\psi}^{d-1}):\partial_{\infty}\Gamma\rightarrow \mathbb{P}(\mathbb{C}^d)\times \mathsf{Gr}_{d-1}(\mathbb{C}^d)$.
Assume that $-\textup{I}_{d} \not\in \psi(\Gamma)$. For an order $2$ element $w\in \mathsf{SL}_d(\mathbb{R})$, let $$\H\coloneqq\big\{\gamma \in \Gamma:\ \overline{\psi}(\gamma)=\psi(\gamma)=w\psi(\gamma)w^{-1}\big\}.$$ 
Let $\H'<\H$ is a quasiconvex, separable subgroup of $\Gamma$ with the following property: a point $x\in \partial_{\infty}\Gamma$ lies in $\partial_{\infty}\H'$ if and only if the following two conditions are simultaneously satisfied:
\begin{enumerate}[label=(\roman*)]
    \item $\xi_{\psi}^{1}(x)\in \overline{\xi_{\psi}^{d-1}}(x)$, and 
    \item $w \xi_{\psi}^{1}(x)\in \xi_{\psi}^{d-1}(x)$.
\end{enumerate}
Then there exists a finite-index subgroup $\Gamma_1<\Gamma$, containing $\H'$, such that the amalgam $\Gamma_1\ast_{\H'}\Gamma_1$ admits a $2$-Anosov representation into $\mathsf{SL}_{r}(\mathbb{C})$, where $r=2\binom{4d(2d-1)}{4}\left(2\binom{4d(2d-1)}{4}+1\right)$.\end{proposition}

\begin{proof} Consider the representation $\tau_2:\mathsf{SL}_d(\mathbb{C})\rightarrow \mathsf{SL}(\mathbb{R}^{d}\oplus \mathbb{R}^{d})$ defined in (\ref{tau2}). Observe that $\bigwedge^2\tau_2\circ \psi:\Gamma\rightarrow \mathsf{SL}(\bigwedge^2 V)$, where $V\coloneqq\mathbb{R}^d\oplus \mathbb{R}^d$, is a faithful $1$-Anosov representation with limit maps $(\xi^1,\xi^{m-1}):\partial_{\infty}\Gamma \rightarrow \mathbb{P}(\bigwedge^2 V)\times \mathsf{Gr}_{m-1}(\bigwedge^2 V)$, where $m=d(2d-1)$, defined as follows: 
\begin{align}
\label{xi1-def} \xi^1(x)&=\left[\begin{pmatrix}u \\ v\end{pmatrix} \wedge \begin{pmatrix}-v \\ u\end{pmatrix}\right],\\ \label{xim-1-d} \xi^{m-1}(x)&=\textup{span}\left \{ \begin{pmatrix} t \\ z\end{pmatrix}\wedge \begin{pmatrix}-z \\ t\end{pmatrix}\right \}^{\perp},
\end{align}
where $\xi_{\psi}^1(x)=[u+vi], \xi_{\psi}^{d-1}(x)=\textup{span}\{t+zi\}^{\perp}$ for $u,v,t,z\in \mathbb{R}^{d}$.

Fix the canonical basis $\{e_{1},\ldots,e_{d}\}$ of $\mathbb{R}^{d}$ and the decomposition $$\bigwedge^2V=\bigwedge^2 V_1\oplus U\oplus \bigwedge^2 V_2,$$ where 
\begin{equation*}
\begin{gathered}
V_1=\left\{\begin{pmatrix}a \\ 0_{d}\end{pmatrix}:\ a\in \mathbb{R}^n\right\}, \ \ V_2=\left\{\begin{pmatrix}0_d \\ a\end{pmatrix}:\ a\in \mathbb{R}^n\right\}, \\ U=\textup{span}\left\{\begin{pmatrix}e_p \\ 0_{d}\end{pmatrix} \wedge \begin{pmatrix}0_{d} \\ e_q\end{pmatrix}:\ p,q=1,\ldots,d\right\}.
\end{gathered}
\end{equation*}
Let $\mathbb{H}=\mathbb{R}\oplus \mathbb{R}i\oplus \mathbb{R}j\oplus \mathbb{R}k$ be Hamilton's quaternions. Given the real vector space $\bigwedge^2 V$, consider the right $\mathbb{H}$-vector space: \begin{equation}\label{splitting-H}\mathbb{H}^{m}=\underbrace{\left(\bigwedge^2 V_1\right)\mathbb{H}}_{W_1\cong \mathbb{H}^{d(d-1)/2}}\oplus \underbrace{U\mathbb{H}}_{W_2\cong \mathbb{H}^{d^2}}\oplus \underbrace{\left(\bigwedge^2 V_2\right)\mathbb{H} }_{W_3\cong \mathbb{H}^{d(d-1)/2}},\end{equation} 
equipped with the inner product $\langle \cdot,\cdot\rangle$ with orthonormal basis the union of the orthonormal bases of $\bigwedge^2 V_1$, $U$, $\bigwedge^2 V_2$ (as real vector spaces).

Let $f_{\pm 1}\in \mathsf{GL}(W_2)$ be the unique linear transformation defined as follows: \begin{equation}\label{def-f}f_{\pm 1} \left(\begin{pmatrix}a \\ 0_{d}\end{pmatrix} \wedge \begin{pmatrix}0_{d} \\ b\end{pmatrix} \right)=\frac{1}{\sqrt{2}}\begin{pmatrix}a \\ 0_{d}\end{pmatrix} \wedge \begin{pmatrix}0_{d} \\ b\end{pmatrix}\pm \frac{i}{\sqrt{2}}\begin{pmatrix}w a \\ 0_{d}\end{pmatrix} \wedge \begin{pmatrix}0_{d} \\ w b\end{pmatrix}\end{equation} for $a,b\in \mathbb{R}^{d}$. Note that $f_{1}f_{-1}=\textup{I}_{W_2}$. Hence $f_{\pm 1}$ is well-defined and invertible. 

The representation $\bigwedge^2\tau_2\circ \psi:\Gamma\rightarrow \mathsf{SL}(\mathbb{H}^m)$ is also $1$-Anosov with limit maps $(\xi_{\mathbb{H}}^1,\xi_{\mathbb{H}}^{m-1}):\partial_{\infty}\Gamma\rightarrow \mathbb{P}(\mathbb{H}^m)\times \mathsf{Gr}_{m-1}(\mathbb{H}^m)$ defined from (\ref{xi1-def}) and (\ref{xim-1-d}) as follows for $x\in \partial_{\infty}\Gamma$: \begin{align}\label{xi1-def2}\mathbb{\xi}_{\mathbb{H}}^1(x)=\left[\begin{pmatrix}u \\ v\end{pmatrix} \wedge \begin{pmatrix}-v \\ u\end{pmatrix}\right], \ \xi^{m-1}_{\mathbb{H}}(x)=\left\{\omega\in \mathbb{H}^m: \left\langle\begin{pmatrix} t \\ z\end{pmatrix}\wedge \begin{pmatrix}-z  \\ t\end{pmatrix} ,\omega\right\rangle=0\right\},\end{align} where $\langle \cdot,\cdot\rangle$ is the inner product on $\mathbb{H}^m$.

With respect to the splitting of $\mathbb{H}^m=W_1\oplus W_2\oplus W_3$ in (\ref{splitting-H}) and using $f_{\pm 1}\in \mathsf{GL}(W_2)$, defined in (\ref{def-f}), we define the element $\beta\in \mathsf{GL}(\mathbb{H}^m)$: 
\begin{equation}\label{def-B} \beta\coloneqq\begin{pmatrix} j\textup{I}_{W_1} & &\\ & f_1 & \\ & & j\textup{I}_{W_3} \end{pmatrix}.\end{equation}

Recall the definition of the quasiconvex subgroup $\H<\Gamma$ defined in the statement of the proposition. We have:

\begin{nclaim} \label{claim-0} For a point $x\in \partial_{\infty}\Gamma$, if $\beta^{\pm 1}\xi_{\mathbb{H}}^1(x)\in \xi_{\mathbb{H}}^{m-1}(x)$, then $x\in \partial_{\infty}\H$.\end{nclaim}

\begin{proof}[Proof of Claim.] Write $\xi_{\psi}^1(x)=[u+vi],\ \xi_{\psi}^{m-1}(x)=\textup{span}\{t+zi\}^{\perp}$ for some $u,v,t,z\in \mathbb{R}^n$ such that $\xi_{\mathbb{H}}^1(x)\in \mathbb{P}(\mathbb{H}^m)$ and $\xi_{\mathbb{H}}^{m-1}(x)\in \mathsf{Gr}_{m-1}(\mathbb{H}^m)$ are given by (\ref{xi1-def2}). 
Observe that for $\varepsilon\in \{-1,1\}$, we have 
\begin{align*} \beta^{\varepsilon}\left(\begin{pmatrix}u \\ v\end{pmatrix} \wedge \begin{pmatrix}-v \\ u \end{pmatrix}\right)&=\beta^{\varepsilon}\left(\underbrace{\begin{pmatrix}u \\ 0_{d}\end{pmatrix} \wedge \begin{pmatrix}0_{d} \\ u \end{pmatrix}+\begin{pmatrix}v \\ 0_{d}\end{pmatrix} \wedge \begin{pmatrix}0_{d} \\ v \end{pmatrix}}_{\in W_2} \right)\\&\qquad\qquad\qquad\qquad\qquad+\beta^{\varepsilon}\left(\underbrace{\begin{pmatrix}u \\ 0_{d}\end{pmatrix} \wedge \begin{pmatrix}-v \\ 0_{d} \end{pmatrix}}_{\in W_1}+\underbrace{\begin{pmatrix}0_{d} \\ v\end{pmatrix} \wedge \begin{pmatrix}0_{d} \\ u \end{pmatrix}}_{\in W_3} \right)\\  &= \varepsilon j B_1+ f_{\varepsilon} \left(\begin{pmatrix}u \\ 0_{d}\end{pmatrix} \wedge \begin{pmatrix}0_{d} \\ u \end{pmatrix}+ \begin{pmatrix}v \\ 0_{d}\end{pmatrix} \wedge \begin{pmatrix}0_{d} \\ v \end{pmatrix} \right)\\ & =\varepsilon j B_1+\frac{1}{\sqrt{2}}B_2+\frac{\varepsilon i}{\sqrt{2}}B_3,
\end{align*}
where
\[
B_1\coloneqq\begin{pmatrix}u \\ 0_{d}\end{pmatrix} \wedge \begin{pmatrix}-v \\ 0_{d} \end{pmatrix}+ \begin{pmatrix} 0_{d}\\ v\end{pmatrix}\wedge \begin{pmatrix} 0_{d}\\ u\end{pmatrix}, \quad B_2\coloneqq\begin{pmatrix}u \\ 0_{d}\end{pmatrix} \wedge \begin{pmatrix}0_{d} \\ u \end{pmatrix}+ \begin{pmatrix}v \\ 0_{d}\end{pmatrix} \wedge \begin{pmatrix}0_{d} \\ v \end{pmatrix},
\]
\[
B_3\coloneqq\begin{pmatrix}w u \\ 0_{d}\end{pmatrix} \wedge \begin{pmatrix}0_{d} \\ wu \end{pmatrix}+ \begin{pmatrix}wv \\ 0_{d}\end{pmatrix} \wedge \begin{pmatrix}0_{d} \\ wv \end{pmatrix}.
\]

Since we assumed that $\beta^{\pm 1}\xi^1(x)\in \xi^{m-1}(x)$, we have $$\left\langle  \begin{pmatrix}t \\z\end{pmatrix}\wedge \begin{pmatrix}-z \\ t\end{pmatrix},j\varepsilon B_1\right \rangle+\left\langle  \begin{pmatrix}t \\z\end{pmatrix}\wedge \begin{pmatrix}-z \\ t\end{pmatrix},\frac{1}{\sqrt2}B_2+\frac{i\varepsilon}{\sqrt{2}}B_3\right \rangle=0,$$
which implies that \begin{equation}\label{vanishing}\left\langle  \begin{pmatrix}t \\z\end{pmatrix}\wedge \begin{pmatrix}-z \\ t\end{pmatrix}, B_p\right \rangle=0, \quad p=1,2,3.\end{equation} 
A direct computation yields
\begin{align*}\left \langle \begin{pmatrix}t \\z\end{pmatrix}\wedge \begin{pmatrix}-z \\ t\end{pmatrix}, B_2\right\rangle &=\left|\langle t,u\rangle\right|^2+\left| \langle z,u\rangle\right|^2+ \left|\langle t,v\rangle \right|^2+\left|\langle z,v\rangle\right|^2,\\ \left \langle \begin{pmatrix}t \\ z\end{pmatrix}\wedge \begin{pmatrix}-z \\ t\end{pmatrix}, B_3\right \rangle &=\left|\langle t,wu\rangle\right|^2+ \left| \langle z,wu \rangle\right|^2 
+ \left|\langle t,wv\rangle  \right|^2+\left|\langle z,w v\rangle\right|^2. \end{align*} 
Combining this with (\ref{vanishing}), we obtain
\begin{align*}\langle t\pm zi, u+vi\rangle= \langle t+zi,w u\rangle=\langle t+zi,wv\rangle=0. \end{align*} 
In particular, $\xi_{\psi}^1(x)\in \overline{\xi_{\psi}^{d-1}}(x)$ and $w\xi_{\psi}^1(x)\in \xi_{\psi}^{d-1}(x)$.
Therefore, by our hypothesis (see the statement of \Cref{prop:4.4}), we have $x\in \partial_{\infty}\H$.\end{proof}

Now we conclude this proof by showing how to virtually amalgamate the Anosov subgroup $\Gamma$ along the quasiconvex subgroup $\H\coloneqq\big\{\gamma \in \Gamma: w\psi(\gamma) w^{-1}=\psi(\gamma)=\overline{\psi}(\gamma)\big\}$: 
consider the faithful $1$-Anosov representations $\psi_1,\psi_2:\Gamma\rightarrow \mathsf{SL}(\mathbb{H}^m)$, where $$\psi_1\coloneqq\bigwedge^2\big(\tau_2\circ \psi\big) \quad\text{and}\quad \psi_2=\beta\psi_1\beta^{-1}.$$ 
Since $w\in \mathsf{SL}_d(\mathbb{R})$, the subgroup $\bigwedge^2 \mathsf{D}<\mathsf{SL}(\mathbb{H}^m)$, where $$\mathsf{D}\coloneqq\left\{\begin{pmatrix}[0.6] X & \\ & X\end{pmatrix}: wXw^{-1}=X,X\in \mathsf{GL}_{d}(\mathbb{R})\right\},$$ is centralized by $\beta\in \mathsf{GL}(\mathbb{H}^m)$ defined in (\ref{def-B}). 
Since $\psi_1(\H)<\bigwedge^2 \mathsf{D}$ we obtain that $\H$ is a subgroup of $$\H'\coloneqq\big\{\gamma \in \Gamma:\psi_1(\gamma)=\psi_2(\gamma)\big\}.$$ By the definition of $\H'<\Gamma$, for every $x\in \partial_{\infty}\H' \subset \partial_{\infty}{\Gamma}$ we have $\beta^{\pm 1}\xi^1(x)\in \xi^{m-1}(x)$. Thus, by \Cref{claim-0}, we conclude that $\partial_{\infty}\H$ is the limit set of $\H'$ in $\partial_{\infty}\Gamma$, and since $\H<\Gamma$ is quasiconvex, we conclude that $\H$ has  finite index in $\H$. To this end, as $\H<\Gamma$ is separable, we can (and will) choose a finite-index subgroup $\Gamma_0<\Gamma$ such that $\Gamma_0\cap \H'=\H$. 

For the faithful 1-Anosov representations $\psi_1,\psi_2:\Gamma_0\rightarrow \mathsf{SL}(\mathbb{H}^m)$, we have that $\H=\{\gamma \in \Gamma_0:\psi_1(\gamma)=\psi_2(\gamma)\}$ and, by Claim \ref{claim-0}, if either $\xi_{\psi_1}^1(x)\in \xi_{\psi_2}^{m-1}(x)$ or $\xi_{\psi_2}^1(x)\in \xi_{\psi_1}^{m-1}(x)$, then $x\in \partial_{\infty}\H'$. Let $\mathsf{SL}(\mathbb{H}^m)\xhookrightarrow{} \mathsf{SL}_{4m}(\mathbb{R})$ be the extension of scalars representation by realizing $\mathbb{H}^m$ as a $4m$-dimensional real vector space. For $i\in \{1,2\}$, since $\psi_i$ is $1$-Anosov, the representations $\bigwedge^4\psi_i:\Gamma_0 \rightarrow \mathsf{SL}\big(\bigwedge^{4}\mathbb{R}^{4m}\big)$ are $1$-Anosov and satisfy the assumptions of Proposition \ref{rho1-rho2} by the previous discussion. Therefore, there exists a $2$-Anosov representation of the amalgam $\Gamma_1\ast_{\H}\Gamma_1$ into $\mathsf{SL}_r(\mathbb{C})$, where $\Gamma_1<\Gamma$ is a finite-index subgroup containing $\H'$ and $r=2\binom{4d(2d-1)}{4}\left(2\binom{4d(2d-1)}{4}+1\right)$.\end{proof}

As a consequence, we obtain the following:

\begin{proposition}\label{rho1-rho2-w} Let $\Gamma$ be a hyperbolic group and let $\rho_1:\Gamma \rightarrow \mathsf{GL}_d(\mathbb{R})$, $\rho_2:\Gamma \rightarrow \mathsf{GL}_d(\mathbb{R})$ be two faithful $1$-Anosov representations with Anosov limit maps $(\xi_{\rho_i}^{1},\xi_{\rho_i}^{d-1}):\partial_{\infty}\Gamma\rightarrow \mathbb{P}(\mathbb{R}^d)\times \mathsf{Gr}_{d-1}(\mathbb{R}^d)$, $i=1,2$. Let $w_0\in \mathsf{GL}_d(\mathbb{R})$ be an order $2$ element.
Suppose that
$$\M\coloneqq\big\{\gamma \in \Gamma: \rho_1(\gamma)=\rho_2(\gamma)=w_0\rho_1(\gamma)w_0^{-1}\big\}$$ is a quasiconvex, separable subgroup of $\Gamma$ with the following property: a point $x\in \partial_{\infty}\Gamma$ lies in $\partial_{\infty}\M$ if and only if the following conditions are simultaneously satisfied:
\begin{enumerate}[label=(\roman*)]
\item $\xi_{\rho_1}^{1}(x)\in \xi_{\rho_2}^{d-1}(x)$ or $\xi_{\rho_2}^{1}(x)\in \xi_{\rho_1}^{d-1}(x)$; and
\item  $w_0\xi_{\rho_1}^{1}(x)\in \xi_{\rho_1}^{d-1}(x)$ or $w_0\xi_{\rho_2}^{1}(x)\in \xi_{\rho_2}^{d-1}(x)$.
\end{enumerate}
 Then there exists a finite-index subgroup $\Gamma_2<\Gamma$, containing $\M$, such that the amalgam $\Gamma_2\ast_{\M}\Gamma_2$ admits a $2$-Anosov representation into $\mathsf{SL}_{r}(\mathbb{C})$, $r=2\binom{4p(2p-1)}{4}\left(2\binom{4p(2p-1)}{4}+1\right)$, where $p=d(2d+1)$.
\end{proposition}

\begin{proof} First, we consider the $2$-Anosov representation $\rho:\Gamma \rightarrow \mathsf{SL}_{}(\mathbb{C}^{2d}\oplus \mathbb{C}e_{2d+1})$, obtained from $\rho_1$ and $\rho_2$ in (\ref{rho-rho1-rho2}) and the $1$-Anosov representation $\psi:\Gamma \rightarrow \mathsf{SL}_{m}(\mathbb{C})$, $\psi\coloneqq\bigwedge^2 \rho$, $m=d(2d+1)$, with Anosov limit maps $(\xi_{\psi}^1,\xi_{\psi}^{m-1})$ defined in (\ref{xi-1-psi}) and (\ref{xi-psi-m-1}). Define also the order $2$ element of $w\in \mathsf{SL}\big(\bigwedge^2 (\mathbb{C}^{2d}\oplus \mathbb{C}e_{2d+1})\big)$, $$w\coloneqq \bigwedge^2\begin{pmatrix} w_0 & &\\ & w_0 & \\ & & 1\end{pmatrix}.$$ 

It is clear from the proof of Proposition \ref{rho1-rho2} that for a point $x\in \partial_{\infty}\Gamma$:
\smallskip

\noindent \textup{(a)} If $\xi_{\psi}^1(x)\in \overline{\xi_{\psi}^{m-1}}(x)$, then either $\xi_{\rho_1}^1(x)\in \xi_{\rho_2}^{d-1}(x)$ or $\xi_{\rho_1}^1(x)\in \xi_{\rho_2}^{d-1}(x)$; and \\
\noindent \textup{(b)} If $w\xi_{\psi}^1(x)\in \xi_{\psi}^{m-1}(x)$, then, since $g\in \mathsf{SL}_{2n}(\mathbb{C})$ is the matrix defined in (\ref{rho-rho1-rho2}) commutes with $\begin{pmatrix}[0.8] w_0 & \\ & w_0\end{pmatrix}$, we have \begin{align*}w\xi_{\psi}^1(x)&=\bigwedge^2 \begin{pmatrix}[0.8] g & \\ & 1\end{pmatrix} \big(\nu_2^{+}\big(w_0\xi_{\rho_1}^1(x)\oplus w_0 \xi_{\rho_2}^1(x)\big)\big),\\ \xi_{\psi}^{m-1}(x)&=\bigwedge^2 \begin{pmatrix}[0.8] g & \\ & 1\end{pmatrix} \big(\nu_2^{-}\big(\xi_{\rho_1}^{d-1}(x)\oplus \xi_{\rho_2}^{d-1}(x)\oplus \mathbb{C}e_{2d+1}\big)\big).\end{align*} 
Therefore, $$\left(w_0\xi_{\rho_1}^1(x)\oplus w_0\xi_{\rho_2}^1(x)\right)\cap \left(\xi_{\rho_1}^{d-1}(x)\oplus \xi_{\rho_2}^{d-1}(x)\right)\neq (0)$$ and $w_0\xi_{\rho_i}^1(x)\in \xi_{\rho_i}^{d-1}(x)$ for some $i\in \{1,2\}$.
\smallskip

By the definition of $\psi$, $\M<\Gamma$ is a subgroup of the group $\{\gamma\in \Gamma:w\psi(\gamma)w^{-1}=\psi(\gamma)=\overline{\psi}(\gamma)\}$ and by our assumption in (i) and (ii), the representation $\psi:\Gamma \rightarrow \mathsf{SL}_{d(2d+1)}(\mathbb{C})$ and the separable subgroup $\M<\Gamma$ satisfy the assumptions of Proposition \ref{rho1-rho2-w}. Therefore, there exists a finite-index subgroup $\Gamma_2<\Gamma$, containing $\M$, such that the double $\Gamma_2\ast_{\M}\Gamma_2$ admits a $2$-Anosov representation into $\mathsf{SL}_r(\mathbb{C})$ for $r=2\binom{4p(2p-1)}{4}\left(2\binom{4p(2p-1)}{4}+1\right)$, where $p=d(2d+1)$.\end{proof}

\subsection{Proof of Theorem \ref{amalgam-rank1}}\label{sec:4.2}
  Let us recall that $\mathsf{G}$ is a rank $1$ Lie group and  $(\mathsf{G},\mathsf{L})$ is a compatible pair, such that the following conditions hold:\begin{enumerate}[label=(\roman*)]
    \item $\mathsf{L}=\big\{g\in \mathsf{G}:\varphi(g)=g\big\}$;
    \item\label{compatible-2} $\Lambda_{\mathsf{L}}=\big\{x\in \partial_{\infty}(\mathsf{G}/\mathsf{K}):\partial\varphi(x)=x\big\}$,
\end{enumerate} where $\varphi:\mathsf{G}\rightarrow \mathsf{G}$ is a Lie automorphism of $\mathsf{G}$.

Let $\tau:\mathsf{G}\rightarrow \mathsf{SL}_d(\mathbb{R})$ be a faithful $1$-proximal representation of minimal dimension $d\coloneqq d(\mathsf{G})$ provided by Proposition \ref{theta-compatible} such that $\tau(\mathsf{G})\cap\{\pm \textup{I}_d\}=\{\textup{I}_d\}$. The representations $\tau|_{\Gamma}:\Gamma \rightarrow \mathsf{SL}_d(\mathbb{R})$ and  $(\tau\circ \varphi)|_{\Gamma}:\Gamma \rightarrow \mathsf{SL}_d(\mathbb{R})$ are $1$-Anosov (e.g. see \cite[Prop. 4.4]{Guichard-Wienhard}). In addition, there exist $\tau$-equivariant continuous embeddings $$ \iota_{\tau}^{+}:\partial_{\infty}(\mathsf{G}/\mathsf{K})\rightarrow \mathbb{P}(\mathbb{R}^d), \ \iota_{\tau}^{-}:\partial_{\infty}(\mathsf{G}/\mathsf{K})\rightarrow \mathsf{Gr}_{d-1}(\mathbb{R}^d),$$ which are antipodal (i.e. $\iota_{\tau}^{+}(x)\notin \iota_{\tau}^{-}(y)$ for $x\neq y$) such that the Anosov limit maps of $\tau|_{\Gamma}$ and $(\tau\circ \varphi)|_{\Gamma}$ respectively are \begin{align*} \begin{array}{lll}
        \xi_{\tau}^1(x)=\iota_{\tau}^{+}(x),  && \xi_{\tau}^{d-1}(x)=\iota_{\tau}^{-}(x), \\
        \xi_{\tau\circ \varphi}^1(x)=\iota_{\tau}^{+}(\partial\varphi(x)), && \xi_{\tau\circ \varphi}^{d-1}(x)=\iota_{\tau}^{-}(\partial\varphi(x)).
     \end{array}\end{align*}

\noindent \ref{amalgam-rank1-1}  Suppose that $\Gamma \cap \mathsf{L}<\Gamma$ is quasiconvex and $\Lambda_{\Gamma \cap \mathsf{L}}=\Lambda_{\Gamma}\cap \Lambda_{\mathsf{L}}$. In addition, as $\mathsf{L}<\mathsf{G}$ is Zariski closed, $\Gamma \cap \mathsf{L}<\Gamma$ is separable (see  \cite{bergeron}). Note that if either $\xi_{\tau}^1(x)\in \xi_{\tau\circ \varphi}^{d-1}(x)$ or $ \xi_{\tau\circ \varphi}^{1}(x)\in \xi_{\tau}^{d-1}(x)$, then, since $\iota_{\tau}^{+}$ and $\iota_{\tau}^{-}$ are antipodal, $\partial\varphi(x)=x$ and by \ref{compatible-2} $x\in \Lambda_{\mathsf{L}}\cap \Lambda_{\Gamma}$ and $x\in \Lambda_{\Gamma\cap \mathsf{L}}$. Thus, the representations $\rho_1\coloneqq\tau|_{\Gamma}$ and $\rho_2\coloneqq(\tau\circ \varphi)|_{\Gamma}$ and the subgroup $\Gamma \cap \mathsf{L}<\Gamma$, satisfy the assumptions of Proposition \ref{rho1-rho2} and the conclusion follows.

\medskip

\noindent \ref{amalgam-rank-2} Let $(\mathsf{G},\mathsf{L}')$ be another compatible pair given by the automorphism $\psi:\mathsf{G}\rightarrow \mathsf{G}$ of order equal to $2q$, $q\in \mathbb{N}$, such that $\mathsf{L}'\coloneqq\big\{g\in \mathsf{G}:\psi(g)=g\big\}$ and $\Lambda_{\mathsf{L}'}=\big\{x\in \partial_{\infty}(\mathsf{G}/\mathsf{K}):\partial\psi(x)=x\big\}$.  

Let $\Gamma<\mathsf{G}$ be a convex cocompact subgroup such that $$\M\coloneqq\big\{\gamma\in \Gamma:\varphi(\gamma)=\psi(\gamma)=\gamma \big\}$$ is a quasiconvex subgroup of $\Gamma$ with $\Lambda_{\Gamma}\cap \Lambda_{\mathsf{L}}\cap \Lambda_{\mathsf{L}'}=\Lambda_{\M}$. Again, as $\mathsf{L},\mathsf{L}'<\mathsf{G}$ are Zariski closed, $\M<\Gamma$ is separable.

Recall the definition of $\tau$ above and consider the faithful representation $\tau':\mathsf{G}\rightarrow \mathsf{GL}\big(\bigwedge^{2q}\mathbb{R}^{2qd}\big)$ defined as follows:
\begin{equation}\label{wedge-tau} \tau'(\gamma)=\bigwedge^{2q}\begin{pmatrix}
\gamma &  &  &   \\
   &\psi(\gamma)  & &     \\
 &  &  \ddots & \\
 &  &  &\psi^{2q-1}(\gamma)
\end{pmatrix}.\end{equation}  
Associated to $\tau'$ there is a pair of continuous, $\tau'$-equivariant antipodal limit maps $(\iota_{\tau'}^{+},\iota_{\tau'}^{-}):\partial_{\infty}\mathsf{G}/\mathsf{K}\rightarrow \mathbb{P}\big(\bigwedge^{2q}\mathbb{R}^{2qd}\big)\times \mathsf{Gr}_{c-1}\big(\bigwedge^{2q}\mathbb{R}^{2qd}\big)$, where $c\coloneqq\binom{2qd}{2q}$, obtained from $(\iota_{\tau}^{+},\iota_{\tau}^{-})$ as follows: for a point $x\in \partial_{\infty}\mathsf{G}/\mathsf{K}$ write, $$\iota_{\tau}^{+}(\partial\psi^{\ell}(x))=\big[u_{\partial\psi^{\ell}(x)}\big], \quad \iota_{\tau}^{-}(\partial\psi^{\ell}(x))=\textup{span}\big\{v_{\partial\psi^{\ell}(x)}\big\}^{\perp}$$ for $\ell=0,\ldots, 2c-1$, and let 
\begin{align}\label{limitmaps-wedge}\iota_{\tau'}^{+}(x)&\coloneqq\textup{span}\left\{\begin{pmatrix}
u_x \\
0 \\
 \vdots \\
0\end{pmatrix} \wedge \begin{pmatrix}
0 \\
u_{\partial\varphi(x)} \\
 \vdots \\
0\end{pmatrix}\wedge\cdots\wedge\begin{pmatrix}
0 \\
0\\
 \vdots \\
u_{\partial\varphi^{2q-1}(x)}\end{pmatrix}\right\},\\
\label{limitmaps-wedge-2}\iota_{\tau'}^{-}(x)&\coloneqq\textup{span}\left \{\begin{pmatrix}
v_x \\
0 \\
 \vdots \\
0\end{pmatrix} \wedge \begin{pmatrix}
0 \\
v_{\partial\varphi(x)} \\
 \vdots \\
0\end{pmatrix}\wedge\cdots\wedge\begin{pmatrix}
0 \\
0\\
 \vdots \\
v_{\partial\varphi^{2q-1}(x)}\end{pmatrix}\right\}^{\perp}.\end{align} 
The representation $\tau'|_{\Gamma}$ (resp. $(\tau'\circ \varphi)|_{\Gamma}$) is $1$-Anosov and its limit maps are the restrictions of $(\iota_{\tau'}^{+},\iota_{\tau'}^{-})$ (resp. $(\iota_{\tau'}^{+}\circ \partial \varphi,\iota_{\tau'}^{-}\circ \partial \varphi)$)  on $\Lambda_{\Gamma}=\partial_{\infty}\Gamma$.

Let us define the order 2 element $w_0\in \mathsf{SL}(\bigwedge^{2q}\mathbb{R}^{2qd})$,
\begin{align*} w_0\coloneqq\bigwedge^{2q}\begin{pmatrix}W_0 & &\\
 &  \ddots  &\\ & & W_0
\end{pmatrix}, \quad \text{where }W_0\coloneqq\begin{pmatrix}
\textup{O}_d & \textup{I}_d  \\
\textup{I}_d & \textup{O}_d\end{pmatrix}.\end{align*}

As $\tau'$ is faithful and $w_0\tau'(\gamma)w_0^{-1}=\bigwedge^{2q}\textup{diag}\big(\psi(\gamma),\gamma, \ldots, \psi^{2q-1}(\gamma),\psi^{2q-2}(\gamma)\big)$, $\gamma \in \mathsf{G}$, we clearly have that $$\M\coloneqq\big\{\gamma \in \Gamma: \tau'(\gamma)=\tau'(\varphi(\gamma))=w_0\tau'(\gamma)w_0^{-1} \big\}.$$

We will need the following claim.

\begin{nclaim} For a point $x\in \Lambda_\Gamma$, we have $x\in \Lambda_\M$ if and only if the following two conditions are satisfied simultaneously:
\begin{enumerate}
\item  $\iota_{\tau'}^{+}(\partial \varphi(x))\in \iota_{\tau'}^{-}(x)$ or $\iota_{\tau'}^{+}(x)\in \iota_{\tau'}^{-}(\partial\varphi(x))$;

\item $w\iota_{\tau'}^{+}(x)\in \iota_{\tau'}^{-}(x)$ or $w\iota_{\tau'}^{+}(\partial\varphi(x))\in \iota_{\tau'}^{-}(\partial\varphi(x))$.
\end{enumerate}
\end{nclaim}

\begin{proof} For every point $y\in \partial_{\infty}\M$ we have that $\iota_{\tau'}^{\pm }(y)=\iota_{\tau'}^{\pm}(\partial\psi(y))=\iota_{\tau'}^{\pm}(\partial\varphi(y))$.

\par Now suppose that $x\in \partial_{\infty}\Gamma$ and (1) and (2) are satisfied simultaneously. 

By (1) and antipodality, we have that $\partial\varphi(x)=x$. In addition, by (2), since $\partial\varphi(x)=x$, we have that $w\iota_{\tau'}^{+}(x)\in \iota_{\tau'}^{-}(x)$. Observe that if we write $\iota^{+}_{\tau}(\partial\varphi^{\ell}(x))=[u_{\partial\varphi^{\ell}(x)}]$, $\iota_{\tau}^{-}(\partial\varphi^{\ell}(x))=\textup{span}(v_{\partial\varphi^{\ell}(x)})^{\perp}$, then by the definition of the limit maps $(\iota_{\tau'}^{+},\iota_{\tau'}^{-})$ in (\ref{limitmaps-wedge}) we have $$w\iota_{\tau'}^{+}(x)=\textup{span}\left\{\begin{pmatrix}
0 \\
u_x \\
 \vdots \\
0\end{pmatrix} \wedge \begin{pmatrix}
u_{\partial\varphi(x)} \\
0 \\
 \vdots \\
0\end{pmatrix}\wedge\cdots \wedge \begin{pmatrix}
 
0 \\
 \vdots \\
 0\\
u_{\partial\psi^{2q-2}(x)}\end{pmatrix}\wedge \begin{pmatrix}
0\\
 \vdots \\
u_{\partial\varphi^{2q-1}(x)} \\ 0\end{pmatrix}\right\}.$$

 The condition $w\iota_{\tau'}^{+}(x)\in \iota_{\tau'}^{-}(x)$, combined with the formula for $\xi_{\tau'}^{p-1}(x)$ from (\ref{limitmaps-wedge-2}) and the definition of the inner product\footnote{Note that $\big\langle a_1\wedge \cdots a_{2q},b_1\wedge\cdots \wedge b_{2q}\big\rangle=\textup{det}\big(\langle a_{r},b_{s}\rangle\big)_{r,s=1}^{2q}$, where $a_{1},b_1,\ldots, a_{2q},b_{2q}\in \mathbb{R}^{2qd}.$} in $\bigwedge^{2p}\mathbb{R}^{2pd}$, implies that $$ \textup{det}\begin{pmatrix}
0 & \langle u_{x},v_{\partial\psi(x)}\rangle  &   &   &  \\
\langle u_{x},v_{\partial\psi(x)}\rangle & 0 &  &  &  \\
 &  & \ddots  &  &  \\
 &  &  &  0& \langle u_{\partial \psi^{2q-2}(x)},v_{\partial \psi^{2q-1}(x)}\rangle  \\
 &  &  &  \langle u_{\partial \psi^{2q-1}(x)},v_{\partial \psi^{2q-2}(x)}\rangle& 0
\end{pmatrix}=0,$$ or equivalently, $$ \big(\langle v_{x},u_{\partial\psi(x)}\rangle  \langle u_{x},v_{\partial\psi(x)}\rangle\big)\cdots \big(\langle v_{\partial \psi^{2q-1}(x)},u_{\partial\psi^{2q-2}(x)}\rangle  \langle u_{\partial\psi^{2q-2}(x)},v_{\partial\psi^{2q-1}(x)}\rangle \big)=0.$$ Therefore, there exists $\ell\in \{1,\ldots, 2q-1\}$ such that $\iota_{\tau}^{+}(\partial \psi^{\ell-1}(x))\in \iota_{\tau}^{-}(\partial \psi^{\ell}(x))$ or $\iota_{\tau}^{+}(\partial \psi^{\ell}(x))\in \iota_{\tau}^{-}(\partial \psi^{\ell-1}(x))$. In any case, since $(\iota_{\tau}^{+},\iota_{\tau}^{-})$ are antipodal, we have $\partial\psi^{\ell-1}(x)=\partial\psi^{\ell}(x)$, or equivalently $\partial\psi(x)=x$. 

 Thus, conditions (1) and (2) imply at the same time that $x\in \Lambda_{\Gamma}\cap \Lambda_{\mathsf{L}}\cap \Lambda_{\mathsf{L}'}=\Lambda_{\M}$. The claim follows.\end{proof}

Using the previous claim, it follows that the group $\Gamma<\mathsf{G}$, its quasiconvex subgroup $\M=\Gamma \cap \mathsf{L}\cap \mathsf{L}'$, and the faithful $1$-Anosov representations $\rho_1\coloneqq\tau'$, $\rho_2'=\tau'\circ \varphi$ and the order $2$ element $w_0\in \mathsf{SL}(\bigwedge^{2q}\mathbb{R}^{2qd})$ satisfy the conditions of Proposition \ref{rho1-rho2-w}. Thus, there exists a finite index subgroup $\Gamma_2<\Gamma$, containing $\M$, such that the double $\Gamma_2\ast_{\M}\Gamma_2$ admits a $2$-Anosov representation into $\mathsf{SL}_r(\mathbb{C})$ where $r=2\binom{4p(2p-1)}{4}\left(2\binom{4p(2p-1)}{4}+1\right)$, where $p=\binom{2qd}{2q}(2\binom{2qd}{2q}+1)$ and $d=d(\mathsf{G})$.
\qed

\subsection{Amalgams along dividing subgroups}
Another application of the main virtual amalgamation theorem using Lemma \ref{antipodal-1} is the following:

\begin{proposition}
    \label{codim1} Let $\Gamma<\mathsf{SL}_n(\mathbb{R})$ be a $1$-Anosov subgroup and $\xi^1:\partial_{\infty}\Gamma \rightarrow \mathbb{P}(\mathbb{R}^n)$, $\xi^1_{\ast}:\partial_{\infty}\Gamma \rightarrow \mathbb{P}(\mathbb{R}^n)$ the Anosov limit maps of $\Gamma<\mathsf{SL}_n(\mathbb{R})$ and the dual group $\Gamma^{\ast}=\{\gamma^t:\gamma \in \Gamma\}$ respectively. Suppose that $\H\coloneqq\Gamma \cap \big(\mathsf{GL}_1(\mathbb{R})\times \mathsf{GL}_{n-1}(\mathbb{R})\big)$ is a quasiconvex non-elementary subgroup of $\Gamma$ satisfying the following conditions: \begin{align*}\xi^1(\partial_{\infty}\H)&=\xi^1(\partial_{\infty}\Gamma)\cap \mathbb{P}\left(\{0\}\times \mathbb{R}^{n-1}\right),\\ \xi_{\ast}^1(\partial_{\infty}\H)&=\xi_{\ast}^1(\partial_{\infty}\Gamma)\cap \mathbb{P}\left(\{0\}\times \mathbb{R}^{n-1}\right).\end{align*} Then there exists a finite-index subgroup $\Gamma_1<\Gamma$, containing $\H$, such that the amalgam $\Gamma_1\ast_{\H}\Gamma_1$ admits a $2$-Anosov representation into $\mathsf{SL}_{2n}(\mathbb{C})$.
\end{proposition}

\begin{proof} Let $\xi^{n-1}:\partial_{\infty}\Gamma \rightarrow \mathsf{Gr}_{n-1}(\mathbb{R}^n)$ be the limit map of $\Gamma$ such that $\xi_{\ast}^{1}(x)=\xi^{n-1}(x)^{\perp}$ for $x\in \partial_{\infty}\Gamma$.

Consider the matrix $\alpha=\begin{pmatrix} i & \\ & \textup{I}_{n-1}\end{pmatrix}$ and the $1$-Anosov subgroup $\Gamma_0\coloneqq\alpha \Gamma \alpha ^{-1}$ of $\mathsf{SL}_{n}(\mathbb{C})$, whose limit maps are $z\mapsto \left(\alpha \xi^1(z),\alpha \xi^{n-1}(z)\right)$. Note that $\Gamma_0(\mathbb{R})=\H$ and by assumption, $$\alpha \xi^1(\partial_{\infty}\Gamma)\cap \mathbb{P}(\mathbb{R}^n)=\xi^1(\partial_{\infty}\Gamma)\cap \mathbb{P}(\{0\}\times \mathbb{R}^{n-1})=\xi^1(\partial_{\infty}\Gamma_0(\mathbb{R})).$$ Moreover, if for some $x\in \partial_{\infty}\Gamma$, $\alpha \xi^1(x)\in \alpha \xi^{n-1}(x)\cap \overline{\alpha}\xi^{n-1}(x)$, then either $\xi^1(x)\in  \mathbb{P}\left(\{0\}\times \mathbb{R}^{n-1}\right)$ or $\xi_{\ast}^1(x)\in \mathbb{P}\left(\{0\}\times \mathbb{R}^{n-1}\right)$, hence $x\in \partial_{\infty}\H$. Therefore, as $\H<\Gamma$ is additionally separable, by Corollary \ref{complex-1} the conclusion follows.\end{proof}

Now using \Cref{codim1}, we prove the following statement providing further examples of Anosov groups obtained by amalgamations along subgroups dividing a domain in one dimension lower.

\begin{theorem}\label{divisible-amalgam}Let $\Gamma$ be a $1$-Anosov subgroup of $\mathsf{SL}_n(\mathbb{R})$, $n\geq 4$, such that the subgroup $\H\coloneqq\Gamma \cap \left(\mathsf{GL}_1(\mathbb{R})\times \mathsf{GL}_{n-1}(\mathbb{R})\right)$ of $\Gamma$ preserves and acts cocompactly on a properly convex domain of $\mathbb{P}\left(\{0\}\times \mathbb{R}^{n-1}\right)$. Then there exists a finite-index subgroup $\Gamma_1$ of $\Gamma$, containing $\H$, such that $\Gamma_1\ast_{\H}\Gamma_1$ admits a $2$-Anosov representation into $\mathsf{SL}_{2n}(\mathbb{C})$.\end{theorem}

\begin{proof} Let $(\xi^1,\xi^{n-1})$ (resp. ($\xi^{1}_{\ast},\xi^{n-1}_{\ast}$)) be the Anosov limit maps of $\Gamma<\mathsf{SL}_n(\mathbb{R})$ (resp. $\Gamma^{\ast}<\mathsf{SL}_n(\mathbb{R})$). By assumption, there is a properly convex domain $\Omega\subset \mathbb{P}(\{0\}\times \mathbb{R}^{n-1})$ acted upon cocompactly by $\H=\Gamma \cap \left(\mathsf{GL}_1(\mathbb{R})\times\mathsf{GL}_{n-1}(\mathbb{R})\right)$. In particular, $\H<\mathsf{GL}_1(\mathbb{R})\times \mathsf{GL}_{n-1}(\mathbb{R})$ is Anosov (hence quasiconvex in $\Gamma$) and its limit maps are $y\mapsto (\xi^1(y),\xi^{n-1}(y)\cap \mathbb{P}(\{0\}\times \mathbb{R}^{n-1}))$.

\par Now suppose that there is $x\in \partial_{\infty}\Gamma$ such that $\xi^1(x)\in \mathbb{P}(\{0\}\times \mathbb{R}^{n-1})$. If $\xi^1(x)\notin \Omega$, since $$\mathbb{P}(\{0\}\times \mathbb{R}^{n-1})\smallsetminus \Omega=\bigcup_{y \in \partial_{\infty}\H}\mathbb{P}\left (\xi^{n-1}(y)\cap \left(\{0\}\times \mathbb{R}^{n-1}\right)\right),$$  there is $z\in \partial_{\infty}\H$ such that $\xi^1(x)\in \xi^{n-1}(z)$. Since $(\xi^1,\xi^{n-1})$ are antipodal we have $x=z$ and $x\in \partial_{\infty}\H$. In the case where $\xi^1(x)\in \Omega$, the projective $(n-2)$-hyperplane $\xi^{(n-1)}(x)\cap \mathbb{P}(\{0\}\times \mathbb{R}^{n-1})$ has to intersect $\partial\Omega=\xi^1(\partial_{\infty}\H)$ and hence there is $z'\in \partial_{\infty}\H$ with $\xi^1(z')\in \xi^{n-1}(x)$. Thus, in this case we also have $x=z'$ and $x\in \partial_{\infty}\H$.

We conclude that $\xi^1(\partial_{\infty}\H)=\xi^1(\partial_{\infty}\Gamma)\cap \mathbb{P}\left(\{0\}\times \mathbb{R}^{n-1}\right)$. Similarly, since the dual group $\H^{\ast}\coloneqq\Gamma^{\ast}\cap \left(\mathsf{GL}_1(\mathbb{R})\times \mathsf{GL}_{n-1}(\mathbb{R})\right)$ also acts cocompactly on a properly convex domain $\Omega'$ of $\mathbb{P}(\{0\}\times \mathbb{R}^{n-1})$, the same argument as in the previous paragraph shows that $\xi^1_{\ast}(\partial_{\infty}\H)=\xi^1_{\ast}(\partial_{\infty}\Gamma)\cap \mathbb{P}\left(\{0\}\times \mathbb{R}^{n-1}\right)$. Therefore, the conclusion of the theorem follows by \Cref{codim1}.\end{proof}

\subsection{Amalgams along fixed-point subgroups of automorphisms}

Given a subset of positive restricted roots $\Theta$ of $\mathsf{G}$, recall that $d=d(\mathsf{G},\Theta)$ denotes the minimal dimension of a faithful $\Theta$-proximal representation $\tau_{\Theta}:\mathsf{G}\rightarrow \mathsf{GL}_d(\mathbb{R})$ (see Proposition \ref{theta-compatible} for existence).

\begin{theorem}\label{amalgam-autom} Let $\Gamma$ be a $\Theta$-Anosov subgroup $\mathsf{G}$ and let $\psi:\Gamma \rightarrow \Gamma$ is an automorphism such that the limit set of the fixed-point subgroup $\Gamma^{\psi}$ in $\partial_{\infty}\Gamma$ is given by $\big\{x\in \partial_{\infty}\Gamma:\ \partial \psi(x)=x\big\}$ (e.g. $\psi$ is of finite order, see \Cref{limitset-autom} below). 
Then there exists a finite-index subgroup $\Gamma_1$ of $\Gamma$, containing $\Gamma^{\psi}$, such that the amalgam $\Gamma_1\ast_{\Gamma^{\psi}}\Gamma_1$ admits a $2$-Anosov representation into $\mathsf{SL}_m(\mathbb{C})$, where $m=2d(\mathsf{G},\Theta)(2d(\mathsf{G},\Theta)+1)$.\end{theorem}

\begin{proof} By \cite{Neumann}, the fixed-point subgroup $\Gamma^{\psi}$  is quasiconvex in $\Gamma$. Since $\Gamma$ is linear (in particular, residually finite), by \cite{Long-Niblo}, $\Gamma^{\psi}$ is separable in $\Gamma$.

Now we follow the argument in the proof of Theorem \ref{amalgam-rank1}. 
First, we compose the embedding $\Gamma \xhookrightarrow{} \mathsf{G}$ with a faithful linear representation $\tau_{\Theta}:\mathsf{G} \rightarrow \mathsf{SL}_{d}(\mathbb{R})$ of minimal dimension $r=d(\mathsf{G},\Theta)$ such that $\tau_{\Theta}(\Gamma)<\mathsf{SL}_{r}(\mathbb{R})$ is $1$-Anosov. 
To this end, by identifying $\Gamma$ with its image under $\tau_{\Theta}$, we may assume that $\Gamma<\mathsf{SL}_r(\mathbb{R})$ is $1$-Anosov. 
The representation $\psi:\Gamma \rightarrow \mathsf{SL}_d(\mathbb{R})$ is also $1$-Anosov with Anosov limit maps $(\xi^{1}\circ \partial \psi, \xi^{d-1}\circ \partial\psi)$, where $(\xi^1,\xi^{d-1})$ are the Anosov limit maps of $\Gamma<\mathsf{SL}_d(\mathbb{R})$. For a point $x\in \partial_{\infty}\Gamma$, if either $\xi^1(\partial\psi (x))\in \xi^{d-1}(x)$ or $\xi^1(\partial\psi (x))\in \xi^{d-1}(x)$, then $\partial\psi(x)=x$ and hence $x\in \partial_{\infty}\Gamma^{\psi}$. 
Thus, applying Proposition \ref{rho1-rho2} to the representations $\rho_1,\rho_2 : \Gamma \to \mathsf{SL}_d(\mathbb{R})$, where $\rho_1$ is the inclusion map $\iota_\Gamma$ and $\rho_2=\varphi\circ\iota_\Gamma$, and the fixed-point subgroup $\H\coloneqq\Gamma^{\psi}$, we arrive at the conclusion. \end{proof}

The final result of this section verifies that the hypothesis of \Cref{amalgam-autom} is satisfied when $\psi$ is of finite order:

\begin{lemma}\label{limitset-autom} 
Let $\Gamma$ be a hyperbolic group.  For any finite order automorphism $\psi:\Gamma \rightarrow \Gamma$, 
$\partial_{\infty}\Gamma^{\psi}=\big\{x\in \partial_{\infty}\Gamma:\ \partial\psi(x)=x\big\}$.
\end{lemma}

\begin{proof} 
Fixing a word metric $d_{\Gamma}$ on $\Gamma$, consider the Cayley graph $C_{\Gamma}$ of $\Gamma$. 
Note that since $\psi$ is of finite order, say of order $p\in \mathbb{N}$, we can define a $\psi$-invariant metric $d_{\psi}:C_{\Gamma}\times C_{\Gamma}\rightarrow \mathbb{R}^{+}$ by
\[
d_{\psi}(g,h)=\frac{1}{p}\sum_{\ell=0}^{p-1}d_{\Gamma}\big(\psi^{\ell}(g),\psi^{\ell}(h)\big),
\quad g,h\in \G.
\]
Clearly, $d_\psi$ is $\Gamma$-invariant and hence is bi-Lipschitz equivalent to the word metric $d_{\Gamma}$. Consequently, $(\Gamma,d_{\psi})$ is also hyperbolic.

Let $x\in \partial_{\infty}\Gamma$ be a point such that $\partial\psi(x)=x$.
Choose a geodesic ray $(\gamma_n)\subset\Gamma$ starting at the identity $1\in \Gamma$ such that $x=[(\gamma_n)]$. 
Then both $(\g_n)$ and $(\psi(\g_n))$ are geodesic rays starting at $1\in \Gamma$ and with the same endpoint $x\in \partial_{\infty}\Gamma$.
Hence $\sup_{n\in \mathbb{N}}d_{\psi}\left(\gamma_n,\psi(\gamma_n)\right)<\infty$. 

Consider the sequence $(\g_n^{-1}\psi(\g_n))$. By above, the entries of this seqeunce form a finite set in $\Gamma$. Thus, there exists $k_0\in\N$ such that for all $n\in\N$,
\[
 \g_n^{-1}\psi(\g_n) = \g_{k_n}^{-1}\psi(\g_{k_n})
 \quad\text{for some }k_n\in\{1,\dots,k_0\}.
\]
From this, we can see that $\psi(\g_n\g_{k_n}^{-1}) = \g_n\g_{k_n}^{-1}\in\G^{\psi}$.
We deduce that the geodesic sequence $(\gamma_n)$ lies in a finite neighborhood of $\Gamma^\psi$. Hence, $x \in \geo\Gamma^\psi$ and, therefore,  
\[
\big\{x \in \partial_{\infty}\Gamma :\ \partial\psi(x) = x\big\} \subset \partial_{\infty}\Gamma^{\psi}.
\]  

The reverse inclusion follows from the definition of $\Gamma^{\psi}$, which concludes the proof.
\end{proof}

\section{Doubles of convex cocompact groups}\label{sec:5}
\Cref{main-rank1-amalgam} follows from the following, which is the main result of this section:

\begin{theorem}\label{complex-quat-amalgam} 
Let $(\mathsf{G},\mathsf{L})$ be one of the following pairs of Lie groups: 
\begin{enumerate} 
\item \label{k-sametype} $(\mathsf{G},\mathsf{L})=(\mathsf{O}(n,1),\mathsf{O}(k)\times \mathsf{O}(n-k,1))$, $(\mathsf{U}(n,1),\mathsf{U}(k)\times \mathsf{U}(n-k,1))$, $(\mathsf{Sp}(n,1),\mathsf{Sp}(k)\times \mathsf{Sp}(n-k,1))$; 
\item \label{n-real} $(\mathsf{G},\mathsf{L})=(\mathsf{U}(n,1),\mathsf{O}(n,1))$;
\item \label{n-complex} $(\mathsf{G},\mathsf{L})=(\mathsf{Sp}(n,1),\mathsf{U}(n,1))$;
\item \label{k-complex} $(\mathsf{G},\mathsf{L})=(\mathsf{Sp}(n,1),\mathsf{Sp}(k)\times \mathsf{U}(n-k,1))$;
\item \label{k-real-complex} $(\mathsf{G},\mathsf{L})=(\mathsf{U}(n,1),\mathsf{U}(k)\times \mathsf{O}(n-k,1))$;
\item\label{max-real} $(\mathsf{G},\mathsf{L})=(\mathsf{Sp}(n,1),\mathsf{O}(n,1))$;
\item \label{k-real-real} $(\mathsf{G},\mathsf{L})=(\mathsf{Sp}(n,1),\mathsf{Sp}(k)\times \mathsf{O}(n-k,1))$;
\end{enumerate}where $n\geq 2$ and $k=1,\ldots, n-1$. 
If $\Gamma$ is a convex cocompact subgroup of $\mathsf{G}$ such that $\Gamma\cap \mathsf{L}$ is a lattice in $\mathsf{L}$, then there exists a finite-index subgroup $\Gamma_1$ of $\Gamma$, containing $\Gamma\cap \mathsf{L}$, such that the double $\Gamma_1\ast_{\Gamma\cap \mathsf{L}}\Gamma_1$ admits a faithful Anosov representation into $\mathsf{SL}_d(\mathbb{R})$ for some $d\in \mathbb{N}$ depending only on $n\in \mathbb{N}$.
\end{theorem}

\begin{remark}\label{rem:complex-quat-amalgam} Let $\mathbb{K}=\mathbb{R}$ or $\mathbb{C}$. If $\Gamma<\mathsf{Sp}(n,1)$, $n\geq 4$, is a uniform lattice and $\Gamma(\mathbb{K})\coloneqq\Gamma \cap \mathsf{GL}_d(\mathbb{K})$ is also a uniform lattice in its Zariski closure in $\mathsf{Sp}(n,1)$ (see the example below), the amalgam $\Gamma\ast _{\Gamma(\mathbb{K})}\Gamma$ fails to admit discrete and faithful representation into any Lie group of real rank $1$ (see Proposition \ref{indiscrete-rank1}). However, Corollary \ref{complex-quat-amalgam} shows that for many uniform lattices $\Gamma<\mathsf{Sp}(n,1)$, the double $\Gamma \ast_{\Gamma(\mathbb{K})}\Gamma$ admits an Anosov representation to some higher rank Lie group. These provide further examples of Anosov subgroups, with connected Gromov boundary, not admitting discrete and faithful embedding into any rank $1$ Lie group. 
\end{remark}

\begin{example}  Let $A_n:=\begin{pmatrix}[0.5]\textup{I}_n & \\ & \sqrt[4]{2}\end{pmatrix}$. The discrete subgroup $$\Gamma=\Big\{g\in \mathsf{Sp}(n,1): g=A_{n}^{-1}hA_n, h\in \mathsf{GL}_{n+1}\big(\mathbb{Z}[\sqrt{2},i,j,k]\big) \Big\}$$ of $\mathsf{Sp}(n,1)$ is a uniform lattice. In addition, for every pair $(\mathsf{G},\mathsf{L})$ in Theorem \ref{complex-quat-amalgam}, where $\mathsf{G}=\mathsf{Sp}(n,1)$, $\Gamma\cap \mathsf{L}$ is a uniform lattice in $\mathsf{L}$.\end{example}

\subsection{Preliminary notation}\label{proof:main-rank1-amalgam}
We set up some notation.
For $z\in \mathbb{H}$, denote by $\overline{z}$ the conjugate of $z$.
For $g = (g_{ij})_{i,j=1}^{m} \in \mathsf{Mat}_{m}(\mathbb{H})$, $g^{\ast}=(\overline{g_{ji}})_{i,j=1}^{m}$ denotes the conjugate transpose of $g$. 

For $n\geq 1$, consider the Lie groups 
\begin{align*}
\mathsf{O}(n,1)&=\big\{g\in \mathsf{GL}_{n+1}(\mathbb{R}):g^{t}J_{n,1}g=J_{n,1}\big\},\\ 
\mathsf{U}(n,1)&=\big\{g\in \mathsf{GL}_{n+1}(\mathbb{C}):g^{\ast}J_{n,1}g=J_{n,1}\big\},\\
\mathsf{Sp}(n,1)&=\big\{g\in \mathsf{GL}_{n+1}(\mathbb{H}):g^{\ast}J_{n,1}g=J_{n,1}\big\},
\end{align*}
where $J_{n,1}\coloneqq \begin{pmatrix}[0.8]\textup{I}_n & \\ & -1\end{pmatrix}$.

Let $\mathbb{K}=\mathbb{R},\mathbb{C}$ or $\mathbb{H}$. Recall that the projective space $\mathbb{P}(\mathbb{K}^{n+1})$ is the set of equivalence classes of vectors in $\mathbb{K}^{n+1}\smallsetminus\{0_{n+1}\}$, where $u,v\in \mathbb{K}^{n+1}$ are equivalent if there is $z\in \mathbb{K}$ such that $u=vz$. The $\mathbb{K}$-hyperbolic $n$-space is the open subset of $\mathbb{P}(\mathbb{K}^{n+1})$ defined by $$\mathbb{H}_{\mathbb{K}}^n\coloneqq\big\{[x_0:\cdots:x_{n-1}:x_{n}]\in \mathbb{P}(\mathbb{K}^{n+1}):\ |x_0|^2+\cdots+|x_{n-1}|^2<|x_n|^2\big\}.$$
Its visual boundary is $$\partial_{\infty}\mathbb{H}_{\mathbb{K}}^n\coloneqq \big\{[x_0:\cdots:x_{n-1}:1]\in \mathbb{P}(\mathbb{K}^{n+1}):\ |x_0|^2+\cdots+|x_{n-1}|^2=1\big\}.$$ The isometry groups of the hyperbolic spaces $\mathbb{H}_{\mathbb{R}}^n$, $\mathbb{H}_{\mathbb{C}}^n$ and $\mathbb{H}_{\mathbb{H}}^n$ are isogenous to $\mathsf{O}(n,1),\mathsf{U}(n,1)$ and $\mathsf{Sp}(n,1)$ respectively.

\subsection{Proof \Cref{complex-quat-amalgam}}\label{proof:complex-quat-amalgam}
The statement follows by applying \Cref{amalgam-rank1} (1) and observing that in cases \ref{k-sametype} through \ref{k-real-complex} we have a compatible pair $(\mathsf{G},\mathsf{L})$, where $\mathsf{L}$ is the fixed-point subgroup of the finite-order automorphism $\varphi:\mathsf{G}\rightarrow \mathsf{G}$:

\smallskip
\noindent \ref{k-sametype}: $\varphi(g)=w_kgw_k^{-1}$, $g\in \mathsf{G}$, where $w_k\coloneqq\begin{pmatrix}[0.6] \textup{I}_{k} & \\ & -\textup{I}_{n+1-k}\end{pmatrix}$.

\smallskip
\noindent \ref{n-real}: $\varphi(g)=\overline{g}$, $g\in \mathsf{G}$.

\smallskip\noindent
\ref{n-complex}: $\varphi(g)=-(i\textup{I}_{n+1})g(i\textup{I}_{n+1})$, $g\in \mathsf{G}$.

\smallskip
\noindent \ref{k-complex}: $\varphi(g)=b_k gb_{k}^{-1}$, $g\in \mathsf{G}$, where $b_k\coloneqq\begin{pmatrix}[0.6] \textup{I}_{k} & \\ & i\textup{I}_{n+1-k}\end{pmatrix}$.

\smallskip
\noindent \ref{k-real-complex}: For this case, we shall consider $\Gamma<\mathsf{U}(n,1)$ as a convex cocompact subgroup of $\mathsf{Sp}(n,1)$ and the Lie subgroup 
$$\mathsf{L}_1\coloneqq\left\{h\in \mathsf{Sp}(n,1):\ a_k h a_k^{-1}=h\right\}, \ a_k\coloneqq\begin{pmatrix}[0.7] \textup{I}_{k} & \\ & j\textup{I}_{n+1-k}\end{pmatrix}.$$ 
Observe that $(\mathsf{Sp}(n,1),\mathsf{L}_1)$ is a compatible pair and the limit set of $\mathsf{L}_1$ is the fixed-point set of $\partial \varphi$ i.e. the set $$\big\{[\{0_k:v_1+v_2j:1]:\ |v_1|^2+|v_2|^2=1,\, v_1,v_2\in \mathbb{R}^{n+1-k}\big\}$$ which intersects $\partial_{\infty}\mathbb{H}_{\mathbb{C}}^n$ to the boundary $\partial_{\infty}\mathbb{H}_{\mathbb{R}}^{n-k}$. Since by assumption $\mathsf{\Gamma}\cap \mathsf{L}_1<\mathsf{U}(k)\times \mathsf{O}(n-k,1)$ is a uniform lattice, we have that $\Lambda_{\Gamma \cap \mathsf{L}_1}=\Lambda_{\mathsf{L}_1}\cap \Lambda_{\Gamma}=\partial_{\infty}\mathbb{H}_{\mathbb{R}}^{n-k}$. Thus, Theorem \ref{amalgam-rank1} (i) applies and we obtain the conclusion.
\smallskip

\noindent \ref{max-real}: Let $\Gamma<\mathsf{Sp}(n,1)$ a convex cocompact subgroup such that $\Gamma\cap \mathsf{O}(n,1)<\mathsf{O}(n,1)$ is a uniform lattice. Consider the compatible pairs  $(\mathsf{Sp}(n,1),\mathsf{L})$ and  $(\mathsf{Sp}(n,1),\mathsf{L}')$, defined by the order $4$ automorphisms $\varphi_1(g)=(-i\textup{I}_{n+1})g(i\textup{I}_{n+1})$ and $\psi_1(g)=(-j\textup{I}_{n+1})g(j\textup{I}_{n+1})$. In other words, $\mathsf{L}=\mathsf{U}(n,1)$ and $\mathsf{L}'=\mathsf{Sp}(n,1)\cap (\mathsf{Mat}_{n+1}(\mathbb{R})\oplus j\mathsf{Mat}_{n+1}(\mathbb{R}))$, $\mathsf{L}\cap \mathsf{L}'=\mathsf{O}(n,1)$ and $\Lambda_{\mathsf{L}}\cap \Lambda_{\mathsf{L}'}=\partial_{\infty}\mathbb{H}_{\mathbb{R}}^n$. Thus, the conclusion follows by applying \Cref{amalgam-rank1} (2) for the compatible pairs $(\mathsf{Sp}(n,1),\mathsf{L})$, $(\mathsf{Sp}(n,1),\mathsf{L}')$, the automorphism $\psi_1$ and the convex cocompact subgroup $\Gamma<\mathsf{Sp}(n,1)$.

\smallskip

\noindent \ref{k-real-real} Let $\Gamma<\mathsf{Sp}(n,1)$ be a convex cocompact subgroup such that $\Gamma \cap (\mathsf{Sp}(k)\times \mathsf{O}(n-k,1))$ is a uniform lattice in $\mathsf{Sp}(k)\times \mathsf{O}(n-k,1)$.
 Consider the compatible pairs  $(\mathsf{Sp}(n,1),\mathsf{L}_1)$ and  $(\mathsf{Sp}(n,1),\mathsf{L}_2)$, defined respectively by the inner, order $4$, automorphisms of $\mathsf{Sp}(n,1)$, $g\mapsto b_kgb_k^{-1}$ and $g\mapsto a_kga_k^{-1}$, where $b_k\coloneqq\begin{pmatrix}[0.6] \textup{I}_{k} & \\ & i\textup{I}_{n+1-k}\end{pmatrix}$ and $a_k\coloneqq\begin{pmatrix}[0.6] \textup{I}_{k} & \\ & j\textup{I}_{n+1-k}\end{pmatrix}$.

 Note that $\mathsf{L}_1\cap \mathsf{L}_2=\mathsf{Sp}(k)\times \mathsf{O}(n-k,1)$ and $\Lambda_{\mathsf{L}}\cap \Lambda_{\mathsf{L}'}=\partial_{\infty}\mathbb{H}_{\mathbb{R}}^{n-k}$. The conclusion follows by applying Theorem \ref{amalgam-rank1} (2) for the compatible pairs $(\mathsf{Sp}(n,1), \mathsf{L}_{1})$, $(\mathsf{Sp}(n,1), \mathsf{L}_{1})$ and the subgroup $\Gamma<\mathsf{Sp}(n,1)$.\qed
\medskip

\begin{remark} In the notation of Corollary \ref{complex-quat-amalgam}, suppose that $(\mathsf{G},\mathsf{L})$ is a compatible pair, $\Gamma<\mathsf{G}$ is convex cocompact, and $\Gamma\cap \mathsf{L}<\mathsf{L}$ is a lattice. We remark that:
\begin{enumerate}[label=(\roman*)]
    \item If $(\mathsf{G},\mathsf{L})=\big(\mathsf{O}(n,1),\mathsf{O}(k)\times \mathsf{O}(n-k,1)\big)$, then the convex cocompact subgroups of $\mathsf{U}(n,1)$ $\Gamma_1\coloneqq\Gamma$, $\Gamma_2\coloneqq b_{k}\Gamma b_{k}^{-1}$, where $b_k\coloneqq\begin{pmatrix}[0.6] \textup{I}_{k} & \\ & i\textup{I}_{n+1-k}\end{pmatrix}$, intersect at $\Gamma \cap \mathsf{L}$ and $b_{k}\Lambda_{\Gamma}\cap \Lambda_{\Gamma}=\Lambda_{\Gamma \cap \mathsf{L}}=\partial_{\infty}\mathbb{H}_{\mathbb{R}}^{n-k}$. Thus, by Theorem \ref{thm:amalgam}, the there exists a finite-index subgroup $\Gamma_1<\Gamma$ such that $\Gamma_1\ast_{\Gamma \cap \mathsf{L}}\Gamma_1$ admits a convex cocompact representation into $\mathsf{U}(n,1)$.

    \item If $(\mathsf{G},\mathsf{L})=(\mathsf{U}(n,1),\mathsf{O}(n,1))$, consider the Lie group embedding $\Phi:\mathsf{U}(n,1)\xhookrightarrow{} \mathsf{Sp}(n,1)$, $\Phi(X+Yi)=X+Yj$, where $X,Y\in \mathsf{Mat}_{n+1}(\mathbb{R})$. The convex cocompact subgroups of $\mathsf{Sp}(n,1)$, $\Gamma_1\coloneqq\Gamma$, $\Gamma_2\coloneqq\Phi(\Gamma)$ intersect at $\Gamma \cap \mathsf{L}$ and since $\Lambda_{\Phi(\Gamma)}$ is contained in $$\left\{[u+vj:1]:|u|^2+|v|^2=1, u,v\in \mathbb{R}^n\right\},$$ we have $\Lambda_{\Gamma}\cap \Lambda_{\Phi(\Gamma)}=\Lambda_{\Gamma \cap \mathsf{L}}=\partial_{\infty}\mathbb{H}_{\mathbb{R}}^{n}$. Thus, by Theorem \ref{thm:amalgam}, the there exists a finite-index subgroup $\Gamma_1<\Gamma$ such that $\Gamma_1\ast_{\Gamma \cap \mathsf{L}}\Gamma_1$ admits a convex cocompact representation into $\mathsf{Sp}(n,1)$.
\end{enumerate} 
\end{remark}

\subsection{Indiscreteness in rank 1}
Using Corlette's Archimedean superrigidity \cite{Corlette} and following the proof of \cite[Thm. 1.7]{TholozanT} we prove that certain Anosov amalgams obtained by \Cref{main-rank1-amalgam} fail to admit discrete and faithful representation into any rank 1 Lie group.

\begin{proposition} \label{indiscrete-rank1} Let $\Gamma<\mathsf{Sp}(n,1)$, $n\geq 4$, be a uniform lattice. Suppose that $\H<\Gamma$ is a quasiconvex subgroup such that $\Lambda_{\H}\cap \partial_{\infty}\mathbb{H}_{\mathbb{H}}^{m}$ is a proper subset of $\partial_{\infty}\mathbb{H}_{\mathbb{H}}^{m}$ which contains $\partial_{\infty}\mathbb{H}_{\mathbb{R}}^m$, for some $m=1,\ldots, n$. Then the amalgam $\Gamma\ast_{\H}\Gamma$ does not admit discrete and faithful representation into any rank $1$ Lie group.\end{proposition}

\begin{proof} To prove the proposition, we appeal to the argument in the proof of \cite[Thm. 1.7]{TholozanT}. Since $\mathsf{Sp}(n,1)$, with $n\geq 4$, does not embed into $\mathsf{F}_4^{-20}$, it suffices to show that there is no discrete and faithful representation $\rho:\Gamma \ast_{\H}\Gamma \to \mathsf{G}_{\ell}$, where $\mathsf{G}_{\ell}$ is isogenous to $\mathsf{Sp}(\ell,1)$ for some $\ell\in \mathbb{N}$. We argue by contradiction. Suppose such a representation $\rho$ exists, and let $\rho_L$ and $\rho_R$ denote its restrictions to the two vertex groups $\Gamma_L$ and $\Gamma_R$ (each isomorphic to $\Gamma$) in $\Gamma_L\ast_{\H}\Gamma_{R}=\Gamma \ast_{\H}\Gamma$. By Corlette's Archimedean superrigidity theorem \cite{Corlette}, there exist $\rho_{L}$- and $\rho_{R}$-equivariant totally geodesic embeddings
\[
f_{L},f_{R}:\mathbb{H}_{\mathbb{H}}^{n}\hookrightarrow \mathbb{H}_{\mathbb{H}}^{\ell}.
\]
Since $\rho_{L}$ and $\rho_{R}$ agree on $\H$, the limit sets of $\rho_L(\H)$ and $\rho_R(\H)$ in $\partial_{\infty}\mathbb{H}_{\mathbb{H}}^{\ell}$ coincide. Thus, $f_{L}(\textup{Conv}(\Lambda_{\H}))=f_{R}(\textup{Conv}(\Lambda_{\H}))$, where $\textup{Conv}(\Lambda_{\H})$ denotes the convex hull of $\Lambda_{\H}$ in $\mathbb{H}_{\mathbb{H}}^{\ell}$. In particular, since $\partial_{\infty}\mathbb{H}_{\mathbb{R}}^m\subset \Lambda_{\H}$, it follows that $f_{L}(\mathbb{H}_{\mathbb{R}}^m)\subset f_{R}(\mathbb{H}_{\mathbb{H}}^{n})$. By the classification of totally geodesic subspaces of $\mathbb{H}_{\mathbb{H}}^{\ell}$ (see \cite[Thm. 2.12]{Meyer}), $f_{L}(\mathbb{H}_{\mathbb{H}}^m)$ is the unique quaternionic hyperbolic $m$-space in $\mathbb{H}_{\mathbb{H}}^{\ell}$ containing $f_{L}(\mathbb{H}_{\mathbb{R}}^m)$, hence $f_{L}(\mathbb{H}_{\mathbb{H}}^m)\subset f_{R}(\mathbb{H}_{\mathbb{H}}^n)$.

Since $\Lambda_{\H}$ is assumed to be a proper subset of $\partial_{\infty}\mathbb{H}_{\mathbb{H}}^m$, we may fix an infinite sequence $(x_{r})_{r\in \mathbb{N}}\subset \mathbb{H}_{\mathbb{H}}^m$ with $\lim_{r}x_r\in \partial_{\infty}\mathbb{H}_{\mathbb{H}}^m\smallsetminus \Lambda_{\H}$. As $f_{L}(\mathbb{H}_{\mathbb{H}}^m)\subset f_{R}(\mathbb{H}_{\mathbb{H}}^n)$, there exists a sequence $(y_r)_{r\in \mathbb{N}}$ in $\mathbb{H}_{\mathbb{H}}^{n}$ such that $f_{L}(x_r)=f_{R}(y_r)$ for every $r$. Choose infinite sequences $(\gamma_{r,L})_{r \in \mathbb{N}}\subset \Gamma_L$ and $(\gamma_{r,R})_{r \in \mathbb{N}}\subset \Gamma_R$ such that $\big(\gamma_{r,L}^{-1}x_r\big)_{r\in \mathbb{N}}$ and $\big(\gamma_{r,R}^{-1}y_r\big)_{r\in \mathbb{N}}$ are bounded. Since $f_{L}$ and $f_{R}$ are $\rho_{L}$- and $\rho_{R}$-equivariant, and $f_{L}(x_r)=f_R(y_r)$, the sequence $\big(\rho_L(\gamma_{r,L}^{-1})\rho_R(\gamma_{r,R})\big)_{r\in \mathbb{N}}=\big(\rho(\gamma_{r,L}^{-1}\gamma_{r,R})\big)_{r\in \mathbb{N}}$ is bounded in $\mathsf{G}_{\ell}$. In particular, since $\rho$ is assumed to be discrete and faithful, the sequence $(\gamma_{r,L}^{-1}\gamma_{r,R})_{r\in \mathbb{N}}$ must be finite in $\Gamma_L\ast_{\H}\Gamma_R$. Passing to a subsequence, there exists $r_0\in \mathbb{N}$ such that $\gamma_{r,L}\gamma_{r_0,L}^{-1}=\gamma_{r,R}\gamma_{r_0,R}^{-1}\in \Gamma_L\cap \Gamma_R= \H$ for every $r$. This implies that $\lim_{r}\gamma_{r,L}=\lim_{r} x_{r}\in \Lambda_{\H}$, contradicting the choice of the sequence $(x_r)_{r\in \mathbb{N}}$. Therefore, there exists no discrete and faithful representation $\rho$ of $\Gamma \ast_{\H}\Gamma$ into any Lie group of real rank one.
\end{proof}

\section{HNN extensions of Anosov groups over cyclic subgroups}\label{sec:6}
Given a group $\Gamma$, two subgroups $\H_{+},\H_{-}<\Gamma$, and an isomorphism $\phi:\H_{-}\rightarrow \H_{+}$, the {\em HNN extension of $\Gamma$ with respect to $\phi$}, denoted $\Gamma\ast_{t\H_{-}t^{-1}=\phi(\H_{-})}$, is the group $$\big\langle \Gamma,t \ \big| \ th_{-}t^{-1}=\phi(h_{-}),h\in \H_{-}\big\rangle=\Gamma \ast \langle t\rangle/N,$$ where $N=\big\langle \big\{g(th_{-}t^{-1}\phi(h_{-}^{-1}))g^{-1}:h_{-}\in \H_{-}, g\in \Gamma \ast \langle t\rangle \big\}\big\rangle$.

A subset $\Pi \subset \Gamma$ of is called {\em separable} if for each $g\in \Gamma \smallsetminus \Pi$, there exists a homomorphism $\beta:\Gamma \rightarrow \mathsf{F}$ onto a finite group $\mathsf{F}$ such that $\beta(g)\in \mathsf{F}\smallsetminus \beta(\Pi)$. Equivalently, there exists a family $\{\Gamma_i\}_{i\in I}$ of normal finite-index subgroups of $\Gamma$ such that $\bigcap_{i\in I}\Gamma_i \Pi=\Pi$.

\medskip 
The goal of this section is to prove the following theorem.

\begin{theorem}\label{HNN-ext} Let $\mathsf{G}$ be a connected, linear, noncompact, real semisimple Lie group, let $\Gamma$ be a $\Theta$-Anosov subgroup, and let $\M_{+},\M_{-}<\Delta$ be quasiconvex, malnormal, separable subgroups of $\Gamma$. Suppose that there exists $\tau\in \mathsf{G}$ such that the following conditions hold:
\begin{enumerate}
\item There exists an automorphism $\phi:\Gamma \rightarrow \Gamma$ such that $\phi(\M_{-})=\M_{+}$ and $\tau w \tau^{-1}=\phi(w)$ for every $w \in \M_{-}$.
\item $\tau\big(\Lambda_{\Gamma}\smallsetminus \Lambda_{\M_{-}}\big)$ is antipodal to $\Lambda_{\Gamma}\smallsetminus \Lambda_{\M_{+}}$.
\item There exists an infinite order element $\gamma \in \Gamma$ with $\gamma^{\pm}\notin  \Lambda_{\M_{+}}\cup \Lambda_{\M_{-}}$, such that for every $m\in \mathbb{Z}$ and any normal finite-index subgroup $\Gamma_1<\Gamma$, the products $$(\gamma^m (\Gamma_1\cap \M_{-})\gamma^{-m})(\Gamma_1\cap \M_{+}) \ \ and \ \ (\Gamma_1\cap \M_{+})(\gamma^m (\Gamma_1\cap \M_{-})\gamma^{-m})$$ are separable subsets of $\Gamma_1$.
\end{enumerate}

Then there exists $p\in \mathbb{N}$ and a finite-index characteristic\footnote{A subgroup of $\Gamma$ is {\em characteristic} if it is invariant under any automorphism of $\Gamma$.} subgroup $\Gamma'<\Gamma$ such that the subgroup $\langle \Gamma', \tau \gamma^m\rangle<\mathsf{G}$ is $\Theta$-Anosov and is isomorphic to the HNN extension $$\Gamma' \ast_{t\H_{-}t^{-1}=\H_{+}}=\big \langle \Gamma', t \ | \ tht^{-1}=\phi(\gamma^{p}h\gamma^{-p}), \ h\in \H_{-}\big \rangle,$$
where $\H_{-}\coloneqq\gamma^{-p}(\Gamma_1\cap \M_{-})\gamma^p$ and $\H_{+}\coloneqq\Gamma_1\cap \M_{+}$.
\end{theorem}

The overall strategy of the proof of this result resembles that of \Cref{thm:amalgam}, relying on the construction of an interactive triple in the relevant flag manifold and an application of the  Combination Theorem for HNN extensions of Anosov subgroups stated below. However, the construction of the interactive triple (see \Cref{sec:6.3}) presents much greater challenges than the interactive pair constructed in \Cref{sec:2.1}.

\begin{theorem}[Dey--Kapovich {\cite[Thm. B]{DeyK:amalgam}}]\label{thm:HNN1}
    Let $\Gamma$ be a $\Theta$-Anosov subgroup of $\mathsf{G}$, let $t\in \mathsf{G}$, and let $\mathrm{H}_\pm$ be a pair of quasiconvex subgroups of $\Gamma$ such that $\mathrm{H}_\pm = \Gamma \cap t^{\pm 1} \Gamma t^{\mp1}$. Suppose that there exist compact sets $A,B_\pm \subset \mathcal{F}_\Theta = \mathsf{G}/\mathsf{P}_\Theta$ with nonempty interiors such that:
    \begin{enumerate}
        \item The pairs of sets $(A^\circ,B_\pm^\circ)$, $(B_+,B_-)$ are antipodal to each other.
        \item $\Lambda_{\G}\smallsetminus \Lambda_{\H_\pm}$ is antipodal to $B_\pm$.
        \item $t^{\pm1} A \subset B_\pm$, $t^{\pm1} B_\pm \subset B_\pm^\circ$, and for all $\g\in\G\smallsetminus (\mathrm{H}_+\cup \mathrm{H}_-)$, $\g B_\pm \subset A^\circ$.
        \item $\H_\pm$ leaves $B_\pm$ precisely invariant. That is, for all $\g\in \Gamma\smallsetminus \H_\pm$, $\g B_\pm \subset A^\circ$, and $B_\pm$ are $H_\pm$-invariant.
    \end{enumerate}
    Then the natural homomorphism $\Gamma*_{t\H_{-}t^{-1}=\phi(\H_{-})} \to \mathsf{G}$ is an injective $\Theta$-Anosov representation, where $\phi: \H_- \to \H_+$ is given by $\eta\mapsto t\eta t^{-1}$.
\end{theorem}

Before we proceed with the proof of \Cref{HNN-ext}, we introduce some notation. We fix a faithful, $\Theta$-proximal Lie group homomorphism $\tau_{\Theta}:\mathsf{G} \rightarrow \mathsf{GL}_d(\mathbb{R})$ inducing a Pl\"ucker embedding $\iota_{\Theta}:\mathcal{F}_{\Theta}\xhookrightarrow{} \mathcal{F}_{1,d-1}(\mathbb{R})$. We also equip $d_{\mathcal{F}_{\Theta}}$ with the metric induced from a distance on $\mathcal{F}_{1,d-1}(\mathbb{R})$ via the embedding $\iota_{\Theta}$. An element $g\in \mathsf{G}$ is called $\Theta$-proximal if  $\tau_{\Theta}(g)\in \mathsf{GL}_d(\mathbb{R})$ is proximal, i.e. it has a unique eigenvalue of maximum modulus. In this case, $g$ has a unique attracting fixed point in $\mathcal{F}_{\Theta}$ denote by $g^{+}$.

For $x,y\in \mathcal{F}_{\Theta}$, define their ``antipodal distance'' by $$d_{1}\big(x,y\big)\coloneqq\min \big\{\textup{dist}_{\mathbb{P}(\mathbb{K}^d)}(x_1,\mathbb{P}(V_2)), \textup{dist}_{\mathbb{P}(\mathbb{K}^d)}(x_2,\mathbb{P}(V_2))\big\},$$
where $(x_1,V_1) = \iota_{\Theta}(x)$ and $(x_2,V_2)=\iota_{\Theta}(y)$.
By the definition of $\iota_{\Theta}$, note that $x$ is antipodal to $y$ if and only if $d_1(x,y)>0$. More generally, for two subsets $A_1,A_2\subset \mathcal{F}_{\Theta}$ define their antipodal distance by
$$\textup{dist}_1(A_1,A_2)\coloneqq\inf\big\{d_1(a_1,a_2): a_1\in A_1, a_2\in A_2\big\}.$$

We will give the proof of Theorem \ref{HNN-ext} in three steps.

\subsection{Step 1: Replacing the stable letter} We assume here, and for the rest of this section, the hypotheses (1), (2), (3) of Theorem \ref{HNN-ext}. Let us consider the element $t_{p}\coloneqq \tau \gamma^p$ and the associated groups $$\M_{+}^{p}\coloneqq\M_{+}, \ \M_{-}^{p}\coloneqq\gamma^{-p}\M_{-}\gamma^p.$$
Note that 
$t_p\M_{-}^pt_p^{-1}=\M_{+}^p.$

For $\theta, \delta>0$ small enough, consider the compact sets \begin{align*}\Sigma_{\pm}^p( \theta)&\coloneqq\bigcup_{m\geq 2}t_{p}^{\pm m} \Lambda_{\Gamma}\cup \bigcup_{m\geq 1} t_p^{\pm m}\mathcal{N}_{\theta}(\Lambda_{\M^p_{\pm}})\\ X_{\pm}^p(\delta,\theta)&\coloneqq\mathcal{N}_{\delta}\big(\Sigma_{\pm}^p( \theta) \big),\end{align*}
where, as before, $\mathcal{N}_\theta(\cdot)$ denotes the {\em closed} $\theta$-neighborhood.

Fix a Cartan decomposition $\mathsf{G}=\mathsf{K}\exp(\mathfrak{a}^+)\mathsf{K}$ of $\mathsf{G}$, where $\mathsf{K}<\mathsf{G}$ is a maximal compact subgroup, $\mathfrak{a}^{+}$ is a Weyl chamber.
Let $\mu:\mathsf{G}\rightarrow \mathfrak{a}^{+}$ denote the Cartan projection. For an element $g\in \mathsf{G}$ written in the Cartan decomposition $g=k_g\exp(\mu(g))k_{g}'$, where $k_g,k_g'\in \mathsf{K}$, let $\Xi_{\Theta}(g)\coloneqq k_gP_{\Theta}$. 
If $g$ is $\Theta$-proximal, then the sequence $(\Xi_{\Theta}(g^m))_{m\geq 1}$ converges to the attracting fixed point of $g$ in $\mathcal{F}_{\Theta}$. A consequence of this is as follows (e.g. see \cite[Lem. 5.8]{GGKW}):

\begin{fact}\label{Xi-1} There exists $p_0\in \mathbb{N}$ such that for all $p\geq p_0$, both $t_p=\tau\gamma^p$ and $t_p^{-1}\in \mathsf{G}$ are $\Theta$-proximal and the following hold:
    \begin{align*}
        \lim_{p\to\infty} \sup_{k\geq 1}d_{\mathcal{F}_{\Theta}}\big(\Xi_{\Theta}(t_p^k),\tau \gamma^{+}\big)
        &=\lim_{p\to\infty} d_{\mathcal{F}_{\Theta}}\big(t_{p}^{+},\tau \gamma^{+})=0,\\
        \lim_{p\to\infty} \sup_{k\geq 1}d_{\mathcal{F}_{\Theta}}\big(\Xi_{\Theta}(t_p^{-k}),\gamma^{-}\big)
        &=\lim_{p\to\infty} d_{\mathcal{F}_{\Theta}}\big(t_p^{-},\gamma^{-}\big)=0.
    \end{align*}
\end{fact}

\begin{lemma}\label{t(N)} 
There exists $p_0\in \mathbb{N}$ such that for all $p\geq p_0$ we have that:
\begin{enumerate}
\item $t_p^{k}\Lambda_{\Gamma}$ is antipodal to $\Lambda_{\Gamma}$ for every $k\geq 2$.
\item $t_p(\Lambda_{\Gamma}\smallsetminus \Lambda_{\M_{-}^p})$ is antipodal to $\Lambda_{\Gamma}\smallsetminus \Lambda_{\M_{+}^p}$.
\item $t_p\in \mathsf{G}$ is biproximal and $t_p^{\pm}$ is antipodal to $\Lambda_{\Gamma}$.\end{enumerate} \end{lemma}

\begin{proof}  
Since $\gamma^{\pm}\notin \Lambda_{\M_+}\cup \Lambda_{\M_{-}}$ and $\tau(\Lambda_{\Gamma}\smallsetminus \Lambda_{\M_{-}})$ is antipodal to $\Lambda_{\Gamma}\smallsetminus \Lambda_{\M_{+}}$, $\tau\Lambda_{\Gamma}$ is antipodal to $\gamma^{\pm}$. 
Moreover, since $\gamma$ is $\Theta$-proximal, we may choose $\epsilon>0$ and $p_0>1$ large enough such that $\tau B_{\epsilon}(\gamma^{+})$ is antipodal to $\Lambda_{\Gamma}$ and for every $p\geq p_0$, $\gamma^{p}B_{\epsilon}(\gamma^{+})\subset B_{\epsilon/2}(\gamma^{+})$, $\gamma^p\big(\tau \Lambda_{\Gamma}\big)\subset B_{\epsilon}(\gamma^{+})$. Then, by the previous choices, it is straightforward to check that $t_p^k\Lambda_{\Gamma}\subset \tau B_{\epsilon}(\gamma^{+})$ for every $k\geq 2$. Thus, (1) follows.

Part (2) follows immediately by the definition of $\M_{\pm}^p$. 

We now prove (3). Using Fact \ref{Xi-1}, we have that for large enough $p_0>1$, $t_{p}^{+}$ is uniformly close to $\tau\gamma^{+}$ and hence is antipodal to $\Lambda_{\Gamma}$ for every $p\geq p_0$. Since $\lim_p (\tau^{-1}\gamma^{-p})^{+}=\tau^{-1}\gamma^{-}$, we may further enlarge $p_0>1$ such that $(\tau^{-1}\gamma^{-p})^{+}$ is antipodal to $\Lambda_{\Gamma}$ for every $p\geq p_0$. Since $\gamma^p t_{p}^{-}=(\gamma^p t_p\gamma^{-p})^{-}=(\gamma^p \tau)^{-}=(\tau^{-1}\gamma^{-p})^{+}$ is antipodal to $\Lambda_{\Gamma}$, we conclude that $t_{p}^{-}$ is antipodal to $\Lambda_{\Gamma}$ for every $p\geq p_0$.\end{proof}

\begin{lemma} \label{ineq-1} 
There exist $p_1,\theta_1,\delta_1>0$ and $\epsilon>0$ such that for all \hbox{$p\geq p_1$ and $k\geq 1$, we have that}
\begin{align*} 
\inf_{\theta<\theta_1, \delta<\delta_1}\textup{dist}_1\big(X_{+}^p(\delta, \theta), \Xi_{\Theta}(t_p^{-k})\big) &\geq \epsilon,\\ 
\inf_{\theta<\theta_1, \delta<\delta_1}\textup{dist}_1\big(X_{-}^p(\delta, \theta), \Xi_{\Theta}(t_p^{k})\big) &\geq \epsilon.
\end{align*}
\end{lemma}

\begin{proof} By \Cref{Xi-1}, for all $p>1$ sufficiently large, there exists $\epsilon_p>0$ with $\lim_p\epsilon_p=0$ such that for every $k\geq 1$, $d_{\mathcal{F}_{\Theta}}(\Xi_{\Theta}(t_p^k),\tau \gamma^{+})\leq \epsilon_p$ and $d_{\mathcal{F}_{\Theta}}(\Xi_{\Theta}(t_p^{-k}),\gamma^{-})\leq \epsilon_p$. 
Observe that by Lemma \ref{t(N)} (1) and (2) and since for small $\theta>0$,  $\mathcal{N}_{\theta}(\Lambda_{\M_{\pm}^p})$ is antipodal to $\gamma^{\pm}$, it follows that $$\lim_{p\rightarrow \infty}\sup_{m\geq 2}d_{\mathcal{F}_{\Theta}}\big(t_p^m\Lambda_{\Gamma},\tau\gamma^{+}\big)=\lim_{p\rightarrow \infty}\sup_{m\geq 1}d_{\mathcal{F}_{\Theta}}\big(t_p^m\mathcal{N}_{\theta}(\Lambda_{\M_{+}^p}),\tau\gamma^{+}\big)=0$$ for every $\theta>0$ sufficiently small. It follows that there exist $p_1>0$ and $\epsilon_1, \theta_1>0$ such that for all $p\geq p_1$ and $0<\theta,\delta<\theta_1$, there exists $\epsilon_p'>0$ with $\epsilon'_p\to 0$ as $p\to\infty$ such that $$\Sigma_{+}^p(\theta)\subset B_{\epsilon_p'}(\tau \gamma^{+}).$$ 
In addition, $\Lambda_{\M_{-}^p}=\gamma^{-p}\Lambda_{\M_{-}}$ Hausdorff converges to $\gamma^{-}$ as $p \rightarrow \infty$.
After shrinking $\theta_1>0$ further and increasing $p_1>0$, there exists $\epsilon_p''>0$ with $\lim_{p\to\infty} \epsilon_p''=0$ such that $$\Sigma_{-}^p( \theta)\subset B_{\epsilon_p''}(\gamma^{-}), \  \forall \ 0<\theta<\theta_1 \ \textup{and} \ p\geq p_1.$$

In the first case, as $\sup_{k\geq 1}d(\Xi(t_p^{-k}),\gamma^{-})\leq \epsilon_p$ and $\Sigma_{+}^p(\theta)\subset B_{\epsilon_p'}(\tau \gamma^{+})$, we deduce that $$\textup{dist}_1\big(X_{+}^p(\delta, \theta), \Xi_{\Theta}(t_p^{-k})\big) \geq \frac{1}{10}\textup{dist}_1(\tau\gamma^{+},\gamma^{-})$$ for every $\delta<\frac{1}{10}\textup{dist}_1(\tau\gamma^{+},\gamma^{-})$, $\theta<\theta_1$ and $p\geq p_1$. Similarly, we conclude that $$\textup{dist}_1\big(X_{-}^p(\delta, \theta), \Xi_{\Theta}(t_p^{k})\big) \geq \frac{1}{10}\textup{dist}_1(\tau\gamma^{+},\gamma^{-})$$ for every $\delta<\frac{1}{10}\textup{dist}_1(\tau\gamma^{+},\gamma^{-})\coloneqq\delta_1$, $\theta<\theta_1$ and $p\geq p_1$.\end{proof}

\begin{lemma}\label{ineq-2} 
There exist $p_1,\theta_1,\delta_1>0$ such that for all  $p\geq p_1$ and all $k\geq 1$,
we have that 
\[
t_p^{\pm k}X_{\pm}^p\big(\delta,\theta\big)\subset
X_{\pm}^p\big(\delta/2,\theta\big).
\]
\end{lemma}

\begin{proof}
Let $\epsilon,\delta_1,\theta_1, p_1>0$ such that the conclusion of Lemma \ref{ineq-1} holds. For all $x,y\in X_{+}^p(\delta,\theta)$, we have that $d_{\mathcal{F}_{\Theta}}(x,\Xi(t_p^{-k})),d_{\mathcal{F}_{\Theta}}(y,\Xi(t_p^{-k}))\geq \epsilon$.
Hence, the action of $t_p^{k}$ on $X_{+}^p(\delta, \theta)$ is $\zeta(p,\epsilon_1)$-Lipschitz\footnote{For $w\in \mathsf{GL}_d(\mathbb{R})$, written in the Cartan decomposition as $w=k_{w}\exp(\mu(w))k_w'$ such that $\Xi_{d-1}(w^{-1})=(k_w')^{-1}\langle e_2,\ldots,e_d\rangle$, then, for every $x,y\notin \Xi_{d-1}(w^{-1})$, $$d_{\mathbb{P}}(wx,wy)\leq \frac{\sigma_2(w)}{\sigma_1(w)}\textup{dist}_{\mathbb{P}(\mathbb{R}^d)}(x,\Xi_{d-1}(w^{-1}))^{-1}\textup{dist}_{\mathbb{P}(\mathbb{R}^d)}(y,\Xi_{d-1}(w^{-1}))^{-1}.$$} for some $\zeta(p,\epsilon_1)>0$ of the form $$\zeta(p,\epsilon)\coloneqq\frac{c_{G}}{\epsilon^{2}}\frac{\sigma_2(\tau_{\Theta}(\tau \gamma^p))}{\sigma_1(\tau_{\Theta}(\tau \gamma^p))},$$ where $\sigma_i(\tau_{\Theta}(\tau \gamma^p))$, $i\in \{1,2\}$, is the $i$-th singular value of $\tau_{\Theta}(\tau \gamma^p)$ and $c_{G}>0$ is a constant depending only the choice of the embedding $\iota_{\Theta}:\mathcal{F}_{\Theta}\xhookrightarrow{} \mathcal{F}_{1,d-1}(\mathbb{R})$. 
In particular, by enlarging $p_1>1$, for every $p>p_1$ we have that $\zeta(p,\epsilon)<\frac{1}{2}$.
Hence, $$t_pX_{+}^p(\delta, \theta)=t_p \mathcal{N}_{\delta}\big(\Sigma_{+}^{p}(\theta)\big)\subset \mathcal{N}_{\delta \zeta(p,\epsilon)}\big(t_p\Sigma_{+}^p(\theta)\big)\subset X_{+}^p(\delta/2,\theta).$$ This also shows $t_p^kX_{+}^p(\delta,\theta)\subset X_{+}^p(\delta/2,\theta)$. 

The proof for the inclusion $t_p^{-k}X_{-}^p(\delta, \theta)\subset X_{-}^p(\delta,\theta)$ is analogous; we skip the details.
\end{proof}

\subsection{Step 2: Replacing the subgroups} Using \Cref{ineq-2,t(N)}, we may assume that we have an Anosov group $\Gamma<\mathsf{G}$, a pair of malnormal separable subgroups $\H_{-}\coloneqq\gamma^{-p}\M_{-}\gamma^p, \H_{+}=\M_{+}<\Gamma$, a biproximal element $t\in \mathsf{G}$ with $t\H_{-}t^{-1}=\H_{+}$ such that the following conditions hold:
\begin{enumerate}[label=(\alph*)]
\item $th_{-}t^{-1}=\phi_p(h_{-})$ for every $h_{-}\in \H_{-}$ and some automorphism $\phi_p:\Gamma \rightarrow \Gamma$ such that $\phi_p(h_{-})=\phi(\gamma^p h_{-}\gamma^{-p})$.
\item $t\big(\Lambda_{\Gamma}\smallsetminus \Lambda_{\H_{-}}\big)$ is antipodal to $\Lambda_{\Gamma}\smallsetminus \Lambda_{\H_{+}}$.
\item $t^{n}\Lambda_{\Gamma}$ is antipodal to $\Lambda_{\Gamma}$ for every $n\in \mathbb{Z}\smallsetminus \{\pm 1\}$ and $t^{\pm}$ is antipodal to $\Lambda_{\Gamma}$.
\item For every finite-index subgroup $\Gamma_1<\Gamma$, the products $(\Gamma_1\cap \H_{+})(\Gamma_1\cap \H_{-})$ and $(\Gamma_1\cap \H_{-})(\Gamma_1\cap \H_{+})$ are separable subsets of $\Gamma_1$.
\item If we let the compact sets:
\begin{align*} 
\Sigma_{+}( \theta)&\coloneqq\bigcup_{m\geq 2}t^m\Lambda_{\Gamma}\cup \bigcup_{m \geq 1}t^{ m}\mathcal{N}_{\theta}(\Lambda_{\H_{+}}) \cup\{ t^+\},\\ 
\Sigma_{-}(\theta)&\coloneqq\bigcup_{m\geq 2}t^{-m}\Lambda_{\Gamma}\cup \bigcup_{m \geq 1}t^{-m}\mathcal{N}_{\theta}(\Lambda_{\H_{-}}) \cup \{t^-\},\\ 
X_{\pm}(\delta, \theta)&=\mathcal{N}_{\delta}\big(\Sigma_{\pm}(\theta)\big),\end{align*} 
there exist $\delta_0,\theta_0>0$, such that for every $k\in \mathbb{N}$ and $\delta<\delta_0$, $\theta<\theta_0$, we have that $$t^{\pm k}X_{\pm}(\delta, \theta)\subset X_{\pm}(\delta/2,\theta).$$
\end{enumerate}

By compactness, we have the following fact.

\begin{lemma}\label{theta} There exists $\delta_1>0$ such that for every $\theta<\delta_1$ and $\delta<\delta_1$ we have that.
\begin{enumerate}
\item $X_{\pm}(\delta, \theta)$ are antipodal to $ \Lambda_{\Gamma}$.
\item  $X_{+}(\delta, \theta)$ is antipodal to $X_{-}(\delta, \theta)$.
\item $\mathcal{N}_{\theta}(\Lambda_{\H_{+}})\cup X_{+}(\delta, \theta)$ is antipodal to $\mathcal{N}_{\theta}(\Lambda_{\H_{-}})\cup X_{-}(\delta, \theta)$.
\end{enumerate}
\end{lemma}

\begin{proof} As $t^{\pm}$ is antipodal to $\Lambda_{\Gamma}$, $t^{\pm m} \Lambda_{\H_{\pm }}$ is antipodal to $\Lambda_{\Gamma}$ for $m\geq 1$, and $t^{\pm q}\Lambda_{\Gamma}$ is antipodal to $\Lambda_{\Gamma}$ for $q\geq 2$, we may fix a constant $\theta_1>0$ such that $\mathcal{N}_{\theta_1}(\Sigma_{+}(\theta_1))$, $\mathcal{N}_{\theta_1}(\Sigma_{-}(\theta_1))$ are antipodal to $\mathcal{N}_{\theta_1}(\Lambda_{\Gamma})$, and $\mathcal{N}_{\theta_1}(\Sigma_{+}(\theta_1))$, $\mathcal{N}_{\theta_1}(\Sigma_{-}(\theta_1))$ are antipodal to each other. This result follows from this.
\end{proof}

For the rest of this section, we will fix $\theta>0$ and $\delta_1>0$ such that  the statements (1), (2) and (3) of  \Cref{theta} are in effect. Given this $\theta>0$, we may now replace $\Gamma$ with a deep characteristic finite-index subgroup that is $\phi_{p}$-invariant and intersects $\H_{+}$ and $\H_{-}$ in the same way, such that the new groups, still denoted by $\H^{\pm}$, satisfy: $$(\H_{\pm}\smallsetminus \{1\})X_{\pm}(\delta, \theta)\subset \mathcal{N}_{\theta}(\Lambda_{H_{\pm}})$$ for every $\delta<\delta_1$. We need the following lemma.

\begin{lemma}\label{semigroup} Let $(\H_{\pm},t^{\pm 1})$ be the semigroup generated by $\H_{\pm}\cup\{t^{\pm 1}\}$. Then for every $\delta<\delta_1$ we have that:
\begin{enumerate}
\item If $g\in (H_{\pm},t^{\pm})$ is a word starting with a non-zero power of $t$, then $gX_{\pm}(\delta,\theta)\subset X_{\pm}(\delta,\theta)$.

\item $(\H^{\pm},t^{\pm 1})X_{\pm}(\delta, \theta)\subset X_{\pm} (\delta, \theta)\cup \mathcal{N}_{\theta}(\Lambda_{\H_{\pm}})$.
\item For every $k\in \mathbb{N}$,  \begin{align}\label{semi-1}t(\H_{\pm},t^{\pm 1})X_{\pm}(\delta, \theta)\subset X_{\pm} (\delta/2,\theta).\end{align} \end{enumerate} \end{lemma}

\begin{proof} Fix any $\delta<\delta_1$. As $t^{\mu}X_{+}(\delta, \theta)\subset X_{+}(\delta/2,\theta)$ and $t^{\mu}\mathcal{N}_{\theta}(\Lambda_{\H_{+}})\subset X_{+}(\delta,\theta)$ for every $\mu \in \mathbb{N}$, and in addition, $(\H_{+}\smallsetminus \{1\})X_{+}(\delta,\theta)\subset \mathcal{N}_{\theta}(\Lambda_{\H_{\pm}})$, we conclude (1) and (2).
In particular, since $t\mathcal{N}_{\theta}(\Lambda_{\H_{+}})\subset X_{+}(r,\theta)$ for every $r<\delta_1$, we conclude that \begin{align*}t(H_{+},t)X_{+}(\delta, \theta)\subset tX_{+}(\delta, \theta)\cup t\mathcal{N}_{\theta}(\Lambda_{H_{\pm}})\subset X_{+} (\delta/2,\theta)\cup  t\mathcal{N}_{\theta}(\Lambda_{H_{+}})\subset X_{+}(\delta/2,\theta).\end{align*} This proves the lemma. \end{proof}

Having fixed the groups $\H_{\pm}$ and $t\in G$, we now fix compact fundamental domains $F^{\pm}\subset \Lambda_{\Gamma}$ for the action of $\H_{\pm}$ on $\Lambda_{\Gamma}\smallsetminus \Lambda_{\H_{\pm}}$. We also need the following lemma.

\begin{lemma}\label{semigroup-transv} For every $\delta<\delta_1$ we have that:
\begin{enumerate}
\item $(\H_{\pm},t)X_{\pm}(\delta,\theta)$ is antipodal to $\Lambda_{\Gamma}$.
\item $\overline{(\H_{\pm},t)X_{\pm}(\delta,\theta)}$ is antipodal to $F^{\pm}$ and hence to $\Lambda_{\Gamma}\smallsetminus \Lambda_{\H_{\pm}}$.\end{enumerate}  \end{lemma}

\begin{proof} Fix $\delta<\delta_1$.
\medskip

\noindent (1) Clearly, $X_+(\delta,\theta)$ is antipodal to $\Lambda_{\Gamma}$ by the choice we made from \Cref{theta}. If $g\in \H_{+}$, $gX_+(\delta,\theta)$ is antipodal to $\Lambda_{\Gamma}$, as $X_+(\delta,\theta)$ is antipodal to $\Lambda_{\Gamma}$. By Lemma \ref{semigroup} (i), if $g\in (\H_{+},t)$ is any word starting from a non-trivial power of $t$, we have $gX_{+}(\delta,\theta)\subset X_+(\delta, \theta)$, and $gX_{+}(\delta, \theta)$ is antipodal to $\Lambda_{\Gamma}$. If $g\in (\Gamma_{+},t)$ is a word of the form $g=h_{0}g'$, where $h_0\in\H_{+}$ and $g'$ starts with a power of $t$, then as $g'X_+(\delta,\theta)\subset X_+(\delta,\theta)$ is antipodal to $\Lambda_{\Gamma}$, we deduce that $gX_+(\delta,\theta)=h_0gX_+(\delta,\theta)$ is antipodal to $\Lambda_{\Gamma}$.

\smallskip
\noindent \textup{(2)} It suffices to prove that for any infinite sequence $(g_n)$ in $(\H_{+},t)$ and any sequence $(w_n)$ in $X_{+}(\delta,\theta)$, $(g_nw_n)$ has an accumulation point which is antipodal to $F^{+}$.

If $(g_n)\subset \H_{+}$ for infinitely many $n$, as $X_+(\delta,\theta)$ is antipodal to $\Lambda_{\Gamma}$, we have that $\lim_n g_nw_n=\lim_n g_n^{+}\in \Lambda_{\H_{+}}$. If $g_n\notin \H_{+}$ for all but finitely many $n$,  write $g_n\coloneqq h_{n,1}t^{m_{1n}}\cdots h_{n,r_n}t^{m_{r_n n}}$ for some $h_{n,1},\ldots,h_{n,k_n}\in \H_{+}$ and $m_{1n}\neq 0$. If $(h_{n,1})$ does not escape to infinity in $\H_{+}$, there is $(k_n)$ such that $h_{k_n,1}=h_0$ and as $h_{k_n,1}^{-1}g_{k_n}\in (\H_{+},t)$ is a word that starts with a power of $t$, by Lemma \ref{semigroup} (i) we have $h_{k_n,1}^{-1}g_{k_n}w_{k_n}\in X_{+}(\delta,\theta)$ and $\lim_n g_{k_n}w_{k_n}\in h_0X_{+}(\delta, \theta)$, which is antipodal to $F^{+}$. If $h_{n,1}\rightarrow \infty$, since $h_{n,1}^{-1}g_n$ starts with a power of $t$, $(h_{n,1}^{-1}g_nw_n)$ accumulates to a point in $X_{+}(\delta,\theta)$ that is antipodal to $\lim_{n}h_{n,1}^{-1}\in \Lambda_{\H_{+}}$. Thus, $\lim_{n}g_nw_n=\lim_n h_{n,1}(h_{n,1}^{-1}g_n))\in \Lambda_{\H_{+}}$ and this limit point is antipodal to $F^{+}$.\end{proof}

\begin{lemma}\label{constant-1} There is  a constant $0<c<1$, depending only on $\H_{\pm}$ and $F^{\pm}$, with the property that for every $\delta<\delta_1$ we have that 
\begin{align}\label{Lip-1} 
t^2\H_{-}\mathcal{N}_{c \delta}(F^{-})
&\subset \mathcal{N}_{\delta/2}\big(t^2 \Lambda_{\Gamma}\big)\subset X_{+}(\delta/2, \theta)\\
t^{-2}\H_{+} \mathcal{N}_{c \delta}(F^{+})
&\subset \mathcal{N}_{\delta/2}\big(t^{-2} \Lambda_{\Gamma}\big)\subset X_{-}(\delta/2, \theta).
\end{align}\end{lemma}

\begin{proof} Since the action of $t^2\in G$ on $\mathcal{F}_{\Theta}$ is bi-Lipschitz, we may choose $0<c_1<1$ such that $t^2 \mathcal{N}_{c_1\delta}(\Lambda_{\Gamma})\subset \mathcal{N}_{\delta/2}(t^2 \Lambda_{\Gamma})$ for every $\delta<\delta_1$. As $F^{\pm}$ is away from $\Lambda_{\H_{\pm}}$, there is a finite susbet $J_{\pm}\subset \H_{\pm}$ such that for every $\alpha \in \H_{\pm}\smallsetminus J_{\pm}$ we have that $$t^2\alpha \mathcal{N}_{\delta}(F^{\pm})\subset t^2\mathcal{N}_{c_1\delta}(\gamma F^{\pm})\subset \mathcal{N}_{\delta/2}(t^2\Lambda_{\Gamma}).$$ Since for every $\alpha \in J_{\pm}$, the action of $t^2 \alpha\in G$ on $\mathcal{F}_{\Theta}$ is Lipschitz, we may choose $0<c_{\pm}<1$ such that $t^2 \alpha \mathcal{N}_{c \delta}(F^{\pm})\subset \mathcal{N}_{\delta/2}(t^2 \alpha F^{\pm})$ for every $\alpha \in J_{\pm}$. Thus, for every $\alpha \in \H_{\pm}$ we have $t^2 \alpha \mathcal{N}_{c\delta}(F^{\pm})\subset \mathcal{N}_{\delta}(t\Lambda_{\Gamma})$ for every $\delta<\delta_1$ and (\ref{Lip-1}) follows. By taking $c\coloneqq\min \{c_{+},c_{-}\}$ we have the conclusion. \end{proof}

\subsection{Step 3: Construction of an interactive triple} \label{sec:6.3}
Having made the previous choices of $\H_{\pm}, F^{\pm}$, $c,\theta>0$, similarly as in the first section we may define the ping pong sets for $\delta<\delta_1$:
\begin{align}
\begin{split}\label{eqn:HNNsets}
B_{+}^{\delta}&\coloneqq\Lambda_{\H_{+}}\cup t\H_{-}\mathcal{N}_{c \delta}(F^{-})\cup (\H_{+},t)X_{+}(\delta,\theta),\\ 
B_{-}^{\delta}&\coloneqq\Lambda_{\H_{-}}\cup t^{-1}\H_{+}\mathcal{N}_{c \delta}(F^{+})\cup (\H_{-},t^{-1})X_{-}(\delta,\theta),\\ 
A_{+}^{\delta}&=\Lambda_{\H_{+}}\cup \H_{+} \mathcal{N}_{c\delta}(F^{+}),\\ 
A_{-}^{\delta}&=\Lambda_{\H_{-}}\cup \H_{-} \mathcal{N}_{c\delta}(F^{-}).
\end{split}\end{align}
Note that we immediately have
\begin{equation}\label{eqn:AB}
    t^{\pm1} A_\mp^\delta \subset B_\pm^\delta.
\end{equation}

\begin{lemma}\label{conditions-1} There exists $\kappa>0$ such that for every $\delta<\kappa$ the following hold.
\begin{enumerate}
    \item $\Lambda_\Gamma\smallsetminus \Lambda_{\rm H_\pm} \subset (A^\delta_\pm)^\circ$.
    \item $t^{\pm 1}B_{\pm}^{\delta}\subset (B_{\pm}^{\delta})^{\circ}$.
    \item $B_{\pm}^{\delta}$ is antipodal to $A_{\pm}^{\delta}\smallsetminus \Lambda_{\H_{\pm}}$.
    \item  $B_{-}^{\delta}$ is antipodal to $B_{+}^{\delta}$.
\end{enumerate}
\end{lemma}

\begin{proof} 
(1) Since $F_\pm$ are fundamental domains for the actions $\H_\pm \acts \Lambda_\Gamma\smallsetminus \Lambda_{\rm H_\pm}$, respectively, this follows immediately from the definition of $A^\delta_\pm$ given by \eqref{eqn:HNNsets}.

\smallskip
\noindent (2) By Lemma \ref{semigroup} (3) we have that $t(\H^{+},t)X(\delta,\theta)\subset X(\delta/2,\theta)\subset (B_{+}^{\delta})^{\circ}$ for every $\delta<\delta_1$. In addition, by  Lemma \ref{constant-1} we have $t(t\H_{-}\mathcal{N}_{c\delta}(F^{-}))=t^2\H_{-}\mathcal{N}_{c\delta}(F^{-})\subset \mathcal{N}_{\delta/2}(t^2\Lambda_{\Gamma})\subset (B_{+}^{\delta})^{\circ}$.
\smallskip

\noindent (3) For $\delta>0$ define the compact set \begin{align}\label{def-C}C_{+}^{\delta}\coloneqq\Lambda_{\H_{+}}\cup t\H_{-}\mathcal{N}_{c \delta}(F^{-})\cup (\H_{+},t)X_{+}(\delta_1/2,\theta)\end{align} which contains $B_{+}^{\delta}$. We will prove that there is small enough $\delta>0$ such that $C_{+}^{\delta}$ is antipodal to $\mathcal{N}_{\delta}(F^{+})$.  If not, there are sequences $\delta_n\rightarrow 0$, $(x_n)\subset C_{+}^{\delta_n}$ and $(y_n)\subset \mathcal{N}_{\delta_n}(F^{+})$ such that $x_n$ is not antipodal to $y_n$ for every $n$. In particular, up to passing to subsequences, $x\coloneqq\lim_n x_n$ and $y\coloneqq\lim_n y_n$, $y\in F^{+}$, are not antipodal.

If for infinitely many $n\in \mathbb{N}$, $x_n\in (\H_{+},t)X_{+}(\delta_1/2, \theta)$, then $x\in \overline{ (\H_{+},t)X_{+}(\delta_1/2, \theta)}$. By Lemma \ref{semigroup-transv} (2), $x$ has to be antipodal to $F^{+}$, a contradiction. Thus eventually we have $x_n\in \Lambda_{\H_{+}}\cup t\H_{-}\mathcal{N}_{c \delta_n}(F^{-})$. By our previous assumption, $\lim_n x_n=x$ is not antipodal to a point $y\in F^{+}$. In this case, note that $x$ lies in the Hausdorff limit $\lim_{n}(\Lambda_{\H_{+}}\cup t\H_{-}\mathcal{N}_{c\delta_n}(F^{-}))=\Lambda_{\H_{+}}\cup t\H_{-}F_{-}$. However, every point in the last set is antipodal to $F^{+}$. This is again a contradiction and the conclusion follows.
\smallskip

\noindent (4) It suffices to check that $C_{+}^{\delta}$ and $C_{-}^{\delta}$ are antipodal for small $\delta>0$. For this it suffices to see that the sets: \begin{align*} \Lambda_{\H_{+}}\cup t\H_{-}F^{-}&\cup \overline{(\H_{+},t)X_{+}(\delta_1/2,\theta)}\\ \Lambda_{\H_{-}}\cup t^{-1}\H_{+}F^{+}&\cup \overline{(\H_{-},t)X_{-}(\delta_1/2,\theta)}\end{align*} are antipodal.

We already know that the sets $\Lambda_{\H_{+}}\cup t\H_{-}F^{-}$ and $\Lambda_{\H_{-}}\cup t^{-1}\H_{+}F^{+}$ are antipodal to each other. By Lemma \ref{semigroup-transv} (2), $\overline{(\H_{+},t)X_{+}(\delta_1/2,\theta)}$ is antipodal to $\Lambda_{\H_{-}}\subset \Lambda_{\Gamma}\smallsetminus \Lambda_{\H_{+}}$. 
In addition, $\overline{(\H_{+},t)X_{+}(\delta_1/2,\theta)}$ is antipodal to $t^{-1}\H_{+}F^{+}$, since we have that 
$$t\overline{(\H_{+},t)X_{+}(\delta_1/2,\theta)}\subset \overline{(\H_{+},t)X_{+}(\delta_1/2,\theta)},$$
and the latter is antipodal to $\H_{+}F^{+}$ by Lemma \ref{semigroup-transv} (2). Similarly, $\overline{(\H_{-},t)X_{-}(\delta_1/2,\theta)}$ is antipodal to $\Lambda_{\H_{+}}\cup t\H_{-}F^{-}$.

It remains to check that $\overline{(\H_{+},t)X_{+}(\delta_1/2,\theta)}$ and $\overline{(\H_{-},t)X_{-}(\delta_1/2,\theta)}$ are antipodal. Note that by Lemma \ref{semigroup} (2), we have that $$\overline{(\H_{\pm},t)X_{\pm }(\delta_1/2,\theta)}\subset X_{\pm}(\delta,\theta)\cup \mathcal{N}_{\theta}(\Lambda_{\H_{\pm}}).$$ By the initial choice of $\theta>0$ from \Cref{theta}, the sets $X_{+}(\delta,\theta)\cup \mathcal{N}_{\theta}(\Lambda_{\H_{+}})$ and $X_{-}(\delta,\theta)\cup \mathcal{N}_{\theta}(\Lambda_{\H_{-}})$ are antipodal. Hence $B_{-}^{\delta}$ is antipodal to $B_{+}^{\delta}$.
\end{proof}

Recall that $\Gamma<\mathsf{G}$ is a $\Theta$-Anosov subgroup, $\H_{\pm}<\G$ are quasiconvex separable subgroups and we fixed fundamental domains $F^{\pm}$ for the action of $\H_{\pm}$ on $\Lambda_{\Gamma}\smallsetminus \Lambda_{\H{\pm}}$. In addition, by the assumptions of the Theorem \ref{HNN-ext}, $\H_{\pm}$, $\H_{+}\H_{-}$, $\H_{-}\H_{+}$ are separable subsets of $\Gamma$. For a compact subset $\mathcal{K}\subset \mathcal{F}_{\Theta}$ and $0<\zeta<1$, let $$T_{\zeta}(\mathcal{K})\coloneqq\big\{x\in \mathcal{F}_{\Theta}: \textup{dist}_1(x,\mathcal{K})\geq \zeta\big\}$$ be the set of all points which are $\zeta$-antipodal to every point in $\mathcal{K}$. We will need the following lemma.

\begin{lemma}\label{product-sep} 
For every $\zeta,\epsilon>0$, there exists a finite-index subgroup $\Gamma_{\epsilon}<\Gamma$ such that for the intersections $\H_{\epsilon^{\pm}}\coloneqq\Gamma_\epsilon \cap \M_{\pm}$ we have $\H_{\epsilon^{+}}\coloneqq\phi(\H_{\epsilon^{-}})$, $tht^{-1}=\phi_p(h)$ for every $h_{-}\in \H_{\epsilon^{-}}$ and the following conditions hold:
\begin{enumerate}
 \item For every $\g\in \G_\epsilon \smallsetminus \H_{+}\H_{-}$ we can write $\g=h_{+}f_1h_{-}$ for some $f_1\in \Gamma \smallsetminus \{1\}$, $h_{\pm}\in \H_{\pm}$ with $f_1T_{\zeta}(F^{-})\subset \mathcal{N}_{\epsilon}(F^{+})$.
 \item For every $\g\in \G_\epsilon \smallsetminus \H_{-}\H_{+}$ we can write $\g=h_{-}f_2h_{+}$ for some $f_2\in \G \smallsetminus \{1\}$, $h_{\pm}\in \H_{\pm}$ with $f_2T_{\zeta}(F^{+})\subset \mathcal{N}_{\epsilon}(F^{-})$.
 \item For every $\g\in \G_{\epsilon}\smallsetminus \H_{\pm}$ we can write $\g=h_{1}wh_{2}$ for some $w \in \G \smallsetminus \{1\}$, $h_{1},h_2\in \H_{\pm}$ with $wT_{\zeta}(F^{\pm})\subset \mathcal{N}_{\epsilon}(F^{\pm})$.   
\end{enumerate} 
\end{lemma}

\begin{proof} Before we give the proof of the lemma we provide some further notation. We fix a word metric $d_{\Gamma}$ on $\Gamma$, a visual metric on the bordification $\Gamma \cup \partial_{\infty}\Gamma$ and denote by $\mathcal{N}_{\epsilon}(S)$ the closed $\epsilon>0$ neighborhood of a subset $S\subset \Gamma \sqcup \partial_{\infty}\Gamma$ (with respect to this fixed visual metric). Given the fundamental domains $F^{\pm}$ for the (left) action of $\H_{\pm}$ on $\Lambda_{\Gamma}\smallsetminus \Lambda_{\H_{\pm}}$, we may fix a fundamental domain $F_{0}^{\pm}$ for the action of $\H_{\pm}$ on $\Gamma \sqcup (\partial_{\infty}\Gamma \smallsetminus \partial_{\infty}\H_{\pm})$, i.e. $$\Gamma \sqcup (\partial_{\infty}\Gamma\smallsetminus \partial_{\infty}\H_{\pm})=\H_{\pm}F_0^{\pm}$$ such that $\xi(F_{0}^{+}\cap \partial_{\infty}\Gamma)=F^{\pm}$, where $\xi:\partial_{\infty}\Gamma \rightarrow \mathcal{F}_{\Theta}$ is the Anosov limit map of $\Gamma<\mathsf{G}$. 
\par Let us now fix $\zeta,\epsilon>0$. Since $F_{0}^{\pm}$ is compact and $F_{0}^{\pm}\cap \partial_{\infty}\H_{\pm}$ is empty, we can choose $C>0$ and $\varepsilon>0$, depending only on $\Gamma$, the choice of $F_{0}^{\pm}$ and $\zeta,\epsilon>0$, such that the following conditions hold:

\begin{enumerate}[label=(\alph*)]
    \item for every $f \in F_0^{\pm}\cap \Gamma$ with $d_{\Gamma}(f,1)>C$, then $f, f^{+} \in \mathcal{N}_{\varepsilon/2}(F_0^{\pm}\cap \partial_{\infty}\Gamma)$,
\item for every $f\in \Gamma$ with $f^{-1}\in \mathcal{N}_{\varepsilon/2}(F_{0}^{\pm})$ and $d_{\Gamma}(f,1)>C$, then $f\H_{\pm}\subset \mathcal{N}_{\varepsilon}(f^{+})$, 
\item for every $g\in \Gamma$ such that $g^{+}\in \mathcal{N}_{\varepsilon}(F_{0}^{\pm}\cap \partial_{\infty}\Gamma)$ and $g^{-}\in \mathcal{N}_{\varepsilon}(F_0^{\mp}\cap \partial_{\infty}\Gamma)$, $ fT_{\zeta}(F^{\mp})\subset \mathcal{N}_{\epsilon}(F^{\pm})$,
\item for every $g\in \Gamma$ with $g^{+}\in \mathcal{N}_{\varepsilon}(F_{0}^{\pm}\cap \partial_{\infty}\Gamma)$ and $g^{-}\in \mathcal{N}_{\varepsilon}(F_0^{\pm}\cap \partial_{\infty}\Gamma)$, $ fT_{\zeta}(F^{\pm})\subset \mathcal{N}_{\epsilon}(F^{\pm})$.
\end{enumerate}

Since $\H_{\pm}, \H_{+}\H_{-}$ and $\H_{-}\H_{+}$ are all separable subsets of $\Gamma$, we may choose a characteristic finite-index subgroup\footnote{For a finite-index subgroup of $\Delta<\Gamma$, the intersection $\Delta'$ of all finite-index subgroups of $\Gamma$ of index at most $[\Gamma:\Delta]$ is characteristic in $\Gamma$.} $\Gamma_{\epsilon}<\Gamma$ such that the following conditions hold true:

\begin{enumerate}[label=(\alph*)]\setcounter{enumi}{4}
\item if  $\H_{\pm}g\H_{\pm}\cap \Gamma_{\epsilon}\neq \emptyset$, either $g\in \H_{\pm}$ or $d_{\Gamma}(g,1)>C$,
\item if $\H_{+}g\H_{-}\cap \Gamma_{\epsilon}\neq \emptyset$, either $g\in \H_{+}\H_{-}$ or $d_{\Gamma}(g,1)>C$,
\item if $\H_{-}g\H_{-}\cap \Gamma_{\epsilon}\neq \emptyset$, either $g\in \H_{-}\H_{+}$ or $d_{\Gamma}(g,1)>C$.
\end{enumerate}

Now we prove the statements of the lemma. First note that as $\Gamma_{\epsilon}<\Gamma$ is a characteristic subgroup of finite index, $\phi_{p}$ restricts to an autmomorphism of $\phi_p:\Gamma_{\epsilon}\rightarrow \Gamma_{\epsilon}$, thus $\phi_{p}(\H_{-}\cap \Gamma_{\epsilon})=\H_{+}\cap \Gamma_{\epsilon}$ and $tht^{-1}=\phi_p(h)$ for every $h\in \H_{-}\cap \Gamma_{\epsilon}$. Now we prove (1). Fix an element $\gamma \in \Gamma_{\epsilon}\smallsetminus \H_{+}\H_{-}$. We may write $\gamma=h_{+}f_{\gamma}=h_{+}f_{1}h_{-}$, where $f_{\gamma}\in F_0^{+}, f_{\gamma}=f_1h_{-}$, and $f_1^{-1}\in F_0^{-}$. By (e) and (f) we have $d_{\Gamma}(f_1,1)>C$ and $d_{\Gamma}(f_{\gamma},1)>C$. By the choice of $C>0$, (a) and (b), we deduce that $f_{\gamma}\in \mathcal{N}_{\varepsilon/2}(F_0^{+}\cap \partial_{\infty}\Gamma)$, $f_1^{-1}\in \mathcal{N}_{\varepsilon/2}(F_0^{-}\cap \partial_{\infty}\Gamma)$ and $f_{\gamma}=f_1h_{-}\in \mathcal{N}_{\varepsilon/2}(f_1^{+})$. This shows $f_{1}^{+}\in \mathcal{N}_{\varepsilon}(F_0^{+}\cap \partial_{\infty}\Gamma)$. Thus, by the choice of $\varepsilon>0$ in (c) we have that $f_1T_{\zeta}(F^{-})\subset \mathcal{N}_{\epsilon}(F^{+})$. This proves part (1) of the lemma. The proof of part (2) is analogous.

\par For part (3), if $\gamma \in \Gamma_{\epsilon}\smallsetminus \H_{\pm}$, we may write $\gamma=h_1w_1=h_1wh_2$, for some $h_1,h_2\in \H_{\pm}$, $w_1\in F_0^{\pm}$, $w^{-1}\in F_0^{\pm}$ and $w_1=wh_2$. Since $\gamma \notin \H_{\pm}$, by (c) we have $d_{\Gamma}(w_1,1)>C$ and $d_{\Gamma}(w,1)>C$. In addition, by (a) and (b), $w_1=wh_2\in \mathcal{N}_{\varepsilon/2}(w^{+})$, $w_1\in \mathcal{N}_{\varepsilon/2}(F_0^{\pm}\cap \partial_{\infty}\Gamma)$ and $w^{-}\in \mathcal{N}_{\varepsilon/2}(F_0^{-}\cap \partial_{\infty}\Gamma)$. This shows that $w^{+}\in \mathcal{N}_{\varepsilon/2}(F_0^{\pm})$. Hence, by the choice of $\varepsilon>0$ in (d), we conclude that $wT_{\zeta}(F^{\pm})\subset \mathcal{N}_{\epsilon}(F^{\pm})$.
\end{proof}

\begin{lemma}\label{fi-separable} There is $\delta>0$, a finite-index subgroup $\Gamma_{\epsilon}<\Gamma$ such that for the intersections $\H_{\epsilon^{\pm}}\coloneqq\Gamma_\epsilon \cap \H_{\pm}$ we have $\H_{\epsilon^{+}}\coloneqq\phi_p(\H_{\epsilon^{-}})$, $th_{-}t^{-1}=\phi_p(h_{-})$ for every $h_{-}\in \H_{\epsilon^{-}}$, and the following conditions hold:
\begin{enumerate}
    \item $\gamma B_{\pm}^{\delta}\subset (A_{\pm}^{\delta})^{\circ}$ for every $\gamma \in \Gamma_{\epsilon}\smallsetminus \H_{\pm}$.
    \item  $t\gamma B_{+}^{\delta}\subset B_{+}^{\delta}$ for every $\gamma \in \Gamma_{\epsilon}\smallsetminus \H_{+}$.
    \item $t^{-1}\gamma B_{-}^{\delta}\subset B_{-}^{\delta}$ for every $\gamma \in \Gamma_{\epsilon}\smallsetminus \H_{-}$.
    \item $\gamma B_{+}^{\delta}$ is antipodal to $B_{-}^{\delta}$ for every $\gamma \in \Gamma_{\epsilon}\smallsetminus \H_{+}$.
    \item $\gamma B_{-}^{\delta}$ is antipodal to $B_{+}^{\delta}$ for every $\gamma \in \Gamma_{\epsilon}\smallsetminus \H_{-}$.
\end{enumerate}
\end{lemma}

\begin{proof} Fix now $0<\delta<\kappa$, such that the sets $B_{\pm}^{\delta}, A_{\pm}^{\delta}$ satisfy the conclusion of Lemma \ref{conditions-1}. Since $B_{\pm}^{\delta}$ are compact sets which are antipodal to $F^{\pm}$, choose $\zeta>0$ such that $B_{\pm}^{\delta}\subset T_{\zeta}(F^{\pm})$. Next, we also fix $0<\epsilon<\frac{c\delta}{6}$, where $c>0$ is the constant provided by Lemma \ref{constant-1} depending only on $\H_{\pm}$, such that $\mathcal{N}_{\epsilon}(F^{\pm})$ is antipodal to $B_{\pm}^{\delta}$. By applying Lemma \ref{product-sep} for $\zeta,\epsilon>0$, we may find a deep enough  characteristic finite-index subgroup $\Gamma_{\epsilon}<\Gamma$ such that $th_{-}t^{-1}=\phi(h_{+})$ for every every $h_{-}\in \Gamma_{\epsilon}\cap \H_{-}$ such that conditions (1), (2) and (3) of \Cref{product-sep} hold true.

We now can prove the statements in the claim:

\medskip
\noindent \textup{(1)} Given any $\gamma \in \Gamma_{\epsilon} \smallsetminus \H_{\pm}$, if we write $\gamma$ as in \Cref{product-sep} (3) and since $B_{\pm}^{\delta}$ is $\H_{\pm}$-invariant, we have that $$\gamma B_{\pm}^{\delta}=h_1wh_2B_{\pm}^{\delta}=h_1wB_{\pm}^{\delta}\subset h_1wT_{\zeta}(F^{\pm})\subset \H^{\pm}\mathcal{N}_{\epsilon}(F^{\pm})\subset (A_{\pm}^{\delta})^{\circ}.$$ 

\smallskip
\noindent \textup{(2)} Let $\gamma \in \Gamma_{\epsilon}\smallsetminus \H_{+}$. If $\gamma \in \H_{-}\H_{+}$, write $\gamma=h_{-}h_{+}$, for some $h_{\pm}\in \H_{\pm}$, and hence $$t\gamma B_{+}^{\delta}=(th_{-}t^{-1})tB_{+}^{\delta}\subset (th_{-}t^{-1})(B_{+}^{\delta})^{\circ}\subset B_{+}^{\delta},$$ since $th_{-}t^{-1}\in \H_{+}$.
\par Now suppose that $\gamma \in \Gamma_{\epsilon}\smallsetminus \H_{-}\H_{+}$, we may write $\gamma=h_{-}f_{2}h_{+}$ as in \Cref{product-sep} (2). Since $B_{+}^{\delta}$ is $\H_{+}$-invariant and a subset of $T_{\zeta}(F^{+})$ we have $$t\gamma B_{+}^{\delta}=th_{-}f_{2}B_{+}^{\delta}\subset th_{-}f_{2}T_{\zeta}(F^{+})\subset t\H_{-}\mathcal{N}_{\epsilon}(F^{-})\subset t\H_{-}\mathcal{N}_{c\delta/6}\subset B_{+}^{\delta}.$$ 

\smallskip
\noindent \textup{(3)} The proof is analogous as in (2) by using \Cref{product-sep} (1).

\smallskip
\noindent \textup{(4)} Let $\gamma \in \Gamma_{\epsilon}\smallsetminus \H_{+}$. If $\gamma=h_{-}h_{+}\in \H_{-}\H_{+}$ then clearly $\gamma B_{+}^{\delta}$ is antipodal to $B_{-}^{\delta}$, since $B_{\pm}^{\delta}$ is $\H_{\pm}$-invariant and $B_{+}^{\delta}$ is antipodal to $B_{-}^{\delta}$ by Lemma \ref{conditions-1} (4). If $\gamma \in \Gamma_{\epsilon}\smallsetminus \H_{-}\H_{-}$, we may write $\gamma=h_{-}f_{2}h_{+}$, $h_{\pm}\in \H_{\pm}$, such that \Cref{product-sep} $(2)$ holds. Since $$f_{2}B_{+}^{\delta}\subset f_{2}T_{\zeta}(F^{+})\subset \mathcal{N}_{\epsilon}(F^{-}),$$ $f_{2}B_{+}^{\delta}$ is antipodal to $B_{-}^{\delta}$, by our initial choice of $\epsilon>0$. This shows that $\gamma B_{+}^{\delta}=h_{-}f_2B_{+}^{\delta}$ is antipodal to $B_{-}^{\delta}$.

\smallskip
\noindent \textup{(5)} The proof is analogous to (4) by using Lemma \ref{product-sep} (1).\end{proof}

\subsection{Proof of \Cref{HNN-ext}}
Now we can conclude the proof of the theorem by proving the following.

\begin{theorem} The group $\langle \Gamma_{\epsilon},t\rangle<\mathsf{G}$ is $\Theta$-Anosov and isomorphic to the HNN extension $\Gamma_{\epsilon}\ast_{t\H_{\epsilon^{-}}t^{-1}=\H_{\epsilon^{+}}}=\big\langle \Gamma_{\epsilon},t \ \big| \ th_{-}t^{-1}=\phi_p(h_{-}), h_{-}\in \H_{\epsilon^{-}}\big\rangle$.

\end{theorem}

\begin{proof} Recall by Lemma \ref{fi-separable} that there are a finite-index subgroup $\Gamma_{\epsilon}<\Gamma$, quasiconvex subgroups $\H_{\epsilon^{\pm}}=\H_{\pm}\cap \Gamma_{\epsilon}$, and $t\in \mathsf{G}$ such that $th_{-}t^{-1}=\phi_p(h_{-})$, for every $h_{-}\in \H_{\epsilon^{-}}$. We define the following compact sets: 
\begin{equation}\label{eqn:A}
    A_{\delta,\epsilon}\coloneqq (A_{+}^{\delta}\cap A_{-}^{\delta})\cup  \bigcup_{\gamma \in \Gamma_{\epsilon}\smallsetminus \H_{+}}\gamma B_{+}^{\delta}\cup  \bigcup_{\gamma \in \Gamma_{\epsilon}\smallsetminus \H_{-}}\gamma B_{-}^{\delta}.
\end{equation} By \Cref{thm:HNN1}, it suffices to establish that the pair $(A_{\delta,\epsilon},B_{\pm}^{\delta})$ is an interactive triple for $(\Gamma_{\epsilon}, \Gamma_{\epsilon}\cap \H_{\pm}, t)$. For this, we need to verify successively the following conditions for the sets $(A_{\delta,\epsilon},B_{\pm}^{\delta})$:

\begin{enumerate}[label=(\alph*)]
    \item 
    $\H_{\epsilon^{+}}=\Gamma_{\epsilon}\cap t\Gamma_{\epsilon}t^{-1}$ and $\H_{\epsilon^{-}}=\Gamma_{\epsilon}\cap t^{-1}\Gamma_{\epsilon}t$. To see this, note that since $t(\Lambda_{\Gamma}\smallsetminus \Lambda_{\H_{-}})$ is antipodal to $\Lambda_{\Gamma}$ and $\H_{\epsilon^{+}}<\Gamma_{\epsilon}$ is quasiconvex,  $\H_{\epsilon^{+}}$ has to be a finite-index subgroup of $\Gamma_{\epsilon}\cap t\Gamma_{\epsilon}t^{-1}$. Since $\H_{+}<\Gamma$ is malnormal in $\Gamma$, $\H_{\epsilon^{+}}=\Gamma_{\epsilon}\cap \H_{+}$ is malnormal in $\Gamma_{\epsilon}$, hence $\H_{\epsilon^{+}}=\Gamma_{\epsilon}\cap t\Gamma_{\epsilon}t^{-1}$. We similarly verify that $\H_{\epsilon^{-}}=\Gamma_{\epsilon}\cap t^{-1}\Gamma_{\epsilon}t$.

    \item
    {\em $(A_{\delta,\epsilon})^{\circ}$ is antipodal to $(B_{\pm}^{\delta})^{\circ}$.} Combining \Cref{conditions-1} (3) and Lemma \ref{fi-separable} (1,5), we see that  the sets $\gamma B_{+}^{\delta}$, $\gamma'B_{-}^{\delta}$ are antipodal to $B_{+}^{\delta}$ for all $\gamma \in \Gamma_{\epsilon}\smallsetminus \H_{+}$, $\gamma'\in \Gamma_{\epsilon}\smallsetminus \H_{-}$
    Moreover, by \Cref{conditions-1} (3), $(A_+^\delta\cap A_-^\delta)\smallsetminus\Lambda_{\rm H_+}$ is antipodal to $B_+^\delta$.
    Thus, it follows from the definition of $A_{\delta,\epsilon}$ given by \eqref{eqn:A} that $A_{\delta,\epsilon}\smallsetminus \Lambda_{\H_{+}}$ is antipodal to $B_{+}^{\delta}$.

    On the other hand, $B_{+}^{\delta}$ is antipodal to $\Lambda_{\Gamma} \smallsetminus \Lambda_{\H_{+}}$ (by Lemma \ref{conditions-1} (1)). Since $\Lambda_{\H_{+}}$ lies in the closure of the latter set, it follows that any interior point of $B_+^\delta$ is antipodal to $\Lambda_{\rm H_+}$. Thus, it follows that $(B_{+}^{\delta})^{\circ}$  is antipodal to $A_{\delta,\epsilon}$.

    \item 
    {\em $B_{+}^{\delta}$ is antipodal to $B_{-}^{\delta}$.} This is the content of \Cref{conditions-1} (4).

    \item {\em $\Lambda_{\G_\epsilon}\smallsetminus \Lambda_{\H_\pm}$ is antipodal to $B_\pm$.} This is immediate from \Cref{conditions-1} (1,3).
    
    \item
    $t^{\pm 1}B_{\pm}^{\delta}\subset (B_{\pm}^{\delta})^{\circ}$. This is the content of Lemma \ref{conditions-1} (2).

    \item
    $t^{\pm 1}A_{\delta,\epsilon}\subset B_{\pm}^{\delta}$. To see this, note that by Lemma \ref{fi-separable} (1,3), we have $t\gamma' B_{-}^{\delta}\subset tA_{-}^{\delta}$ 
    for all $\gamma' \in \Gamma_{\epsilon}\smallsetminus \H_{-}$, and $t\gamma B_{+}^{\delta}\subset B_{+}^{\delta}$ for all $\gamma \in \Gamma_{\epsilon}\smallsetminus \H_{+}$. Hence 
    \[
    tA_{\delta,\epsilon}\;\subset\;
    tA_{-}^{\delta}\cup \bigcup_{\gamma \in \Gamma_{\epsilon}\smallsetminus \H_{+}}t\gamma B_{+}^{\delta}\;\subset\;
    tA_{-}^{\delta}\cup B_{+}^{\delta}\;\subset\;
    B_{+}^{\delta};
    \]
    the rightmost inclusion in the above follows form \eqref{eqn:AB}.
    Similarly, we see that $t^{-1}A_{\delta,\epsilon}\subset B_{-}^{\delta}$.

    \item 
    {\em $\H_{\epsilon^{\pm}}$ leaves $B_{\pm}^{\delta}$ precisely invariant}. Note that $B_{\pm}^{\delta}$ is $\H_{\pm}$-invariant. Also, by \eqref{eqn:A} and \Cref{fi-separable} (1), $\gamma B_{\pm}^{\delta}\subset A_{\delta, \epsilon}\cap (A_{\delta}^{\pm})^{\circ}\subset (A_{\delta,\epsilon})^{\circ}$ for every $\gamma \in \Gamma_{\epsilon}\smallsetminus \H_{\pm}$.
\end{enumerate}

Thus, the conclusion now follows by \Cref{thm:HNN1}.\end{proof}

\subsection{Proof of Theorem \ref{main:cyclic-HNN}}\label{sec:cyclic-HNN}
Using Theorem \ref{HNN-ext}, we can prove the following more general theorem from which Theorem \ref{main:cyclic-HNN} will follow. We call a matrix $g\in \mathsf{GL}_n(\mathbb{C})$ semisimple if it is similar to a diagonal matrix.

\begin{theorem}\label{thm:cyclic-HNN2} 
Let $\Gamma<\mathsf{SL}_d(\mathbb{R})$ be a $1$-Anosov subgroup. For any infinite order semisimple element $a\in \Gamma$, there are $\gamma\in \Gamma$, $m\in \mathbb{N}$, and a finite-index normal subgroup $\Gamma'<\Gamma$ such that $\gamma a^m\gamma^{-1},a^m\in \Gamma'$ and the HNN extension $$\Gamma'\ast_{ta^mt^{-1}=\gamma a^m \gamma^{-1}}=\big \langle \Gamma', t \ | \ ta^mt^{-1}=\gamma a^m\gamma^{-1}\big\rangle$$ is hyperbolic and admits a faithful Anosov representation into $\mathsf{SL}_{d}(\mathbb{C})$.
\end{theorem}

\begin{proof}
Hamilton \cite{Hamilton-double-coset} showed that if $\Gamma$ is a uniform lattice in $\mathsf{G}=\mathsf{O}(n,1)$, $\mathsf{U}(n,1)$, $\mathsf{Sp}(n,1)$, $n\geq 2$, or $\mathsf{F}_4^{-20}$, then for every two cyclic subgroups $\gamma_1,\gamma_2\in \Gamma$ the product $\langle \gamma_1\rangle \langle \gamma_2\rangle$ is a separable subset of $\Gamma$. An important element in her proof is that loxodromic elements in $\mathsf{G}$ are semisimple. We observe that this phenomenon extends to other cases\footnote{For hyperbolic groups in which all quasiconvex subgroups are separable, this is known by \cite[Thm. 1]{Minasyan}. }:

\begin{lemma} \label{prod-sep-gen}
Let $\Delta<\mathsf{GL}_d(\mathbb{C})$ be a torsion-free finitely generated subgroup, let $a\in \Delta$ be a semisimple element such that $\langle a\rangle$ is a maximal abelian subgroup of $\Delta$ and $b\in \Delta$ be any element such that $\langle a\rangle \cap \langle bab^{-1}\rangle=\{1\}$. Then the product $\langle a\rangle \langle bab^{-1}\rangle$ is separable in $\Gamma$.
\end{lemma}

\begin{proof} 
The proof of the proposition follows by the arguments in \cite{Hamilton-double-coset} as follows. Let $R\subset \mathbb{C}$ be the ring generated by the entries of a finite generating subset of $\Delta$ such that $\Delta<\mathsf{GL}_d(R)$. First, by using the specialization theorem of Grunewald--Segal \cite[Thm. A]{Grunewald-Segal}, one checks immediately that the conclusion of Propositions 2.1, 2.2 and 2.3 from \cite{Hamilton-double-coset} hold true for $R$. By \cite{Tsouvalas-cyclic-separability}, we know that for all $r\in\N$, the cyclic groups $\langle a^r\rangle, \langle ba^rb^{-1}\rangle<\Gamma$ are separable. Following the proof of \cite[Thm. 3.2]{Hamilton-double-coset} and using the fact that $\langle a\rangle, \langle bab^{-1}\rangle$ are separable in $\Delta$, one can find a normal finite-index subgroup $\Delta_0<\Delta$ such that $\Delta_0\cap \langle a\rangle=\langle a^n\rangle$, $\Delta_0\cap \langle bab^{-1}\rangle=\langle ba^nb^{-1}\rangle$ and the ratio of any two eigenvalues of $a^n\in \Delta_0$ is either equal to $1$ or has an infinite order (i.e., not a nontrivial root of $1$). Then the argument of proof of \cite[Thm. 3.2, p. 2331-2334]{Hamilton-double-coset} goes through and shows that $\langle a^n\rangle \langle ba^nb^{-1}\rangle$ is separable in $\Delta_0$.
Hence $\langle a\rangle \langle bab^{-1}\rangle$ is separable in $\Delta$ (see Proposition 3.1 of \cite{Hamilton-double-coset}).
\end{proof} 

We now return to the proof of \Cref{thm:cyclic-HNN2}. Let $\Gamma<\mathsf{SL}_d(\mathbb{R})$ be a $1$-Anosov subgroup and $a\in \Gamma$ a semisimple matrix. Since $\langle a\rangle<\Gamma$ is separable and $\langle a\rangle$ has finite index in the group $\big\{h\in \Gamma: ha^{\pm}\in \{a^{+},a^{-}\}\big\}$, up to passing to a finite index subgroup of $\Gamma$, we may assume that $\langle a\rangle$ is malnormal, maximal abelian separable subgroup of $\Gamma$. Fix $b\in \Gamma$ an infinite order element whose fixed points in $\partial_{\infty}\Gamma$ are different from that of $a$. Since $\langle a\rangle<\Gamma$ is maximal abelian, by \Cref{prod-sep-gen}, for every normal finite-index subgroup $\Gamma_0<\Gamma$ and any $r\in \mathbb{N}$, $(\langle a\rangle \cap \Gamma_0)b^r(\langle a\rangle \cap \Gamma_0)b^{-r}\subset \Gamma_0$ is separable.
 
Up to conjugating the representation $\rho$ with an element in $\mathsf{GL}_d(\mathbb{R})$, we may assume that the attracting and repelling fixed points of $\rho(a)\in \mathsf{GL}_d(\mathbb{R})$ in $\mathcal{F}_{1,d-1}(\mathbb{R})$ are respectively the flags $(\textup{span}\{e_{d-1}\},\textup{span}\{e_1,\ldots,e_{d-1}\})$  and $(\textup{span}\{e_{d}\},\textup{span}\{e_1,\ldots,e_{d-2},e_d\})$. 

Let us consider the separable malnormal subgroups $\M_{+}=\M_{-}=\langle a\rangle$ of $\Gamma$ and the element 
$$\tau\coloneqq\begin{pmatrix}[0.8] \textup{I}_{d-2} & &\\
& 1 & \\ & & i \end{pmatrix}.$$ 
As the proof of Theorem \ref{cyclic-amalgam}, one can check that $\tau(\Lambda_{\Gamma}\smallsetminus \Lambda_{\M_{-}})$ is antipodal to $\Lambda_{\Gamma}\smallsetminus \Lambda_{\M_{+}}$. Thus, since $b\in \Gamma$ has distinct fixed points from $a\in \Gamma$, by the discussion in the previous paragraph, we check that the assumptions of Theorem \ref{HNN-ext} are satisfied for $(\Gamma, \M_{+},\M_{-})$. Thus, there are $m,p\in \mathbb{N}$ large enough and $\Gamma'<\Gamma$ a characteristic finite-index subgroup of $\Gamma$ with $a^m\in \Gamma'$ such that the group $\langle \rho(\Gamma'),\tau \rho(b)^p\rangle<\mathsf{SL}_{d}(\mathbb{C})$ is $1$-Anosov and isomorphic to the HNN extension $\big \langle \Gamma',t \ \big| \ t(b^{-p}a^mb^{p})t^{-1}=a^m\rangle \cong \big \langle \Gamma',t \ \big| \ ta^mt^{-1}=b^{-p}a^mb^{p}\big\rangle $.\end{proof}

\begin{remark}  The hyperbolicity of the HNN extension $\Gamma'\ast_{ta^mt^{-1}=b^{-p}a^mb^p}$ in Theorem~\ref{thm:cyclic-HNN2} follows from the fact that it admits an Anosov representation. We observe that although the cyclic subgroups $\langle b^{-p}a^mb^p \rangle$ and $\langle a^m\rangle$ are conjugate in $\Gamma$, they are maximal and non-conjugate in $\Gamma'$. Therefore, the hyperbolicity of $\Gamma'\ast_{ta^mt^{-1}=b^{-p}a^mb^p}$ is in fact a consequence of the Bestvina–Feighn Combination Theorems~\cite{Bestvina-Feighn}.
\end{remark}

\end{document}